%% file: MMESFD.tex
\documentclass[11pt]{elsarticle}
\usepackage{amsmath,amssymb,amsthm}
\usepackage{mathrsfs}
\usepackage[bookmarks=true,
            bookmarksnumbered=true,
            bookmarksopen=true,
            colorlinks,
            pdfborder=001,
            linkcolor=black]{hyperref}
\usepackage{bm}
\usepackage{array}
\usepackage{float}
\usepackage{color}
\usepackage{lineno}
\usepackage{subcaption}
\usepackage{stmaryrd}
\usepackage{extarrows}
\usepackage{subcaption}
\usepackage{multirow}
\newtheorem{thm}{Theorem}[section]
\newtheorem{lem}[thm]{Lemma}
\newtheorem{rmk}{Remark}[section]
\newtheorem{example}{Example}[section]

\newtheorem{definition}{Definition}[section]
\newproof{pf}{Proof}
\numberwithin{equation}{section}
\numberwithin{figure}{section}
\numberwithin{table}{section}
\allowdisplaybreaks

\newcommand\dd{\mathrm{d}}

\newcommand\abs[1]{\lvert #1 \rvert}

\newcommand\pd[2]{\dfrac{\partial {#1}}{\partial {#2}}}

\newcommand\bx{\bm{x}}
\newcommand\bv{\bm{v}}
\newcommand\bw{\bm{w}}
\newcommand\bF{\bm{F}}

\newcommand\bV{\bm{V}}
\newcommand\bU{\bm{U}}

\newcommand\bT{\bm{T}}
\newcommand\vx{v_1}
\newcommand\vy{v_2}
\newcommand\vz{v_3}
\newcommand\tbF{\widetilde{\bF}}

\newcommand\avx[1]{\overline{#1}}
\newcommand\avy[1]{\avx{\avx{#1}}}
\newcommand\avz[1]{\avx{\avy{#1}}}

\newcommand\jump[1]{\llbracket #1 \rrbracket}
\newcommand\mean[1]{\{\!\!\{ #1 \}\!\!\}}
\newcommand\meanln[1]{\{\!\!\{ #1 \}\!\!\}^{\text{ln}}}

\newcommand\jumpangle[1]{\langle\!\langle #1 \rangle\!\rangle}

\begin{document}

\begin{frontmatter}

  \title{Entropy stable adaptive moving mesh schemes for 2D and 3D special relativistic hydrodynamics}

  \author{Junming Duan}
  \ead{duanjm@pku.edu.cn}
  \author{Huazhong Tang\corref{cor1}}
  \ead{hztang@math.pku.edu.cn}
  \address{Center for Applied Physics and Technology, HEDPS and LMAM,
    School of Mathematical Sciences, Peking University, Beijing 100871, P.R. China}
  \cortext[cor1]{Corresponding author. Fax:~+86-10-62751801.}

  \begin{abstract}
This paper develops  entropy stable (ES) adaptive moving mesh schemes for the 2D and 3D special relativistic hydrodynamic (RHD) equations.
 They are built on the ES finite {volume} approximation of the RHD equations in curvilinear coordinates, the discrete geometric conservation laws,
and the mesh adaptation implemented by iteratively solving the Euler-Lagrange equations of the mesh   adaption   functional
in the computational domain  with suitably chosen monitor functions.
   First,  a sufficient condition is proved for the two-point entropy conservative (EC) flux,
   by mimicking the derivation of the continuous entropy identity in curvilinear coordinates and using the discrete geometric conservation laws given by the conservative metrics method.
 Based on such sufficient condition, the EC fluxes for the RHD equations in curvilinear coordinates are derived and  the second-order accurate semi-discrete EC schemes are developed to
satisfy the entropy identity for the given convex entropy pair.
Next, the semi-discrete ES schemes satisfying the entropy inequality
 are proposed by adding a suitable dissipation term to the EC scheme and utilizing linear reconstruction with the minmod limiter
in the scaled entropy variables in order to suppress the numerical oscillations of the above EC scheme.
Then, the semi-discrete ES schemes are integrated in time by using the second-order strong stability preserving explicit Runge-Kutta schemes.
Finally,
several  numerical results show that our 2D and 3D ES adaptive moving mesh schemes effectively capture the localized structures,
such as sharp transitions or discontinuities, and are more efficient than {their counterparts} on uniform mesh.
  \end{abstract}

  \begin{keyword}
    Entropy conservative flux\sep entropy stable scheme\sep moving mesh scheme
    \sep mesh adaptation\sep special relativistic hydrodynamics
  \end{keyword}

\end{frontmatter}
\input{sect1-Intro}

\input{EntAnal}
\input{3D}

\input{MovingMesh}
\input{NumTests}

\input{Conc}

\section*{Acknowledgments}
The authors were partially supported by the Special Project on High-performance Computing under the
National Key R\&D Program (No. 2016YFB0200603), Science Challenge Project (No. TZ2016002), the Sino-German Cooperation Group Project (No. GZ 1465),
the National Natural Science Foundation of China (No. 11421101),
and High-performance Computing Platform of Peking University.


\input MMESFD-ref.tex

\end{document}

%% file: sect1-Intro.tex
\section{Introduction}
This paper is concerned with the entropy stable (ES) adaptive moving mesh schemes
for the special relativistic hydrodynamic (RHD) equations.
In the laboratory frame,  the 2D and 3D special RHD equations can be cast in the  divergence form
\begin{equation}\label{eq:RHD1}
\pd{\bU}{t}+\sum_{k=1}^{d}\pd{\bF_k(\bU)}{x_k}=0,~d=2,3,
\end{equation}
where $\bU$ and $\bF_k$ are respectively
the conservative vector and the flux vector in the $x_k$-direction
and defined by
\begin{align}\label{eq:RHD2}
\bU=\begin{pmatrix}
D \\ \bm{m} \\ E \\
\end{pmatrix},\quad
\bF_k=\begin{pmatrix}
Dv_k \\ \bm{m}v_k+p\bm{e}_k\\
m_k
\end{pmatrix}, \ k=1,\cdots,d,
\end{align}
with the mass density $D=\rho W$, the momentum density
$\bm{m}=(m_1,\cdots,m_d)^\mathrm{T}=DhW\bv$,
the energy density $E=DhW-p$, the pressure $p$,  the fluid velocity
$\bv=(v_1,\cdots,v_d)^\mathrm{T}$, and
the rest-mass density $\rho$.
Here $\bm{e}_k$ is the $k$th column of the unit $d\times d$ matrix,
$k=1,\cdots,d$,
$W=1/\sqrt{1-\abs{\bv}^2}$ is the Lorentz factor and $h=1+e+p/\rho$ is
the specific enthalpy with  the specific internal energy $e$ and units in which the speed
of light is equal to one.
The governing equations \eqref{eq:RHD1}-\eqref{eq:RHD2} need to be closed
by the equation of state (EOS).
This paper will only consider the perfect gas with the simple EOS  given by
\begin{equation}\label{EQ:EOS01}
p=(\Gamma-1)\rho e,
\end{equation}
with the adiabatic index $\Gamma\in(1,2]$.
Since there is no explicit {expression} for the primitive variables
$(\rho,\bv^\mathrm{T},p)$ and the flux $\bF_k$ in terms of $\bU$,
a nonlinear algebraic equation such as
\begin{equation*}
E+p=DW+\dfrac{\Gamma}{\Gamma-1}pW^2,
\end{equation*}
needs to be (numerically) solved in order to recover the value of the pressure $p$ from the given $\bU$  and
then the rest-mass density $\rho$, the specific enthalpy
$h$, and the velocity $\bv$ by using
\begin{equation*}
\rho=\dfrac{D}{W},\quad h=1+\dfrac{\Gamma p}{(\Gamma-1)\rho},\quad
\bv=\dfrac{\bm{m}}{Dh}.
\end{equation*}
The relativistic description for the fluid dynamics
 at nearly the speed of light should be considered
 in investigating  the astrophysical
phenomena  from stellar to galactic scales, e.g. coalescing neutron stars, core
collapse supernovae,  active galactic nuclei, superluminal jets, the formation of black holes,
and gamma-ray bursts etc.
The  system \eqref{eq:RHD1}-\eqref{eq:RHD2} becomes much more complicated than the Euler equations in gas dynamics due to the relativistic effect, so its analytic treatment is very challenging.
Numerical simulation is a powerful way to help us  better understand   the physical mechanisms in the RHD.
The pioneering numerical work may date back to the finite difference methods with the artificial viscosity technique
in the Lagrangian coordinates \cite{May1966Hydrodynamic,May1967Stellar}
and the Eulerian coordinates \cite{Wilson1972Numerical}.
Since the early 1990s, the modern shock-capturing methods were extended to
the special or general RHD or relativastic magnetohydronamics {(RMHD)}.
They include, but are  not limited to,  
the Roe solver \cite{Eulderink1994General},
the Harten-Lax-van Leer methods \cite{DelZanna2002An,Schneider1993New},
the Harten-Lax-van Leer Contact methods \cite{Ling2019Physical,Mignone2005An},
the essentially non-oscillatory (ENO) and the weighted ENO (WENO) methods \cite{DelZanna2002An,Dolezal1995Relativistic,Tchekhovskoy2007WHAM},
the piecewise parabolic methods \cite{Marti1996Extension,Mignone2005The},
the adaptive mesh refinement method \cite{Zhang2006RAM},
the Runge-Kutta discontinuous Galerkin (DG) methods with WENO limiter \cite{Zhao2013Discontinuous},
the direct Eulerian generalized Riemann problem schemes
\cite{Wu2016A,Wu2014A,Yang2011A1D,Yang2012A2D}, the adaptive moving mesh method \cite{He2012RHD},
the two-stage fourth-order accurate time discretizations \cite{Yuan2020Two} and so on.
The readers are  also referred to the early review articles \cite{Font2008Numerical,Marti2003Numerical,Marti2015Grid}
for more references.
Recently, the properties of the admissible state set
and the physical-constraints-preserving (PCP) numerical schemes were well studied for the RHD
and the special RMHD \cite{Ling2019Physical,Wu2015Finite,Wu2016Physical,Wu2017Admissible,Wu2018Admissible}.

 For the RHD equations \eqref{eq:RHD1}-\eqref{eq:RHD2},
it is interesting to design a numerical scheme being consistent with the Clausius inequality, i.e., the entropy inequality.
For a general quasi-linear hyperbolic conservation laws, the entropy condition {is needed} to single out the unique physical relevant solution among all the weak solutions. 
However, in practice, it is very hard to show that the high-order schemes of
the scalar conservation laws and the schemes for the hyperbolic system
satisfy the entropy inequality for any convex entropy function.
In view of this, many researchers are trying to study the high-order accurate
entropy conservative (EC) or {ES} schemes,
which satisfy the entropy identity or inequality for a given entropy pair.
The second-order EC schemes 
were studied in \cite{Tadmor1987The,Tadmor2003Entropy}, and their higher-order
extension was considered in \cite{Lefloch2002Fully}.
Unfortunately, the EC schemes may become oscillatory near the
discontinuities. To suppress possible numerical oscillation,
some additional dissipation term  has to be added to obtain the ES schemes.
Combining the EC flux
with the ``sign'' property of the ENO reconstruction, the
arbitrary high-order ES schemes were constructed by using high-order
dissipation terms \cite{Fjordholm2012Arbitrarily}.
The ES schemes based on summation-by-parts (SBP) operators were
developed for the Navier-Stokes equations \cite{Fisher2013High}.
Several ES DG schemes were also
studied, such as the semi-discrete DG for scalar conservation laws \cite{Jiang1994On},
the space-time DG formulation \cite{Barth1999Numerical,Hiltebrand2014Entropy}
and the DG schemes using suitable quadrature rules
for the conservation laws on hexahedron meshes \cite{Carpenter2014Entropy,Gassner2013A} and unstructured simplex meshes \cite{Chen2020Review}.
As a base of those works, constructing the affordable two-point EC flux is key. Recently, the EC or ES schemes were also
extended to the shallow water equations \cite{Fjordholm2011Well},
the shallow water magnetohydrodynamics \cite{Duan2020SWMHD,Winters2016An},
the RHD equations {\cite{Bhoriya2020Entropy,Duan2020RHD}},
the magnetohydrodynamics \cite{Chandrashekar2016Entropy,Winters2016Affordable},
the RMHD equations \cite{Duan2020RMHD,Wu2020Entropy}, and so on.

In view of the fact that the solutions of the RHD equations often exhibit localized structures,
e.g. containing sharp transitions or discontinuities in
relatively localized regions, the adaptive mesh strategy
can improve the efficiency and quality of numerical simulation.
Up to now, adaptive moving mesh methods have been successfully applied to  many problems in science and engineering,
see e.g. \cite{Brackbill1993An,Brackbill1982Adaptive,He2012RHD,He2012RMHD,Huang2001Variational,Huang2011Adaptive,Li2001Moving,Ren2000An,
  Tang2003Adaptive,Tang2003An,Wang2004A,Winslow1967Numerical,Yang2012A,Zhang2020An}.
The readers are also  referred to the review papers \cite{Budd2009Adaptivity,Tang2005Moving} and references therein.
This paper aims at developing the ES adaptive moving mesh  schemes for the
2D and 3D RHD equations \eqref{eq:RHD1}-\eqref{eq:RHD2}.
Our schemes will be built on the ES finite {volume} approximation of the RHD
equations in curvilinear coordinates,  the discrete geometric conservation laws, and the mesh adaptation implemented by iteratively solving the Euler-Lagrange equations of the mesh   adaption  functional in the computational domain with
suitably chosen monitor functions.
To do that, we first prove  a sufficient condition for the two-point EC fluxes and then derive the EC fluxes in curvilinear coordinates by utilizing the procedure in \cite{Duan2020RMHD}.
The key point is that the geometric conservation laws (GCLs) introduced by the coordinate transformation should be satisfied by the  discretization of the metrics.
The conservative metric method  \cite{Thomas1979} is adopted to  guarantee the GCLs and the suitable dissipation term utilizing linear reconstruction with the minmod limiter
in the scaled entropy variables is added to the EC flux to get the second-order accurate ES schemes.
The final fully discrete schemes are developed by integrated the semi-discrete ES schemes with the second-order accurate explicit strong-stability preserving (SSP) Runge-Kutta (RK) schemes.
Two approximations of the volume conservation law  are presented and compared.

The paper is organized as follows.
Section \ref{section:Symm} introduces the entropy conditions for the RHD equations in Cartesian and curvilinear coordinates.
Section \ref{section:3D} presents the  EC and ES schemes, including the discretization of the metrics,
and construction of the two-point EC flux in curvilinear coordinates.
Section \ref{section:MM} gives the adaptive moving mesh strategy.
Several 2D and 3D numerical experiments are conducted in Section \ref{section:Num} to validate the efficiency and the ability of our schemes in capturing the sharp transitions or discontinuities.
Section \ref{section:Conclusion}  concludes the work
with final remarks.

%% file: EntAnal.tex
\section{Entropy conditions for the RHD}\label{section:Symm}
For the RHD equations \eqref{eq:RHD1}-\eqref{eq:RHD2} with the EOS \eqref{EQ:EOS01}, there exists an entropy pair $(\eta,q_k)$,
\begin{equation*}
\eta(\bU)=-\dfrac{\rho Ws}{\Gamma-1}, \quad q_k(\bU)=\eta v_k,
\end{equation*}
where $s=\ln(p/\rho^\Gamma)$ is the thermodynamic entropy,  $\eta$ is a convex function of $\bU$ and $(\eta,q_k)$ satisfies
\begin{equation*}
q_k'(\bU)=\eta'(\bU)\bF_k'(\bU),
\quad k=1,\cdots,d.
\end{equation*}
Here $\eta$ and $q_k$ are called the {\em entropy function} and {\em entropy flux}, respectively.
From those, we can also define the {\em entropy variables}
$\bV$  by
\begin{equation*}
  \bV:=\eta'(\bU)^\mathrm{T}=\left(\dfrac{\Gamma-s}{\Gamma-1}+\frac{\rho}{p},
  \dfrac{\rho W\bv^\mathrm{T}}{p}, -\dfrac{\rho W}{p}\right)^\mathrm{T},
\end{equation*}
and the {\em entropy potential} $\phi$  and {\em entropy potential flux} $\psi_k$ by using the conjugate variables as follows
\begin{equation}\label{EQ:potential}
  \phi:=\bV^\mathrm{T}\bU-\eta=\rho W,\quad \psi_k:=\bV^\mathrm{T}\bF_k - q_k=\rho Wv_k,
\end{equation}
respectively.

For the smooth solutions of \eqref{eq:RHD1}-\eqref{eq:RHD2} with the entropy pair $(\eta,q_k)$,
{multiplying \eqref{eq:RHD1} by $\bV^\mathrm{T}$ {left} gives the entropy identity
\begin{equation*}
\bV^\mathrm{T}\left(\pd{\bU}{t}+\sum_{k=1}^d\pd{\bF_k(\bU)}{x_k}\right)
=\pd{\eta(\bU)}{t}+\sum\limits_{k=1}^d\pd{q_k(\bU)}{x_k} = 0.
\end{equation*}
For the discontinuous solutions,  it is replaced with the entropy inequality
\begin{equation*}
\pd{\eta(\bU)}{t}+\sum\limits_{k=1}^d\pd{q_k(\bU)}{x_k} \leqslant 0,
\end{equation*}
which holds in the sense of distributions.

Next, let us derive the RHD equations in curvilinear coordinates
and corresponding entropy condition.
Let $\Omega_p$ be the domain where the physical problem \eqref{eq:RHD1}-\eqref{eq:RHD2} is defined,
and $\Omega_c$ be the computational domain with coordinates $\bm{\xi}=(\xi_1,\cdots,\xi_d)$ that
is artificially chosen for the sake of mesh redistribution or movement.
Our adaptive moving meshes for $\Omega_p$ can be generated as the images of a reference mesh in $\Omega_c$ by a time dependent,
differentiable, one-to-one coordinate  mapping {$\bx = \bx(\bm{\xi},t)$},
which can be expanded as
\begin{align}\label{eq:transf}
    t=\tau,\ \ \bx=\bx(\bm{\xi},\tau),\ \
  \bm{\xi}=(\xi_1,\cdots,\xi_d)\in\Omega_c.
\end{align}
Under this transformation, the detailed transformation of the system \eqref{eq:RHD1}-\eqref{eq:RHD2} in the coordinates {$(\bm{\xi},\tau)$} reads
\begin{align}\label{eq:RHD_MM}
\pd{\left(J\bU\right)}{\tau}+\sum_{k=1}^d\dfrac{\partial}{\partial\xi_k}{\left(J\pd{\xi_k}{t}\bU\right)}
+\sum_{k,l=1}^d\dfrac{\partial}{\partial\xi_k}{\left(J\pd{\xi_k}{x_l}\bF_l\right)}=0,
\end{align}
where  $J$ denotes the determinant of the Jacobian matrix and
its 3D version   is explicitly given by
\begin{equation*}
J=\det\left(\pd{(t,\bx)}{(\tau,\bm{\xi})}\right)=
\begin{vmatrix}
1 & 0 & 0 & 0\\
\pd{x_1}{\tau} & \pd{x_1}{\xi_1} & \pd{x_1}{\xi_2} & \pd{x_1}{\xi_3} \\
\pd{x_2}{\tau} & \pd{x_2}{\xi_1} & \pd{x_2}{\xi_2} & \pd{x_2}{\xi_3} \\
\pd{x_3}{\tau} & \pd{x_3}{\xi_1} & \pd{x_3}{\xi_2} & \pd{x_3}{\xi_3}
\end{vmatrix}.
\end{equation*}
The metric coefficients should satisfy the following geometric conservation laws (GCLs) consisting of the
the volume conservation law (VCL) and surface conservation laws (SCLs)
\begin{equation}\label{eq:GCL}
\begin{aligned}
&\text{VCL:}\quad \pd{J}{\tau}+\sum_{k=1}^d\dfrac{\partial}{\partial\xi_k}{\left(J\pd{\xi_k}{t}\right)}=0,\\
&\text{SCLs:}\quad \sum_{k=1}^d\dfrac{\partial}{\partial\xi_k}{\left(J\pd{\xi_k}{x_l}\right)}=0,~ l=1,\cdots,d.
\end{aligned}
\end{equation}
The former indicates that volumetric increment of a moving cell must be equal to the sum of the changes along the surfaces that enclose the volume,
while the latter indicates that cell volumes must be closed by its surfaces \cite{Zhang1993Discrete}.
Those GCLs mean that free-stream solution is preserved by \eqref{eq:RHD_MM},
that is to say, if a physical constant state is given as the initial condition, it will remain unchanged.
If the free-stream solution cannot be preserved by the numerical schemes on the moving mesh, it may cause some large errors.

Finally, let us derive the entropy identity for the RHD equations \eqref{eq:RHD_MM}. 
The three parts of the left-hand side of the product of  $\bV^\mathrm{T}$
and \eqref{eq:RHD_MM}
 can be respectively rewritten as follows
\begin{align*}
&\bV^\mathrm{T}\pd{\left(J\bU\right)}{\tau}=\pd{(J\eta)}{\tau}+\left(\bV^\mathrm{T}\bU-\eta\right)\pd{J}{\tau}, \\
&\bV^\mathrm{T}\sum_{k=1}^d\dfrac{\partial}{\partial\xi_k}{\left(J\pd{\xi_k}{t}\bU\right)}
=\sum_{k=1}^d\dfrac{\partial}{\partial\xi_k}\left(J\pd{\xi_k}{t}\eta\right)
+\left(\bV^\mathrm{T}\bU-\eta\right)\sum_{k=1}^d\dfrac{\partial}{\partial\xi_k}\left(J\pd{\xi_k}{t}\right), \\
&\bV^\mathrm{T}\sum_{k,l=1}^d\dfrac{\partial}{\partial\xi_k}{\left(J\pd{\xi_k}{x_l}\bF_l\right)}
=\sum_{k,l=1}^d\dfrac{\partial}{\partial\xi_k}\left(J\pd{\xi_k}{x_l}q_l\right)
+\sum_{l=1}^d\left(\bV^\mathrm{T}\bF_l-q_l\right)\sum_{k=1}^d
\dfrac{\partial}{\partial\xi_k}\left(J\pd{\xi_k}{x_l}\right).
\end{align*}
Using the GCLs \eqref{eq:GCL} gives
\begin{align}\label{eq:EntropyId_MM}
\pd{\left(J\eta\right)}{\tau}+\sum_{k=1}^d\dfrac{\partial}
{\partial\xi_k}{\left(J\pd{\xi_k}{t}\eta\right)}
+\sum_{k,l=1}^d\dfrac{\partial}{\partial\xi_k}
{\left(J\pd{\xi_k}{x_l}q_l\right)}=0,
\end{align}
which is the entropy identity in the coordinates $(\tau,\bm{\xi})$.
Similarly, it will be replaced with corresponding entropy inequality when the solutions $\bU$ are not smooth.

%% file: 3D.tex
\section{Numerical schemes}\label{section:3D}
This section focuses on constructing the 3D moving mesh EC and ES schemes for the RHD equations \eqref{eq:RHD_MM} in curvilinear coordinates on the structured mesh.
The 2D schemes can be obtained by setting $x_3=\xi_3$ and removing all the dependence of $\bU$ on $\xi_3$ and $x_3$,  $\bF_3$ and the $x_3$-component of $\bU$ and $\bF_k$, $k=1,2$. In view of $t=\tau$, the symbol $\tau$ will be replaced with $t$ hereafter.

\subsection{EC scheme}
Assume that the computational domain $\Omega_c$
is rectangular, e.g. $[0,1]\times[0,1]\times[0,1]$, and divided into a fixed orthogonal mesh
 $\{(\xi_{1,i_1+\frac12},\xi_{2,i_2+\frac12},
\xi_{3,i_3+\frac12})$:
$0=\xi_{k,\frac12}<\xi_{k,1+\frac12}<\cdots<\xi_{k,i_k+\frac12}
<\cdots<\xi_{k,N_k+\frac12}=1$, $k=1,2,3\}$
%
%
with the constant  step-size $\Delta \xi_k=\xi_{k,i_k+\frac12}-\xi_{k,i_k-\frac12}$.
%
%
For the sake of brevity, the index $\bm{i}=(i_1,i_2,i_3)$
is used to denote the cell
$[\xi_{1,i_1-\frac12},\xi_{1,i_1+\frac12}]\times[\xi_{2,i_2-\frac12},\xi_{2,i_2+\frac12}]\times[\xi_{3,i_3-\frac12},\xi_{3,i_3+\frac12}]$
and $\hat{\bm{i}}_{1,\pm}$, $\cdots$, $\hat{\bm{i}}_{3,\pm}$
denote the middle points of the cell interfaces, i.e.
$(\xi_{1,i_1\pm\frac12},\xi_{2,i_2},\xi_{3,i_3})$,
$(\xi_{1,i_1},\xi_{2,i_2\pm \frac12},\xi_{3,i_3})$,
$(\xi_{1,i_1},\xi_{2,i_2},\xi_{3,i_3\pm \frac12})$, respectively,
where $\xi_{k,i_k}=(\xi_{k,i_k+\frac12}+\xi_{k,i_k-\frac12})/2$,
$k=1,2,3$.

%

For the cell $\bm{i}$, the RHD system \eqref{eq:RHD_MM}
and the first equation of \eqref{eq:GCL} can be approximated  as
the following semi-discrete conservative finite volume scheme
\begin{align}
\label{eq:semi_U}
&\dfrac{\dd}{\dd t}(J\bU)_{\bm{i}}=
-\sum_{k=1}^3\dfrac{1}{\Delta \xi_k}\delta_k\left[\widehat{\bF}_k\right]_{\bm{i}},
\\
\label{eq:semi_J}
&\dfrac{\dd}{\dd t}J_{\bm{i}}=
-\sum_{k=1}^3\dfrac{1}{\Delta \xi_k}\delta_k\left[J\pd{\xi_k}{t}\right]_{\bm{i}},
\end{align}
where $\delta_k[\cdot]$ is the second-order central difference operator in the $i_k$-direction,
e.g. $\delta_k[a]_{\bm{i}}=a_{\hat{\bm{i}}_{k,+}}-a_{\hat{\bm{i}}_{k,-}}$,
 $J_{\bm{i}}(t)$ and $(J\bU)_{\bm{i}}(t)$ approximate the {cell average values} of
$J\left(t,\bm{\xi}\right)$ and $(J\bU)(t,\bm{\xi})$  over the cell $\bm{i}$, respectively,
and $\widehat{\bF}_k(t)$
is the numerical flux approximating the flux
$\left(J\pd{\xi_k}{t}\bU+\sum\limits_{l=1}^3 J\pd{\xi_k}{x_l}\bF_l\right)(t,\bm{\xi})$, $k=1,2,3$.
The metrics $\left(J\pd{\xi_k}{t}\right)_{\hat{\bm{i}}_{k,\pm}}$ and $\left(J\pd{\xi_k}{x_l}\right)_{\hat{\bm{i}}_{k,\pm}}$
 in \eqref{eq:semi_U}-\eqref{eq:semi_J} are calculated by \eqref{eq:SCLCoeff}-\eqref{eq:VCLCoeff}, see {Section \ref{subsec:GCLs}},
with which the SCLs in the second equation of \eqref{eq:GCL} are satisfied at the discrete level, i.e.
\begin{equation}\label{eq:SCL_dis}
	\sum_{k=1}^3\dfrac{1}{\Delta \xi_k}\delta_k\left[J\pd{\xi_k}{x_l}\right]_{\bm{i}}=0,~l=1,2,3.
\end{equation}

\begin{definition}[EC scheme]
	The semi-discrete scheme \eqref{eq:semi_U}-\eqref{eq:semi_J} is EC
	and corresponding numerical flux
	$\widehat{\bF}_{k}$ is called the {\em EC flux},
	if its solution  satisfies a semi-discrete entropy identity
	\begin{equation*}
	\dfrac{\dd}{\dd t}J_{\bm{i}}\eta(\bU_{\bm{i}}(t))
	+\sum_{k=1}^3\dfrac{1}{\Delta \xi_k}\delta_k\left[\widetilde{q}_{k}(t)\right]_{\bm{i}}=0,
	\end{equation*}
	for some numerical entropy fluxes
	$\widetilde{q}_{k}$ consistent with the continuous entropy flux
	$J\pd{\xi_k}{t}\eta+\sum\limits_{l=1}^3J\pd{\xi_k}{x_l}q_l$.
\end{definition}

The following lemma gives a sufficient condition for the semi-discrete scheme \eqref{eq:semi_U}-\eqref{eq:semi_J} to be EC.

\begin{lem}\label{lem:ECFluxMM}\rm
	Assume that symmetric two-point flux $\tbF_{k,\hat{\bm{i}}_{k,\pm}}$ is consistent with
	$J\pd{\xi_k}{t}\bU+\sum\limits_{l=1}^3 J\pd{\xi_k}{x_l}\bF_l$,
	and satisfies
	\begin{align}\label{eq:ECConditionMM}
	\jump{\bV}_{\hat{\bm{i}}_{k,\pm}}^\mathrm{T}\cdot\tbF_{k,\hat{\bm{i}}_{k,\pm}}=
	\left(J\pd{\xi_k}{t}\right)_{\hat{\bm{i}}_{k,\pm}}\jump{\phi}_{\hat{\bm{i}}_{k,\pm}}
	+\sum_{l=1}^3\left(J\pd{\xi_k}{x_l}\right)_{\hat{\bm{i}}_{k,\pm}}\jump{\psi_l}_{\hat{\bm{i}}_{k,\pm}},
	\end{align}
where $\phi$ and $\psi_l$ are defined in \eqref{EQ:potential},
	then the semi-discrete scheme \eqref{eq:semi_U}-\eqref{eq:semi_J}
with  $\widehat{\bF}_k(t)=\tbF_{k,\hat{\bm{i}}_{k,\pm}}$
	is EC with the numerical entropy fluxes
	\begin{subequations}\label{eq:NumEntropyFluxMM}
	\begin{align*}
	\widetilde{q}_{k,\hat{\bm{i}}_{k,\pm}}=\mean{\bV}_{\hat{\bm{i}}_{k,\pm}}^\mathrm{T}\tbF_{k,\hat{\bm{i}}_{k,\pm}}
	-\left(J\pd{\xi_k}{t}\right)_{\hat{\bm{i}}_{k,\pm}}\mean{\phi}_{\hat{\bm{i}}_{k,\pm}}
	-\sum_{l=1}^3\left(J\pd{\xi_k}{x_l}\right)_{\hat{\bm{i}}_{k,\pm}}\mean{\psi_l}_{\hat{\bm{i}}_{k,\pm}},
	\end{align*}
    \end{subequations}
    where $\jump{a}_{\hat{\bm{i}}_{k,\pm}}$ and $\mean{a}_{\hat{\bm{i}}_{k,\pm}}$
    denote the jumps and the arithmetic means of $a$ in the $i_k$-direction, respectively, e.g.
    \begin{align*}
      &\jump{a}_{\hat{\bm{i}}_{1,+}}=a_{i_1+1,i_2,i_3}-a_{i_1,i_2,i_3},~\jump{a}_{\hat{\bm{i}}_{2,-}}=a_{i_1,i_2,i_3}-a_{i_1,i_2-1,i_3}, \\
      &\mean{a}_{\hat{\bm{i}}_{1,+}}=(a_{i_1+1,i_2,i_3}+a_{i_1,i_2,i_3})/2,~\mean{a}_{\hat{\bm{i}}_{2,-}}=(a_{i_1,i_2,i_3}+a_{i_1,i_2-1,i_3})/2.
    \end{align*}
\end{lem}

\begin{proof}
	Multiplying \eqref{eq:semi_U} by $\bV_{\bm{i}}^\mathrm{T}$ left and using \eqref{eq:semi_J} gives
	\begin{align*}
	\dfrac{\dd}{\dd t}(J_{\bm{i}}\eta_{\bm{i}})=&
	-\sum_{k=1}^3\dfrac{1}{\Delta \xi_k} \left\{\bV_{\bm{i}}^\mathrm{T}\delta_k\left[\tbF_{k}\right]_{\bm{i}}
	-\phi_{\bm{i}}\delta_k\left[J\pd{\xi_k}{t}\right]_{\bm{i}} \right\}.
	\end{align*}
	Utilizing the discrete SCLs \eqref{eq:SCL_dis} gives
	\begin{align*}
	{\dfrac{\dd}{\dd t}(J_{\bm{i}}\eta_{\bm{i}})}=&
	-\sum_{k=1}^3\dfrac{1}{\Delta \xi_k} \left\{\bV_{\bm{i}}^\mathrm{T}\delta_k\left[\tbF_{k}\right]_{\bm{i}}
	-\phi_{\bm{i}}\delta_k\left[J\pd{\xi_k}{t}\right]_{\bm{i}}
	-\sum_{l=1}^3\psi_{l,\bm{i}}\delta_k\left[J\pd{\xi_k}{x_l}\right]_{\bm{i}}\right\}.
	\end{align*}
The term in braces at the right end of the above equation
can be further rearranged as follows
	\begin{align*}
	&\bV_{\bm{i}}^\mathrm{T}\delta_k\left[\tbF_{k}\right]_{\bm{i}}
	-\phi_{\bm{i}}\delta_k\left[J\pd{\xi_k}{t}\right]_{\bm{i}}
	-\sum_{l=1}^3\psi_{l,\bm{i}}\delta_k\left[J\pd{\xi_k}{x_l}\right]_{\bm{i}}\\
	=&\left(\mean{\bV}_{\hat{\bm{i}}_{k,+}}-\frac12\jump{\bV}_{\hat{\bm{i}}_{k,+}}\right)^\mathrm{T}\tbF_{k,\hat{\bm{i}}_{k,+}}
	-\left(\mean{\bV}_{\hat{\bm{i}}_{k,-}}+\frac12\jump{\bV}_{\hat{\bm{i}}_{k,-}}\right)^\mathrm{T}\tbF_{k,\hat{\bm{i}}_{k,-}} \\
	&-\left(\mean{\phi}_{\hat{\bm{i}}_{k,+}}-\frac12\jump{\phi}_{\hat{\bm{i}}_{k,+}}\right)\left(J\pd{\xi_k}{t}\right)_{\hat{\bm{i}}_{k,+}}
	+\left(\mean{\phi}_{\hat{\bm{i}}_{k,-}}+\frac12\jump{\phi}_{\hat{\bm{i}}_{k,-}}\right)\left(J\pd{\xi_k}{t}\right)_{\hat{\bm{i}}_{k,-}} \\
	&-\sum_{l=1}^3\left(\mean{\psi_l}_{\hat{\bm{i}}_{k,+}}-\frac12\jump{\psi_l}_{\hat{\bm{i}}_{k,+}}\right)\left(J\pd{\xi_k}{x_l}\right)_{\hat{\bm{i}}_{k,+}}
	+\sum_{l=1}^3\left(\mean{\psi_1}_{\hat{\bm{i}}_{k,-}}+\frac12\jump{\psi_1}_{\hat{\bm{i}}_{k,-}}\right)\left(J\pd{\xi_k}{x_l}\right)_{\hat{\bm{i}}_{k,-}} \\
	=&\left(\mean{\bV}_{\hat{\bm{i}}_{k,+}}^\mathrm{T}\tbF_{k,\hat{\bm{i}}_{k,+}}
	-\left(J\pd{\xi_k}{t}\right)_{\hat{\bm{i}}_{k,+}}\mean{\phi}_{\hat{\bm{i}}_{k,+}}
	-\sum_{l=1}^3\left(J\pd{\xi_k}{x_l}\right)_{\hat{\bm{i}}_{k,+}}\mean{\psi_l}_{\hat{\bm{i}}_{k,+}}\right) \\
	&-\left(\mean{\bV}_{\hat{\bm{i}}_{k,-}}^\mathrm{T}\tbF_{k,\hat{\bm{i}}_{k,-}}
	-\left(J\pd{\xi_k}{t}\right)_{\hat{\bm{i}}_{k,-}}\mean{\phi}_{\hat{\bm{i}}_{k,-}}
	-\sum_{l=1}^3\left(J\pd{\xi_k}{x_l}\right)_{\hat{\bm{i}}_{k,-}}\mean{\psi_l}_{\hat{\bm{i}}_{k,-}}\right) \\
	=&\widetilde{q}_{k,\hat{\bm{i}}_{k,+}}-\widetilde{q}_{k,\hat{\bm{i}}_{k,-}},
	\end{align*}
	where $a_{\bm{i}}=\mean{a}_{\hat{\bm{i}}_{k,+}}-\frac12\jump{a}_{\hat{\bm{i}}_{k,+}}$
	and $a_{\bm{i}}=\mean{a}_{\hat{\bm{i}}_{k,-}}+\frac12\jump{a}_{\hat{\bm{i}}_{k,-}}$ have been used in the first equality,
	and the condition \eqref{eq:ECConditionMM} has been used in the second equality.
	Moreover, it is easy to check the consistency of the numerical entropy fluxes $\widetilde q_{k,\hat{\bm{i}}_{k,\pm}}$ with
	$ J\pd{\xi_k}{t}\eta+\sum\limits_{l=1}^3J\pd{\xi_k}{x_l}q_l$.
	Thus the scheme \eqref{eq:semi_U} with
	$\widehat{\bF}_{k,\hat{\bm{i}}_{k,\pm}}= \tbF_{k,\hat{\bm{i}}_{k,\pm}}$
	is EC in the sense of
	\begin{align*}
	\dfrac{\dd}{\dd t}J_{\bm{i}}\eta(\bU_{\bm{i}}(t))
	+\sum_{k=1}^3\dfrac{1}{\Delta \xi_k}\delta_k\left[\widetilde{q}_{k}(t)\right]_{\bm{i}}=0.
	\end{align*}
\end{proof}

\begin{rmk}\rm
  The sufficient condition \eqref{eq:ECConditionMM} is different from that in \cite{Duan2020RHD}, due to
  the metrics introduced by the coordinate transformation.
\end{rmk}

\subsection{Discrete GCLs}\label{subsec:GCLs}
 For the transformation \eqref{eq:transf},
we have the following identities
\begin{equation*}
\begin{aligned}
J\pd{\xi_k}{t}=-\sum_{l=1}^3\pd{x_l}{t}\left(J\pd{\xi_k}{x_l}\right),~k=1,2,3,
\end{aligned}
\end{equation*}
and
\begin{equation*}
\begin{aligned}
J\pd{\xi_1}{x_1}=\pd{x_2}{\xi_2}\pd{x_3}{\xi_3}-\pd{x_2}{\xi_3}\pd{x_3}{\xi_2},~
J\pd{\xi_1}{x_2}=\pd{x_3}{\xi_2}\pd{x_1}{\xi_3}-\pd{x_3}{\xi_3}\pd{x_1}{\xi_2},~
J\pd{\xi_1}{x_3}=\pd{x_1}{\xi_2}\pd{x_2}{\xi_3}-\pd{x_1}{\xi_3}\pd{x_2}{\xi_2},\\
J\pd{\xi_2}{x_1}=\pd{x_2}{\xi_3}\pd{x_3}{\xi_1}-\pd{x_2}{\xi_1}\pd{x_3}{\xi_3},~
J\pd{\xi_2}{x_2}=\pd{x_3}{\xi_3}\pd{x_1}{\xi_1}-\pd{x_3}{\xi_1}\pd{x_1}{\xi_3},~
J\pd{\xi_2}{x_3}=\pd{x_1}{\xi_3}\pd{x_2}{\xi_1}-\pd{x_1}{\xi_1}\pd{x_2}{\xi_3},\\
J\pd{\xi_3}{x_1}=\pd{x_2}{\xi_1}\pd{x_3}{\xi_2}-\pd{x_2}{\xi_2}\pd{x_3}{\xi_1},~
J\pd{\xi_3}{x_2}=\pd{x_3}{\xi_1}\pd{x_1}{\xi_2}-\pd{x_3}{\xi_2}\pd{x_1}{\xi_1},~
J\pd{\xi_3}{x_3}=\pd{x_1}{\xi_1}\pd{x_2}{\xi_2}-\pd{x_1}{\xi_2}\pd{x_2}{\xi_1}.
\end{aligned}
\end{equation*}
The last nine identities can be reformulated into the divergence form 
\begin{equation}\label{eq:CMM_SCL}
\begin{aligned}
&J\pd{\xi_1}{x_1}=\dfrac{\partial}{\partial\xi_3}\left(\pd{x_2}{\xi_2}x_3\right)-\dfrac{\partial}{\partial\xi_2}\left(\pd{x_2}{\xi_3}x_3\right),~
J\pd{\xi_1}{x_2}=\dfrac{\partial}{\partial\xi_3}\left(\pd{x_3}{\xi_2}x_1\right)-\dfrac{\partial}{\partial\xi_2}\left(\pd{x_3}{\xi_3}x_1\right),\\
&J\pd{\xi_1}{x_3}=\dfrac{\partial}{\partial\xi_3}\left(\pd{x_1}{\xi_2}x_2\right)-\dfrac{\partial}{\partial\xi_2}\left(\pd{x_1}{\xi_3}x_2\right),\\
&J\pd{\xi_2}{x_1}=\dfrac{\partial}{\partial\xi_1}\left(\pd{x_2}{\xi_3}x_3\right)-\dfrac{\partial}{\partial\xi_3}\left(\pd{x_2}{\xi_1}x_3\right),~
J\pd{\xi_2}{x_2}=\dfrac{\partial}{\partial\xi_1}\left(\pd{x_3}{\xi_3}x_1\right)-\dfrac{\partial}{\partial\xi_3}\left(\pd{x_3}{\xi_1}x_1\right),\\
&J\pd{\xi_2}{x_3}=\dfrac{\partial}{\partial\xi_1}\left(\pd{x_1}{\xi_3}x_2\right)-\dfrac{\partial}{\partial\xi_3}\left(\pd{x_1}{\xi_1}x_2\right),\\
&J\pd{\xi_3}{x_1}=\dfrac{\partial}{\partial\xi_2}\left(\pd{x_2}{\xi_1}x_3\right)-\dfrac{\partial}{\partial\xi_1}\left(\pd{x_2}{\xi_2}x_3\right),~
J\pd{\xi_3}{x_2}=\dfrac{\partial}{\partial\xi_2}\left(\pd{x_3}{\xi_1}x_1\right)-\dfrac{\partial}{\partial\xi_1}\left(\pd{x_3}{\xi_2}x_1\right),\\
&J\pd{\xi_3}{x_3}=\dfrac{\partial}{\partial\xi_2}\left(\pd{x_1}{\xi_1}x_2\right)
-\dfrac{\partial}{\partial\xi_1}\left(\pd{x_1}{\xi_2}x_2\right),
\end{aligned}
\end{equation}
which are useful to compute the discrete metrics and to get the discrete SCLs approximating conservatively \eqref{eq:GCL}
by the so-called conservative metrics method \cite{Thomas1979}.

To establish the discrete SCLs \eqref{eq:SCL_dis},
 using the same discretizations for the first-order
spatial derivatives in \eqref{eq:CMM_SCL} as those
in \eqref{eq:semi_U}-\eqref{eq:semi_J}
gives
\begin{equation}\label{eq:SCLCoeff}
	\begin{aligned}
	&\left(J\pd{\xi_1}{x_1}\right)_{\hat{\bm{i}}_{1,+}}=\dfrac{1}{\Delta\xi_2\Delta\xi_3}
	\left(\delta_3\left[\delta_2\left[x_2\right]\avy{x_3}\right]-\delta_2\left[\delta_3\left[x_2\right]\avz{x_3}\right]\right),\\
	&\left(J\pd{\xi_1}{x_2}\right)_{\hat{\bm{i}}_{1,+}}=\dfrac{1}{\Delta\xi_2\Delta\xi_3}
	\left(\delta_3\left[\delta_2\left[x_3\right]\avy{x_1}\right]-\delta_2\left[\delta_3\left[x_3\right]\avz{x_1}\right]\right),\\
	&\left(J\pd{\xi_1}{x_3}\right)_{\hat{\bm{i}}_{1,+}}=\dfrac{1}{\Delta\xi_2\Delta\xi_3}
	\left(\delta_3\left[\delta_2\left[x_1\right]\avy{x_2}\right]-\delta_2\left[\delta_3\left[x_1\right]\avz{x_2}\right]\right),\\
	&\left(J\pd{\xi_2}{x_1}\right)_{\hat{\bm{i}}_{2,+}}=\dfrac{1}{\Delta\xi_3\Delta\xi_1}
	\left(\delta_1\left[\delta_3\left[x_2\right]\avz{x_3}\right]-\delta_3\left[\delta_1\left[x_2\right]\avx{x_3}\right]\right),\\
	&\left(J\pd{\xi_2}{x_2}\right)_{\hat{\bm{i}}_{2,+}}=\dfrac{1}{\Delta\xi_3\Delta\xi_1}
	\left(\delta_1\left[\delta_3\left[x_3\right]\avz{x_1}\right]-\delta_3\left[\delta_1\left[x_3\right]\avx{x_1}\right]\right),\\
	&\left(J\pd{\xi_2}{x_3}\right)_{\hat{\bm{i}}_{2,+}}=\dfrac{1}{\Delta\xi_3\Delta\xi_1}
	\left(\delta_1\left[\delta_3\left[x_1\right]\avz{x_2}\right]-\delta_3\left[\delta_1\left[x_1\right]\avx{x_2}\right]\right),\\
	&\left(J\pd{\xi_3}{x_1}\right)_{\hat{\bm{i}}_{3,+}}=\dfrac{1}{\Delta\xi_1\Delta\xi_2}
	\left(\delta_2\left[\delta_1\left[x_2\right]\avx{x_3}\right]-\delta_1\left[\delta_2\left[x_2\right]\avy{x_3}\right]\right),\\
	&\left(J\pd{\xi_3}{x_2}\right)_{\hat{\bm{i}}_{3,+}}=\dfrac{1}{\Delta\xi_1\Delta\xi_2}
	\left(\delta_2\left[\delta_1\left[x_3\right]\avx{x_1}\right]-\delta_1\left[\delta_2\left[x_3\right]\avy{x_1}\right]\right),\\
	&\left(J\pd{\xi_3}{x_3}\right)_{\hat{\bm{i}}_{3,+}}=\dfrac{1}{\Delta\xi_1\Delta\xi_2}
	\left(\delta_1\left[\delta_1\left[x_1\right]\avx{x_2}\right]-\delta_1\left[\delta_2\left[x_1\right]\avy{x_2}\right]\right),
	\end{aligned}
\end{equation}
where $\avx{a},\avy{a},\avz{a}$ denote the averages in the $i_1,i_2,i_3$-directions, respectively.
To be more specific, the right hand-side (RHS) of the first equation in
\eqref{eq:SCLCoeff} can be expanded as follows
\begin{align*}
  \dfrac{1}{2\Delta\xi_2\Delta\xi_3}\Big\{
  &\left[(x_2)_{i_1+\frac12,i_2+\frac12,i_3+\frac12}-(x_2)_{i_1+\frac12,i_2-\frac12,i_3+\frac12}\right]
  \left[(x_3)_{i_1+\frac12,i_2+\frac12,i_3+\frac12}+(x_3)_{i_1+\frac12,i_2-\frac12,i_3+\frac12}\right]\\
 -&\left[(x_2)_{i_1+\frac12,i_2+\frac12,i_3-\frac12}-(x_2)_{i_1+\frac12,i_2-\frac12,i_3-\frac12}\right]
  \left[(x_3)_{i_1+\frac12,i_2+\frac12,i_3-\frac12}+(x_3)_{i_1+\frac12,i_2-\frac12,i_3-\frac12}\right]\\
 -&\left[(x_2)_{i_1+\frac12,i_2+\frac12,i_3+\frac12}-(x_2)_{i_1+\frac12,i_2+\frac12,i_3-\frac12}\right]
  \left[(x_3)_{i_1+\frac12,i_2+\frac12,i_3+\frac12}+(x_3)_{i_1+\frac12,i_2+\frac12,i_3-\frac12}\right]\\
 +&\left[(x_2)_{i_1+\frac12,i_2-\frac12,i_3+\frac12}-(x_2)_{i_1+\frac12,i_2-\frac12,i_3-\frac12}\right]
  \left[(x_3)_{i_1+\frac12,i_2-\frac12,i_3+\frac12}+(x_3)_{i_1+\frac12,i_2-\frac12,i_3-\frac12}\right]
  \Big\}.
\end{align*}
Based on the above discretizations, it can be verified that the SCLs \eqref{eq:SCL_dis} are satisfied.
For example,
\begin{align*}
   \sum_{k=1}^3\dfrac{1}{\Delta \xi_k}\delta_k\left[J\pd{\xi_k}{x_1}\right]_{\bm{i}}
  =\dfrac{1}{2\Delta\xi_1\Delta\xi_2\Delta\xi_3}
  \Big(&\delta_1\delta_3\left[\delta_2\left[x_2\right]\avy{x_3}\right]-\delta_1\delta_2\left[\delta_3\left[x_2\right]\avz{x_3}\right]
  +\delta_2\delta_1\left[\delta_3\left[x_2\right]\avz{x_3}\right]\\
  &-\delta_2\delta_3\left[\delta_1\left[x_2\right]\avx{x_3}\right]
  +\delta_3\delta_2\left[\delta_1\left[x_2\right]\avx{x_3}\right]-\delta_3\delta_1\left[\delta_2\left[x_2\right]\avy{x_3}\right]\Big)=0,
\end{align*}
since $\delta_l$ and $\delta_k$ are commutative, i.e. $\delta_l\delta_k=\delta_k\delta_l$.

The following lemma tells us that the scheme also preserves the free-stream states by
integrating \eqref{eq:semi_U}-\eqref{eq:semi_J} with the same explicit SSP
RK  schemes \cite{Gottlieb2001Strong}.

\begin{lem}\label{lem:GCL}\rm
	If the semi-discrete scheme \eqref{eq:semi_U}-\eqref{eq:semi_J} is
  integrated in time with the
	explicit SSP RK scheme from $t=t_n$ to
$t_{n+1}=t_n+\Delta t_n$, then
 the resulting fully-discrete scheme preserves the free-stream states.
\end{lem}

\begin{proof}
	The forward Euler time discretization is only considered here, since the explicit SSP RK schemes
	are a convex combination of the forward Euler time discretizations.
	Assuming that $\bU_{\bm{i}}^n=\bU_0$ is a physical constant state and the time step size is $\Delta t_n$,
	then the update of the metric Jacobian $J_{\bm{i}}$ and the solution $\bU_{\bm{i}}$ can be rewritten as follows
	\begin{align*}
	J_{\bm{i}}^{n+1}=&J_{\bm{i}}^{n}-\sum_{k=1}^3\dfrac{\Delta t_n}{\Delta \xi_k}\delta_k\left[J\pd{\xi_k}{t}\right]_{\bm{i}},\\
	(J\bU)_{\bm{i}}^{n+1}=&(J\bU)_{\bm{i}}^{n}-\sum_{k=1}^3\dfrac{\Delta t_n}{\Delta \xi_k}\delta_k\left[\widehat{\bF}_k\right]_{\bm{i}}\\
	=&J_{\bm{i}}^n\bU_0-\sum_{k=1}^3\dfrac{\Delta t_n}{\Delta \xi_k}\Bigg[\left(J\pd{\xi_k}{t}\right)_{\hat{\bm{i}}_{k,+}}\bU_0
	+\sum\limits_{l=1}^3 \left(J\pd{\xi_k}{x_l}\right)_{\hat{\bm{i}}_{k,+}}\bF_l(\bU_0) \\
	&-\left(J\pd{\xi_k}{t}\right)_{\hat{\bm{i}}_{k,-}}\bU_0
	-\sum\limits_{l=1}^3 \left(J\pd{\xi_k}{x_l}\right)_{\hat{\bm{i}}_{k,-}}\bF_l(\bU_0)\Bigg]\\
	=&\left(J_{\bm{i}}^{n}-\sum_{k=1}^3\dfrac{\Delta t_n}{\Delta \xi_k}\delta_k\left[J\pd{\xi_k}{t}\right]_{\bm{i}}\right)\bU_0
	-\sum_{l=1}^3\left(\sum_{k=1}^3\dfrac{\Delta t_n}{\Delta \xi_k}\delta_k\left[J\pd{\xi_k}{x_l}\right]_{\bm{i}}\right)\bF_l(\bU_0)\\
	=&J_{\bm{i}}^{n+1}\bU_0,
	\end{align*}
	where the discrete GCLs have been used in the last equality. Thus $\bU_{\bm{i}}^{n+1}=(J\bU)_{\bm{i}}^{n+1}/J_{\bm{i}}^{n+1}=\bU_0$.
The proof is completed.
\end{proof}

  In the above proof, no specific form of the ``fluxes'' $\left(J\partial_t{\xi_k}\right)_{\hat{\bm{i}}_{k,\pm}}$ in \eqref{eq:semi_J}
  are given. Two suggested  versions of the ``fluxes''   $\left(J\partial_t{\xi_k}\right)_{\hat{\bm{i}}_{k,\pm}}$
   are presented here and compared below. It is worth noting that they do not affect the conclusion of Lemma \ref{lem:GCL}.
  The first version  is given by
\begin{equation}\label{eq:VCLCoeff}
\begin{aligned}
&\left(J\pd{\xi_1}{t}\right)_{\hat{\bm{i}}_{1,\pm}}=-\sum_{l=1}^3(\dot{x_l})_{\hat{\bm{i}}_{1,\pm}}\left(J\pd{\xi_1}{x_l}\right)_{\hat{\bm{i}}_{1,\pm}},\\
&\left(J\pd{\xi_2}{t}\right)_{\hat{\bm{i}}_{2,\pm}}=-\sum_{l=1}^3(\dot{x_l})_{\hat{\bm{i}}_{2,\pm}}\left(J\pd{\xi_2}{x_l}\right)_{\hat{\bm{i}}_{2,\pm}},\\
&\left(J\pd{\xi_3}{t}\right)_{\hat{\bm{i}}_{3,\pm}}=-\sum_{l=1}^3(\dot{x_l})_{\hat{\bm{i}}_{3,\pm}}\left(J\pd{\xi_3}{x_l}\right)_{\hat{\bm{i}}_{3,\pm}},
\end{aligned}
\end{equation}
where the ``mesh'' velocities { $(\dot{\bx})_{\hat{\bm{i}}_{1,\pm}},
  (\dot{\bx})_{\hat{\bm{i}}_{2,\pm}}$, and
  $(\dot{\bx})_{\hat{\bm{i}}_{3,\pm}}$} may be calculated by the arithmetic mean, e.g.
\begin{align*}
(\dot{\bx})_{i_1+\frac12,i_2,i_3}=\dfrac14
\left[(\dot{\bx})_{i_1+\frac12,i_2-\frac12,i_3-\frac12}
+(\dot{\bx})_{i_1+\frac12,i_2-\frac12,i_3+\frac12}
+(\dot{\bx})_{i_1+\frac12,i_2+\frac12,i_3-\frac12}
+(\dot{\bx})_{i_1+\frac12,i_2+\frac12,i_3+\frac12}\right].
\end{align*}
Here $(\dot{\bx})_{i_1+\frac12,i_2+\frac12,i_3+\frac12}$ is the mesh
velocity of the mesh point $({\bx})_{i_1+\frac12,i_2+\frac12,i_3+\frac12}$,
which will be provided by solving the mesh equations in Section \ref{section:MM}.  Combining \eqref{eq:VCLCoeff} with
\eqref{eq:semi_J} gives our first semi-discrete VCL, denoted by {\tt VCL1}, which is easy to be implemented.

  The  second version of the ``fluxes''   $\left(J\partial_t{\xi_k}\right)_{\hat{\bm{i}}_{k,\pm}}$ is based on the reformulations of the Jacobian and the temporal metrics \cite{Abe2013Conservative} as follows
{\small \begin{equation}\label{eq:CMM_VCL}
\begin{aligned}
J=&\dfrac{\partial}{\partial\xi_3}\left\{\left[\dfrac{\partial}{\partial\xi_2}\left(\pd{x_1}{\xi_1}x_2\right)
-\dfrac{\partial}{\partial\xi_1}\left(\pd{x_1}{\xi_2}x_2\right)\right]x_3\right\}
+\dfrac{\partial}{\partial\xi_2}\left\{\left[\dfrac{\partial}{\partial\xi_1}\left(\pd{x_1}{\xi_3}x_2\right)
-\dfrac{\partial}{\partial\xi_3}\left(\pd{x_1}{\xi_1}x_2\right)\right]x_3\right\} \\
&+\dfrac{\partial}{\partial\xi_1}\left\{\left[\dfrac{\partial}{\partial\xi_3}\left(\pd{x_1}{\xi_2}x_2\right)
-\dfrac{\partial}{\partial\xi_2}\left(\pd{x_1}{\xi_3}x_2\right)\right]x_3\right\},\\
J\pd{\xi_1}{t}=&\dfrac{\partial}{\partial\xi_2}\left\{\left[\dfrac{\partial}{\partial\xi_3}\left(\pd{x_1}{ t}x_2\right)
-\dfrac{\partial}{\partial t}\left(\pd{x_1}{\xi_3}x_2\right)\right]x_3\right\}
+\dfrac{\partial}{\partial\xi_3}\left\{\left[\dfrac{\partial}{\partial t}\left(\pd{x_1}{\xi_2}x_2\right)
-\dfrac{\partial}{\partial\xi_2}\left(\pd{x_1}{ t}x_2\right)\right]x_3\right\} \\
&+\dfrac{\partial}{\partial t}\left\{\left[\dfrac{\partial}{\partial\xi_2}\left(\pd{x_1}{\xi_3}x_2\right)
-\dfrac{\partial}{\partial\xi_3}\left(\pd{x_1}{\xi_2}x_2\right)\right]x_3\right\}, \\
J\pd{\xi_2}{t}=&\dfrac{\partial}{\partial\xi_3}\left\{\left[\dfrac{\partial}{\partial\xi_1}\left(\pd{x_1}{ t}x_2\right)
-\dfrac{\partial}{\partial t}\left(\pd{x_1}{\xi_1}x_2\right)\right]x_3\right\}
+\dfrac{\partial}{\partial t}\left\{\left[\dfrac{\partial}{\partial\xi_3}\left(\pd{x_1}{\xi_1}x_2\right)
-\dfrac{\partial}{\partial\xi_1}\left(\pd{x_1}{\xi_3}x_2\right)\right]x_3\right\} \\
&+\dfrac{\partial}{\partial\xi_1}\left\{\left[\dfrac{\partial}{\partial t}\left(\pd{x_1}{\xi_3}x_2\right)
-\dfrac{\partial}{\partial\xi_3}\left(\pd{x_1}{ t}x_2\right)\right]x_3\right\}, \\
J\pd{\xi_3}{t}=&\dfrac{\partial}{\partial t}\left\{\left[\dfrac{\partial}{\partial\xi_1}\left(\pd{x_1}{\xi_2}x_2\right)
-\dfrac{\partial}{\partial\xi_2}\left(\pd{x_1}{\xi_1}x_2\right)\right]x_3\right\}
+\dfrac{\partial}{\partial\xi_1}\left\{\left[\dfrac{\partial}{\partial\xi_2}\left(\pd{x_1}{ t}x_2\right)
-\dfrac{\partial}{\partial t}\left(\pd{x_1}{\xi_2}x_2\right)\right]x_3\right\} \\
&+\dfrac{\partial}{\partial\xi_2}\left\{\left[\dfrac{\partial}{\partial t}\left(\pd{x_1}{\xi_1}x_2\right)
-\dfrac{\partial}{\partial\xi_1}\left(\pd{x_1}{ t}x_2\right)\right]x_3\right\},
\end{aligned}
\end{equation}}%
and  the second-order central difference and average  approximated
the spatial derivatives, similar to \eqref{eq:SCLCoeff},
thus the  VCL in 
\eqref{eq:semi_J} is approximated in space  by
\begin{align}\nonumber
\dfrac{\dd J_{\bm{i}}}{\dd t}=&-\sum_{l=1}^3\dfrac{1}{\Delta\xi_k}\delta_k\left[J\pd{\xi_k}{t}\right]_{\bm{i}} \\ \nonumber
=&-\dfrac{1}{\Delta\xi_1}\delta_1\left[\dfrac{\partial}{\partial t}\left\{\left[\dfrac{\partial}{\partial\xi_2}\left(\pd{x_1}{\xi_3}x_2\right)
-\dfrac{\partial}{\partial\xi_3}\left(\pd{x_1}{\xi_2}x_2\right)\right]x_3\right\}\right]_{\bm{i}} \\ \nonumber
&-\dfrac{1}{\Delta\xi_2}\delta_2\left[\dfrac{\partial}{\partial t}\left\{\left[\dfrac{\partial}{\partial\xi_3}\left(\pd{x_1}{\xi_1}x_2\right)
-\dfrac{\partial}{\partial\xi_1}\left(\pd{x_1}{\xi_3}x_2\right)\right]x_3\right\}\right]_{\bm{i}} \\ \nonumber
&-\dfrac{1}{\Delta\xi_3}\delta_3\left[\dfrac{\partial}{\partial t}\left\{\left[\dfrac{\partial}{\partial\xi_1}\left(\pd{x_1}{\xi_2}x_2\right)
-\dfrac{\partial}{\partial\xi_2}\left(\pd{x_1}{\xi_1}x_2\right)\right]x_3\right\}\right]_{\bm{i}}\\
=:&-\dfrac{1}{\Delta\xi_1}\delta_1\left[\dfrac{\partial}{\partial t}A_1\right]_{\bm{i}}
   -\dfrac{1}{\Delta\xi_2}\delta_2\left[\dfrac{\partial}{\partial t}A_2\right]_{\bm{i}}
   -\dfrac{1}{\Delta\xi_3}\delta_3\left[\dfrac{\partial}{\partial t}A_3\right]_{\bm{i}},\label{eq:VCL_update}
\end{align}
which gives our second semi-discrete VCL, denoted by {\tt VCL2}.
Obviously, it requires more operation, but it
can well approach to the value of the Jacobian $J$ calculated by the first equation of \eqref{eq:CMM_VCL}.
In fact,
if the mesh trajectories are assumed to be linear in time as follows
\begin{equation}\label{eq:mesh_linear}
\bx(t)=\dfrac{t_{n+1}-t}{\Delta t_n}\bx^n + \dfrac{t-t_n}{\Delta t_n}\bx^{n+1},~t\in[t_n,t_{n+1}],~\Delta t_n=t_{n+1}-t_n,
\end{equation}
inspired by \cite{Pathak2016Adaptive},
then the terms  $A_k,k=1,2,3$ are cubic polynomials of $t$,
so that $\dfrac{\partial}{\partial t}A_k$ is a quadratic polynomial of $t$, which can be expressed as
\begin{align}\nonumber
\dfrac{\partial}{\partial t}A_k
=&\dfrac{1}{2(\Delta t_n)^3}\Bigg[(\Delta t_n)^2\left(-11A_k^{n}+18A_k^{n+\frac13}-9A_k^{n+\frac23}+2A_k^{n+1}\right)
\\ \nonumber
&+18\Delta t_n\left(2A_k^{n}-5A_k^{n+\frac13}+4A_k^{n+\frac23}-A_k^{n+1}\right)(t-t_n)\\
&-27\left(A_k^{n}-3A_k^{n+\frac13}+3A_k^{n+\frac23}-A_k^{n+1}\right)(t-t_n)^2\Bigg],
\label{eq:quadratic} \end{align}
where the superscript denotes the value at corresponding time level.
If  following the first equation in \eqref{eq:CMM_VCL} to compute $J$ at time $t_m$ by
\begin{equation}\label{eq:J_discrete}
J^m_{\bm{i}}=-\sum_{k=1}^3\dfrac{1}{\Delta\xi_k}\delta_k\left[A_k^m\right]_{\bm{i}},
\end{equation}
then  substituting \eqref{eq:quadratic} into \eqref{eq:VCL_update} and using \eqref{eq:J_discrete} gives
\begin{align}\nonumber
\dfrac{\dd J}{\dd t}=&\dfrac{1}{2(\Delta t_n)^3}\Bigg[
(\Delta t_n)^2\left(-11J^{n}+18J^{n+\frac13}-9J^{n+\frac23}+2J^{n+1}\right)\\
\nonumber &+18\Delta t_n\left(2J^{n}-5J^{n+\frac13}+4J^{n+\frac23}-J^{n+1}\right)(t-t_n)\\
&-27\left(J^{n}-3J^{n+\frac13}+3J^{n+\frac23}-J^{n+1}\right)(t-t_n)^2\Bigg].
\label{eq:J_ODE} \end{align}
Here $J^{n},J^{n+\frac13},J^{n+\frac23},J^{n+1}$ are  known, so that \eqref{eq:J_ODE} is a linear ordinary differential equation (ODE)
with the RHS of a quadratic polynomial of $t$.
If the third-order SSP RK method is used to integrate \eqref{eq:J_ODE}, then it will hold
exactly. In the 2D case, it will be a linear ODE with the RHS of a linear polynomial of $t$, so the second-order SSP RK method is enough.

\begin{rmk}
  In the 3D case, the second-order SSP RK method can approximate \eqref{eq:semi_J} to the second-order accuracy and no obvious difference of the solutions is found between  the SSP second- and third-order and  RK methods, see the numerical results in Section \ref{section:Num}.  
\end{rmk}

\subsection{EC flux}
What follows is to find an EC flux satisfying \eqref{eq:ECConditionMM}.
For the second-order accurate scheme,  we choose the EC flux as follows
\begin{align}\label{eq:ECfluxMM}
\tbF_{k,\hat{\bm{i}}_{k,\pm}}=\left(J\pd{\xi_k}{t}\right)_{\hat{\bm{i}}_{k,\pm}}\widetilde{\bU}_{\hat{\bm{i}}_{k,\pm}}^{\rm RHD}
+\sum_{l=1}^3\left(J\pd{\xi_k}{x_l}\right)_{\hat{\bm{i}}_{k,\pm}}\widetilde{\bF}_{l,\hat{\bm{i}}_{k,\pm}}^{\rm RHD},
\end{align}
where $\widetilde{\bF}_{l,\hat{\bm{i}}_{k,\pm}}^{\rm RHD}$ is the EC flux
of the RHD equations on the static mesh satisfying
\begin{align*}
&\jump{\bV}_{\hat{\bm{i}}_{k,\pm}}^\mathrm{T}{\tbF}_{l,\hat{\bm{i}}_{k,\pm}}^{\rm RHD}=\jump{\psi_l}_{\hat{\bm{i}}_{k,\pm}},
\end{align*}
and $\widetilde{\bU}_{\hat{\bm{i}}_{k,\pm}}^{\rm RHD}$ is  obtained by the same procedure in \cite{Duan2020RMHD} satisfying
\begin{align*}
&\jump{\bV}_{\hat{\bm{i}}_{k,\pm}}^\mathrm{T}\widetilde{\bU}_{\hat{\bm{i}}_{k,\pm}}^{\rm RHD}=\jump{\phi}_{\hat{\bm{i}}_{k,\pm}}.
\end{align*}
For example, the specific expressions of $\tbF_1^{\rm RHD}$ and $\widetilde{\bU}^{\rm RHD}$ are respectively given as follows
\begin{align*}
&\tbF_1^{\rm RHD}=
\begin{pmatrix}
\meanln{\rho}\mean{W\vx},\\
\dfrac{\mean{W\vx}}{\mean{W}} \tbF_{1,5}^{\rm RHD}+\dfrac{\mean{\rho}}{\mean{\beta}}\\
\dfrac{\mean{W\vy}}{\mean{W}} \tbF_{1,5}^{\rm RHD}\\
\dfrac{\mean{W\vz}}{\mean{W}} \tbF_{1,5}^{\rm RHD}\\
\left(\mean{W}^2-\sum_{k=1}^3\mean{Wv_k}^2\right)^{-1}\mean{W}
\left(\mean{\rho}\mean{W\vx}/\mean{\beta}+\alpha_0\tbF_{1,1}^{\rm RHD}\right)
\end{pmatrix}, \\
&\widetilde{\bU}^{\rm RHD}=
\begin{pmatrix}
\meanln{\rho}\mean{W}\\
\dfrac{\mean{W\vx}}{\mean{W}}\left(\dfrac{\mean{\rho}}{\mean{\beta}}+\widetilde{\bU}_5^{\rm RHD}\right)\\
\dfrac{\mean{W\vy}}{\mean{W}}\left(\dfrac{\mean{\rho}}{\mean{\beta}}+\widetilde{\bU}_5^{\rm RHD}\right)\\
\dfrac{\mean{W\vz}}{\mean{W}}\left(\dfrac{\mean{\rho}}{\mean{\beta}}+\widetilde{\bU}_5^{\rm RHD}\right)\\
\left(\mean{W}^2-\sum_{k=1}^3\mean{Wv_k}^2\right)^{-1}\mean{W}
\left(\dfrac{\mean{\rho}\sum_{k=1}^3\mean{Wv_k}^2}{\mean{\beta}\mean{W}}
+\meanln{\rho}\mean{W}\alpha_0\right)
\end{pmatrix},
\end{align*}
where $\meanln{a}=\jump{a}/\jump{\ln{a}}$ is the logarithmic mean, see   \cite{Ismail2009Affordable},
 $\alpha_0=1+1/(\Gamma-1)/\meanln{\beta}$, $\beta=\rho/p$,
and $\tbF_{1,5}^{\rm RHD}$ and $\widetilde{\bU}_5^{\rm RHD}$ denote the $5$-th component of $\tbF_1^{\rm RHD}$ and $\widetilde{\bU}^{\rm RHD}$, respectively.

\subsection{ES schemes}
It is known that for the quasi-linear hyperbolic conservation laws,
the entropy identity is available only if the solution is smooth.
In other words, the entropy is not conserved if the discontinuities such as the
shock waves appear in the solution. Moreover, the EC scheme may produce serious nonphysical oscillations near
the discontinuities. Those motivate us to develop the ES scheme (satisfying the entropy inequality for the given entropy pair) in this section by adding
a suitable dissipation term to the EC flux \eqref{eq:ECfluxMM}.

Following \cite{Tadmor1987The}, adding a dissipation term to the EC flux
$\tbF_{k,\hat{\bm{i}}_{k,\pm}}$ gives the ES flux
\begin{align}\label{eq:ESFluxMM}
&\widehat{\bF}_{k,\hat{\bm{i}}_{k,\pm}}=\tbF_{k,\hat{\bm{i}}_{k,\pm}}-\dfrac12 \bm{D}_{\hat{\bm{i}}_{k,\pm}}\jump{\bV}_{\hat{\bm{i}}_{k,\pm}},
\end{align}
satisfying
\begin{align}
\jump{\bV}_{\hat{\bm{i}}_{k,\pm}}^\mathrm{T}\cdot\widehat{\bF}_{k,\hat{\bm{i}}_{k,\pm}}
-\left(J\pd{\xi_k}{t}\right)_{\hat{\bm{i}}_{k,\pm}}\jump{\phi}_{\hat{\bm{i}}_{k,\pm}}
-\sum_{l=1}^3\left(J\pd{\xi_k}{x_l}\right)_{\hat{\bm{i}}_{k,\pm}}\jump{\psi_l}_{\hat{\bm{i}}_{k,\pm}}\leqslant 0,
\end{align}
where $\bm{D}_{\hat{\bm{i}}_{k,\pm}}$ is a symmetric positive semi-definite matrix.
It is easy to prove that the scheme \eqref{eq:semi_U}-\eqref{eq:semi_J} with the numerical flux \eqref{eq:ESFluxMM}
is ES, that is,  it satisfies the semi-discrete entropy inequality
\begin{align*}
	\dfrac{\dd}{\dd t}J_{\bm{i}}\eta(\bU_{\bm{i}}(t))
+\sum_{k=1}^3\dfrac{1}{\Delta \xi_k}\delta_k\left[\widehat{q}_{k}(t)\right]_{\bm{i}}\leqslant0,
\end{align*}
with the numerical entropy flux   
\begin{equation*}
  \widehat{q}_{k,\hat{\bm{i}}_{k,\pm}}=\widetilde{q}_{k,\hat{\bm{i}}_{k,\pm}}
  -\dfrac12\mean{\bV}_{\hat{\bm{i}}_{k,\pm}}\bm{D}_{\hat{\bm{i}}_{k,\pm}}
  \jump{\bV}_{\hat{\bm{i}}_{k,\pm}},
\end{equation*}
being consistent with the continuous entropy flux
$J\pd{\xi_k}{t}\eta+\sum\limits_{l=1}^3J\pd{\xi_k}{x_l}q_l$.

Let us give a choice of $\bm{D}_{\hat{\bm{i}}_{k,\pm}}$ in the ES flux \eqref{eq:ESFluxMM}.
According to \cite{Merriam1989An}, there exists a set of scaled eigenvectors $\bm{R}$ such that
\begin{equation*}
\pd{\bU}{\bV}=\bm{R}\bm{R}^\mathrm{T},\
\pd{\bF_1}{\bU}=\bm{R}\bm{\Lambda}\bm{R}^{-1},\
\bm{\Lambda}=\mbox{diag}\{\lambda_1,\ldots,\lambda_5\},
\end{equation*}
where the eigenvalues $\lambda_1,\ldots,\lambda_5$ are given by
$$ \lambda_1=\lambda_-, \ \lambda_5=\lambda_+, \
 \lambda_\ell=v_1, \ \ell=2,3,4, \
\lambda_{\pm}=\dfrac{v_1(1-c_s^2)\pm
    c_s/W\sqrt{1-v_1^2-(\abs{\bv}^2-v_1^2)c_s^2}}{1-\abs{\bv}^2c_s^2},
    $$ and
  \begin{align*}
  \bm{R}=
  \begin{bmatrix}
  	1                            & 1/W & Wv_2           & Wv_3           & 1                            \\
  	hW\mathcal{A}_{-}\lambda_{-} & v_1 & 2hW^2v_1v_2    & 2hW^2v_1v_3    & hW\mathcal{A}_{+}\lambda_{+} \\
  	hWv_2                        & v_2 & h(1+2W^2v_2^2) & 2hW^2v_2v_3    & hWv_2                        \\
  	hWv_3                        & v_3 & 2hW^2v_2v_3    & h(1+2W^2v_3^2) & hWv_3                        \\
  	hW\mathcal{A}_{-}            & 1   & 2hW^2v_2       & 2hW^2v_3       & hW\mathcal{A}_{+}
  \end{bmatrix} \\ \times
  \begin{bmatrix}
  	\sqrt{\frac{\mathcal{B}-\mathcal{C}}{2}} & 0                                        & 0                                           & 0                                                  & 0                                        \\
  	0                                        & \sqrt{\frac{(\Gamma-1)\rho W^3}{\Gamma}} & 0                                           & 0                                                  & 0                                        \\
  	0                                        & 0                                        & \sqrt{\frac{pW(1-v_1^2-v_2^2)}{h(1-v_1^2)}} & 0                                                  & 0                                        \\
  	0                                        & 0                                        & -v_2v_3\sqrt{\frac{pW}{h(1-v_1^2)(1-v_1^2-v_2^2)}}                                           & \sqrt{\frac{p}{hW(1-v_1^2-v_2^2)}}                 & 0                                        \\
  	0                                        & 0                                        & 0                                           & 0 & \sqrt{\frac{\mathcal{B}+\mathcal{C}}{2}}
  \end{bmatrix},
  \end{align*}
  here $\mathcal{A}_{\pm}=\dfrac{1-v_1^2}{1-v_1\lambda_{\pm}}$,
  $\mathcal{B}=\dfrac{\rho W(1-v_1^2-(\abs{\bv}^2-v_1^2)c_s^2)}{\Gamma(1-v_1^2)}$,
  $\mathcal{C}=\dfrac{\rho v_1c_s\sqrt{1-v_1^2-(\abs{\bv}^2-v_1^2)c_s^2}}{\Gamma(1-v_1^2)}$.
}
Using the rotational invariance gives
\begin{align*}
 &\partial{\left(J\pd{\xi_k}{t}\bU+\sum\limits_{l=1}^3 J\pd{\xi_k}{x_l}\bF_l\right)}/{\partial\bU} \\
=&\partial{\left(J\pd{\xi_k}{t}\bU+L_k\bT^{-1}\bF_1(\bT\bU)\right)}/{\partial\bU} \\
=&J\pd{\xi_k}{t}\bm{I}+L_k\bT^{-1}\bm{R}(\bT\bU)\bm{\Lambda}(\bT\bU)\bm{R}^{-1}(\bT\bU)\bT\\
=&\bT^{-1}\bm{R}(\bT\bU)\left(J\pd{\xi_k}{t}\bm{I}+L_k\bm{\Lambda}(\bT\bU)\right)\bm{R}^{-1}(\bT\bU)\bT,
\end{align*}
where $L_k=\sqrt{\sum\limits_{l=1}^3\left(J\pd{\xi_k}{x_l}\right)^2}$,
and $\bT$ denotes the ``rotational'' matrix defined by
\begin{align*}
&\bT =
\begin{bmatrix}
	1 & 0                      & 0                      & 0           & 0 \\
	0 & \cos\varphi\cos\theta  & \cos\varphi\sin\theta  & \sin\varphi & 0 \\
	0 & -\sin\theta            & \cos\theta             & 0           & 0 \\
	0 & -\sin\varphi\cos\theta & -\sin\varphi\sin\theta & \cos\varphi & 0 \\
	0 & 0                      & 0                      & 0           & 1
\end{bmatrix},\\
&\theta = \arctan\left(\left(J\pd{\xi_k}{x_2}\right)\Big/\left(J\pd{\xi_k}{x_1}\right)\right),\\
&\varphi = \arctan\left(\left(J\pd{\xi_k}{x_3}\right)\Bigg/\sqrt{\left(J\pd{\xi_k}{x_1}\right)^2+\left(J\pd{\xi_k}{x_2}\right)^2}\right).
\end{align*}
Then following the dissipation term in the Roe scheme yields
 \begin{align*}
  &-\dfrac12\bT^{-1}\bm{R}(\bT\bU)\left|J\pd{\xi_k}{t}\bm{I}+L_k\bm{\Lambda}(\bT\bU)\right|\bm{R}^{-1}(\bT\bU)\bT\jump{\bU} \\
 =&-\dfrac12\bT^{-1}\bm{R}(\bT\bU)\left|J\pd{\xi_k}{t}\bm{I}+L_k\bm{\Lambda}(\bT\bU)\right|\bm{R}^{-1}(\bT\bU)\jump{\bT\bU} \\
 \approx&-\dfrac12\bT^{-1}\bm{R}(\bT\bU)\left|J\pd{\xi_k}{t}\bm{I}+L_k\bm{\Lambda}(\bT\bU)\right|\bm{R}^{-1}(\bT\bU)
 \bm{R}(\bT\bU)\bm{R}^\mathrm{T}(\bT\bU)\bT\jump{\bV} \\
 =&-\dfrac12\bT^{-1}\bm{R}(\bT\bU)\left|J\pd{\xi_k}{t}\bm{I}+L_k\bm{\Lambda}(\bT\bU)\right|\bm{R}^\mathrm{T}(\bT\bU)\bT\jump{\bV}.
 \end{align*}
Based on that, the matrix $\bm{D}_{\hat{\bm{i}}_{k,\pm}}$ in \eqref{eq:ESFluxMM} can be chosen as follows (evaluated at the {interface point} ${\hat{\bm{i}}_{k,\pm}}$)
\begin{equation*}
\bm{D}=\bT^{-1}\bm{R}(\bT\bU)\left|J\pd{\xi_k}{t}\bm{I}+L_k\bm{\Lambda}(\bT\bU)\right|\bm{R}^\mathrm{T}(\bT\bU)\bT,
\end{equation*}
where the matrix $\left|J\pd{\xi_k}{t}\bm{I}+L_k\bm{\Lambda}(\bT\bU)\right|$ is taken as
\begin{equation*}
\left|J\pd{\xi_k}{t}\bm{I}+L_k\bm{\Lambda}(\bT\bU)\right|:=\max\left\{ \left| J\pd{\xi_k}{t}+L_k\lambda_1(\bT\bU) \right|,\dots, \left| J\pd{\xi_k}{t}+L_k\lambda_5(\bT\bU) \right|\right\}\bm{I},
\end{equation*}
and the values of $\left|J\pd{\xi_k}{t}\bm{I}+L_k\bm{\Lambda}(\bT\bU)\right|_{\hat{\bm{i}}_{k,\pm}}$
and $\bm{R}_{\hat{\bm{i}}_{k,\pm}}(\bT\bU)$ are calculated by
using the arithmetic mean values of the left and right states.

To obtain a second-order accurate ES scheme, the dissipation term in \eqref{eq:ESFluxMM} has to be improved.
Here the second-order TVD reconstruction is performed in the scaled entropy variables
$\bw=\bm{R}^\mathrm{T}\bV$.
More specifically,  the linear reconstruction of $\bw$ with the minmod limiter is used in the $i_k$-direction to obtain
the left and right limit values at $\hat{\bm{i}}_{k,+}$, denoted by
$\bw_{\hat{\bm{i}}_{k,+}}^-$ and $\bw_{\hat{\bm{i}}_{k,+}}^+$, and then to define
\begin{equation*}
\jumpangle{\bw}_{\hat{\bm{i}}_{k,+}}=\bw_{\hat{\bm{i}}_{k,+}}^+
-\bw_{\hat{\bm{i}}_{k,+}}^-.
\end{equation*}
Because of the ``sign'' property
\begin{equation*}
\text{sign}(\jumpangle{\bw}_{\hat{\bm{i}}_{k,+}})
=\text{sign}(\jump{\bw}_{\hat{\bm{i}}_{k,+}}),
\end{equation*}
utilizing the reconstructed jump in the dissipation term can give the following second-order ES scheme
\begin{equation}\label{eq:ES2nd}
\dfrac{\dd}{\dd t}(J\bU)_{\bm{i}}=
-\sum_{k=1}^3\dfrac{1}{\Delta \xi_k}\delta_k\left[\widehat{\bF}_k^{\mbox{\scriptsize 2nd}}\right]_{\bm{i}},
\end{equation}
where
\begin{equation}\label{eq:HOstable}
\widehat{\bF}_{k,\hat{\bm{i}}_{k,\pm}}^{\mbox{\scriptsize 2nd}}=\tbF_{k,\hat{\bm{i}}_{k,\pm}}
-\dfrac12 \bm{D}_{\hat{\bm{i}}_{k,\pm}}\jumpangle{\bw}_{\hat{\bm{i}}_{k,\pm}}.
\end{equation}

\begin{rmk}
  The ES schemes  preserve the free-stream states since the dissipation terms are given by using the jump of the entropy variables,
 which vanish as the solution is a constant state.
\end{rmk}

%% file: MovingMesh.tex
\section{Adaptive moving mesh strategy}\label{section:MM}
This section  presents our adaptive moving mesh strategy at time $t=t_n$,
 but focus on the mesh iteration redistribution. The dependence of the variables on $t$ will be omitted, unless otherwise stated.
Consider the following mesh adaption functional 
\begin{equation}\label{eq:mesh_func}
  \widetilde{E}(\bx)=\dfrac12\sum_{k=1}^3\int_{\Omega_l}\left(\nabla_{\bm{\xi}}x_k\right)^\mathrm{T}\bm{G}_k\left(\nabla_{\bm{\xi}}x_k\right)\dd\bm{\xi},
\end{equation}
where $\bm{G}_k$ is the given symmetric positive definite matrix, depending on the solution $\bU$.
More terms can be added to the above functional to control other aspects of the mesh such as the orthogonality
and the alignment with a given vector field, see e.g. \cite{Brackbill1993An,Brackbill1982Adaptive,Huang2011Adaptive}.
Solving the Euler-Lagrange equations of \eqref{eq:mesh_func}
\begin{equation}\label{eq:mesh_EL}
  \nabla_{\bm{\xi}}\cdot\left(\bm{G}_k\nabla_{\bm{\xi}}x_k\right)=0,
  ~\bm{\xi}\in\Omega_c,~k=1,2,3,
\end{equation}
 will give directly a coordinate transformation
$\bx=\bx(\bm{\xi})$ from the computational domain $\Omega_c$ to the physical domain $\Omega_p$.

The concentration of the mesh points is controlled by  $\bm{G}_k$,
which in general depends on the solutions or their derivatives of the
underlying governing equations and is one of the most important elements in the adaptive moving mesh method. Different problems may be equipped with different $\bm{G}_k$.
Following  the Winslow variable diffusion method \cite{Winslow1967Numerical}, the simplest choice of $\bm{G}_k$ is
\begin{equation}
  \bm{G}_k=\omega\bm{I}_3,
\end{equation}
where $\omega$ is a positive weight function, called the monitor function.
For example, $\omega$ can be taken as
\begin{align}\label{eq:monitor}
\omega=\sqrt{1+\alpha{\abs{\nabla_{\bm{\xi}}\sigma}}/{\max\abs{\nabla_{\bm{\xi}}\sigma}}},
\end{align}
where $\sigma$ is some physical variable and $\alpha>0$ is a parameter.
There are several other
choices of the monitor functions, see
\cite{Cao1999A,Han2007An,He2012RHD,Tang2006A,Tang2003An}.

\begin{rmk}\rm
  The monitor function is computed from the solutions of the underlying physical equations,
  thus is not smooth in general. To get a smoother (adaptive) mesh, the following low pass filter
  \begin{align*}
  \omega_{i_1,i_2,i_3}\leftarrow&\sum_{j_1,j_2,j_3=0,\pm 1}\left(\dfrac{1}{2}\right)^{\abs{j_1}+\abs{j_2}+\abs{j_3}+3}
  \omega_{i_1+j_1,i_2+j_2,i_3+j_3},
  \end{align*}
  is applied $2\sim 3$ times in this work.
\end{rmk}

The mesh equations \eqref{eq:mesh_EL} are approximated by the central difference scheme
on the computational mesh and then solved by using the Jacobi iteration method
\begin{equation*}
\begin{aligned}
  &\omega_{i_1+1,i_2+\frac12,i_3+\frac12}\left(\bx_{i_1+\frac32,i_2+\frac12,i_3+\frac12}^{[\nu]}-\bx_{i_1+\frac12,i_2+\frac12,i_3+\frac12}^{[\nu+1]}\right)\\
 -&\omega_{i_1,i_2+\frac12,i_3+\frac12}\left(\bx_{i_1+\frac12,i_2+\frac12,i_3+\frac12}^{[\nu+1]}-\bx_{i_1-\frac12,i_2+\frac12,i_3+\frac12}^{[\nu]}\right)\\
 +&\omega_{i_1+\frac12,i_2+1,i_3+\frac12}\left(\bx_{i_1+\frac12,i_2+\frac32,i_3+\frac12}^{[\nu]}-\bx_{i_1+\frac12,i_2+\frac12,i_3+\frac12}^{[\nu+1]}\right)\\
 -&\omega_{i_1+\frac12,i_2,i_3+\frac12}\left(\bx_{i_1+\frac12,i_2+\frac12,i_3+\frac12}^{[\nu+1]}-\bx_{i_1+\frac12,i_2-\frac12,i_3+\frac12}^{[\nu]}\right)\\
 +&\omega_{i_1+\frac12,i_2+\frac12,i_3+1}\left(\bx_{i_1+\frac12,i_2+\frac12,i_3+\frac32}^{[\nu]}-\bx_{i_1+\frac12,i_2+\frac12,i_3+\frac12}^{[\nu+1]}\right)\\
 -&\omega_{i_1+\frac12,i_2+\frac12,i_3}\left(\bx_{i_1+\frac12,i_2
 +\frac12,i_3+\frac12}^{[\nu+1]}-\bx_{i_1+\frac12,i_2
 +\frac12,i_3-\frac12}^{[\nu]}\right)=0, \ \nu=0,1,\cdots,\mu,
\end{aligned}
\end{equation*}
in parallel, where ${\bx}^{[0]}_{i_1+\frac12,i_2+\frac12,i_3+\frac12}
  :=\bx^n_{i_1+\frac12,i_2+\frac12,i_3+\frac12}$, and
 the values of $\omega$ are obtained by averaging the values of $\omega$ computed from the solutions $\bU$ at $t^n$, e.g.
\begin{align*}
  \omega_{i_1+1,i_2+\frac12,i_3+\frac12}:=\dfrac14\left(
  \omega_{i_1+1,i_2+1,i_3+1}+\omega_{i_1+1,i_2+1,i_3}+\omega_{i_1+1,i_2,i_3+1}+\omega_{i_1+1,i_2,i_3}
  \right).
\end{align*}
In our numerical tests, the total iteration number $\mu$ is taken as $10$, unless otherwise stated.

Once the  mesh $\{{\bx}^{[\mu]}_{i_1+\frac12,i_2+\frac12,i_3+\frac12}\}$ is obtained,
the final adaptive mesh  is given by
\begin{equation*}
\bx^{n+1}_{i_1+\frac12,i_2+\frac12,i_3+\frac12}
:=\bx^n_{i_1+\frac12,i_2+\frac12,i_3+\frac12}
+  
{\Delta_\tau}   (\delta_\tau{\bx})^{n}_{i_1+\frac12,i_2+\frac12,i_3+\frac12},
%
\end{equation*}
%
%
where
\begin{equation*}
 (\delta_\tau {\bx})^{n}_{i_1+\frac12,i_2+\frac12,i_3+\frac12}:=
   {\bx}^{[\mu]}_{i_1+\frac12,i_2+\frac12,i_3+\frac12}
  -\bx^n_{i_1+\frac12,i_2+\frac12,i_3+\frac12},
\end{equation*}
and the parameter ${\Delta_\tau}$
 is used  to limit the movement of mesh points
\begin{equation*}
  {\Delta_\tau}\leqslant
  \begin{cases}
  -\frac{1}{2(\delta_\tau{x_1})_{i_1+\frac12,i_2+\frac12,i_3+\frac12}}\left[(x_1)^n_{i_1+\frac12,i_2+\frac12,i_3+\frac12}-(x_1)^n_{i_1-\frac12,i_2+\frac12,i_3+\frac12}\right],
  ~ (\delta_\tau{x_1})_{i_1+\frac12,i_2+\frac12,i_3+\frac12}<0, \\
  \frac{1}{2(\delta_\tau{x_1})_{i_1+\frac12,i_2+\frac12,i_3+\frac12}}\left[(x_1)^n_{i_1+\frac32,i_2+\frac12,i_3+\frac12}-(x_1)^n_{i_1+\frac12,i_2+\frac12,i_3+\frac12}\right],
  ~ (\delta_\tau{x_1})_{i_1+\frac12,i_2+\frac12,i_3+\frac12}>0, \\
  -\frac{1}{2(\delta_\tau{x_2})_{i_1+\frac12,i_2+\frac12,i_3+\frac12}}\left[(x_2)^n_{i_1+\frac12,i_2+\frac12,i_3+\frac12}-(x_2)^n_{i_1+\frac12,i_2-\frac12,i_3+\frac12}\right],
  ~ (\delta_\tau{x_2})_{i_1+\frac12,i_2+\frac12,i_3+\frac12}<0, \\
  \frac{1}{2(\delta_\tau{x_2})_{i_1+\frac12,i_2+\frac12,i_3+\frac12}}\left[(x_2)^n_{i_1+\frac12,i_2+\frac32,i_3+\frac12}-(x_2)^n_{i_1+\frac12,i_2+\frac12,i_3+\frac12}\right],
  ~ (\delta_\tau{x_2})_{i_1+\frac12,i_2+\frac12,i_3+\frac12}>0, \\
  -\frac{1}{2(\delta_\tau{x_3})_{i_1+\frac12,i_2+\frac12,i_3+\frac12}}\left[(x_3)^n_{i_1+\frac12,i_2+\frac12,i_3+\frac12}-(x_3)^n_{i_1+\frac12,i_2+\frac12,i_3-\frac12}\right],
  ~ (\delta_\tau{x_3})_{i_1+\frac12,i_2+\frac12,i_3+\frac12}<0, \\
  \frac{1}{2(\delta_\tau{x_3})_{i_1+\frac12,i_2+\frac12,i_3+\frac12}}\left[(x_3)^n_{i_1+\frac12,i_2+\frac12,i_3+\frac32}-(x_3)^n_{i_1+\frac12,i_2+\frac12,i_3+\frac12}\right],
  ~ (\delta_\tau{x_3})_{i_1+\frac12,i_2+\frac12,i_3+\frac12}>0.
  \end{cases}
\end{equation*}
Finally, the mesh velocity at $t=t_n$ in \eqref{eq:VCLCoeff} is defined by
$$
\dot{\bx}^{n}_{i_1+\frac12,i_2+\frac12,i_3+\frac12}:=
{\Delta_\tau}   (\delta_\tau{\bx})^{n}_{i_1+\frac12,i_2+\frac12,i_3+\frac12}/\Delta t_n
$$
where $\Delta t_n$ is the time step size, determined by \eqref{eq:cfl}.

%% file: NumTests.tex
\section{Numerical results}\label{section:Num}
This section conducts several numerical experiments
to validate the performance of our schemes.
 Our schemes are implemented in parallel by utilizing the MPI parts of the PLUTO code \cite{Mignone2007PLUTO}, and all simulations are performed
 with the CPU nodes of the High-performance Computing Platform of Peking University
(Linux redhat environment, two Intel Xeon E5-2697A V4 (16 cores $\times2$) per node, and core frequency of 2.6GHz).
Unless otherwise stated, the adiabatic index $\Gamma$ is taken as $5/3$ and
the time step size $\Delta t_n$ is 
 determined by the usual CFL condition
\begin{equation}\label{eq:cfl}
 {\Delta t}_n\leqslant \dfrac{\text{CFL}}{\sum\limits_{k=1}^3\max\limits_{\bm{i}}
  {\varrho_{k,\bm{i}}^n}/{\Delta \xi_k}},
\end{equation}
where $\varrho_{k,\bm{i}}$ is the spectral radius of the eigen-matrix in the $i_k$-direction of \eqref{eq:semi_U}, and
the CFL number is taken as $0.4$ for 2D cases and $0.3$ for 3D cases.
Moreover, except for a comparison in Example \ref{ex:SyRP}, all numerical results are obtained by using  {\tt VCL1} and SSP second-order RK method (still denoted {\tt VCL1} for the sake of brevity).

\subsection{2D results}

\begin{example}[2D vortex problem]\label{ex:2DVortex}\rm
  This  2D relativistic isentropic vortex problem, being a modification of the vortex problem in \cite{Ling2019Physical}, is used to test the accuracy.  It describes a relativistic vortex moves with a constant speed of magnitude $w$ in $(-1,-1)$ direction. Specially,
  the initial rest-mass density, pressure   and  velocities are given by
  \begin{align*}
  &\rho(x_1,x_2)=(1-C_1 e^{1-r^2})^{\frac{1}{\Gamma-1}},\quad p=\rho^\Gamma,
 \\ 
  &v_1=\dfrac{1}{1-{w(\widetilde{v}_1+\widetilde{v}_2)}/\sqrt{2}}\left[\dfrac{\widetilde{v}_1}{\gamma}-\dfrac{w}{\sqrt{2}}+\dfrac{\gamma
    w^2}{2(\gamma+1)}(\widetilde{v}_1+\widetilde{v}_2)\right],\\
  &v_2=\dfrac{1}{1-{w(\widetilde{v}_1+\widetilde{v}_2)}/\sqrt{2}}\left[\dfrac{\widetilde{v}_2}{\gamma}-\dfrac{w}{\sqrt{2}}+\dfrac{\gamma
    w^2}{2(\gamma+1)}(\widetilde{v}_1+\widetilde{v}_2)\right],
  \end{align*}
  where
  \begin{align*}
  &C_1=\dfrac{(\Gamma-1)/\Gamma}{8\pi^2}\epsilon^2, \ r=\sqrt{\widetilde{x}_1^2+\widetilde{x}_2^2}, \
  \widetilde{x}_1=x_1+\dfrac{\gamma-1}{2}(x_1+x_2)-1,
  \\
  & \quad \widetilde{x}_2=x_2+\dfrac{\gamma-1}{2}(x_1+x_2)-1,\quad \gamma = \dfrac{1}{\sqrt{1-w^2}},
\\
& (\widetilde{v}_1,\widetilde{v}_2)=(-\widetilde{x}_2,\widetilde{x}_1)f, f=\sqrt{\dfrac{C_2}{1+C_2 r^2}},
  \quad C_2=\dfrac{2\Gamma C_1 e^{1-r^2}}{2\Gamma-1-\Gamma C_1 e^{1-r^2}}.
  \end{align*}
\end{example}
The computational domain $\Omega_c$ and the parameters $w$ and $\epsilon$ are  taken as  $[-5,5]\times [-5,5]$ with periodic boundary
conditions, $0.5\sqrt{2}$, and $5$, respectively,
the monitor function is chosen as \eqref{eq:monitor} with $\alpha=20,\sigma=\rho$,
and the number of the Jacobi iterations is $3$.
Table \ref{tab:2DVortex} lists the errors in the rest-mass density $\rho$ and orders of convergence obtained by using our ES moving mesh scheme with $N\times N$ cells.
It can be seen that the adaptive ES  scheme can achieve second-order accuracy.
Figure \ref{fig:2DVortex} plots the adaptive meshes of $N=40$ at different times, which show that
the concentration of the mesh points well follows the propagation of the vortex.
Figure \ref{fig:2DVortexEntropy} presents
 contour of $\rho$ with
40 equally spaced contour lines obtained by the adaptive ES scheme
and  the changes of the total entropy
 $\sum_{i_1,i_2} J_{i_1,i_2} \eta(\bU_{i_1,i_2})\Delta \xi_1\Delta \xi_2$
with respect to time obtained by the adaptive EC and ES schemes with $N=320$.
We can see that the total entropy of the adaptive EC scheme
almost keeps conservative, while the total entropy of the adaptive ES scheme  decays as expected.

\begin{table}[!ht]
  \centering
  \begin{tabular}{r|cc|cc|cc}
    \hline
    $N$ & $\ell^1$ error & order & $\ell^2$ error & order & $\ell^\infty$ error & order  \\ \hline
     20 & 1.371e-02 &  -   & 3.947e-02 &  -   & 2.360e-01 &  -   \\
    40 & 7.458e-03 & 0.88 & 1.999e-02 & 0.98 & 1.250e-01 & 0.92 \\
    80 & 2.385e-03 & 1.64 & 6.934e-03 & 1.53 & 5.217e-02 & 1.26 \\
    160 & 5.561e-04 & 2.10 & 1.817e-03 & 1.93 & 1.766e-02 & 1.56 \\
    320 & 1.251e-04 & 2.15 & 4.449e-04 & 2.03 & 4.723e-03 & 1.90 \\
    \hline
  \end{tabular}
  \caption{Example \ref{ex:2DVortex}: Errors and orders of convergence in $\rho$ at $t=4$.}
  \label{tab:2DVortex}
\end{table}

\begin{figure}[!ht]
  \centering
  \begin{subfigure}[b]{0.3\textwidth}
    \centering
    \includegraphics[width=1.0\textwidth]{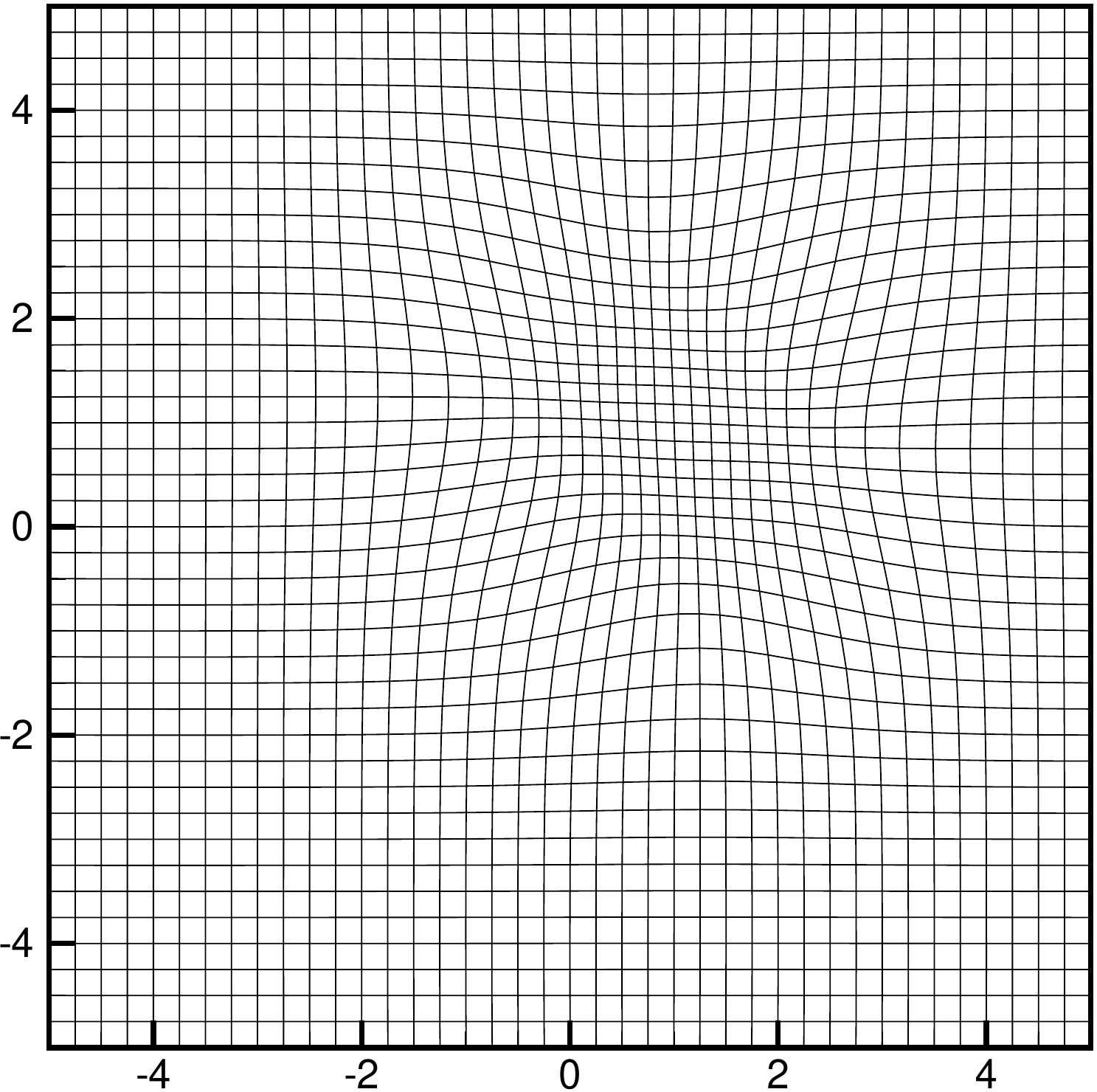}
    \caption{$t=0$}
  \end{subfigure}
  \begin{subfigure}[b]{0.3\textwidth}
    \centering
    \includegraphics[width=1.0\textwidth]{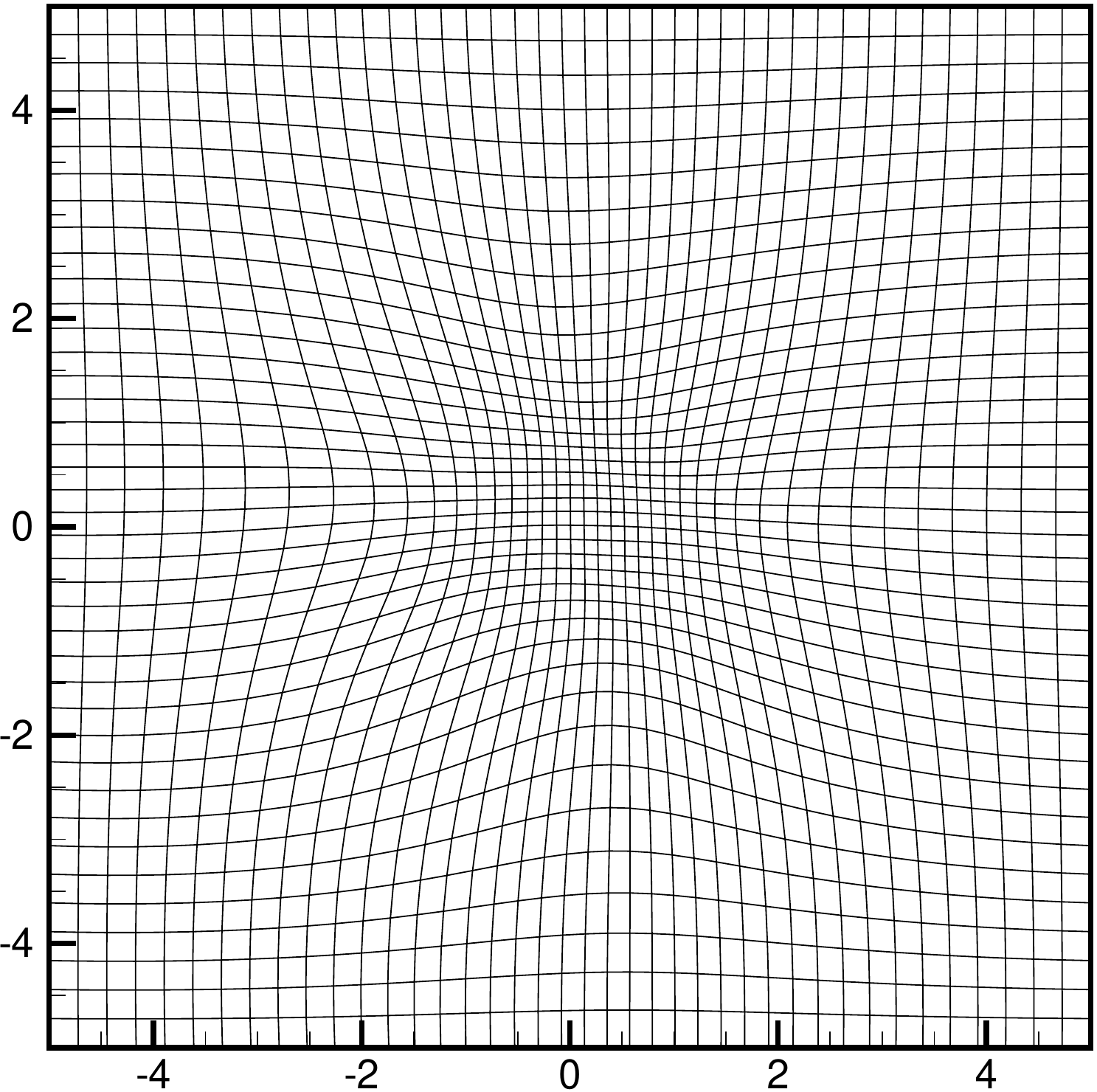}
    \caption{$t=2$}
  \end{subfigure}
  \begin{subfigure}[b]{0.3\textwidth}
    \centering
    \includegraphics[width=1.0\textwidth]{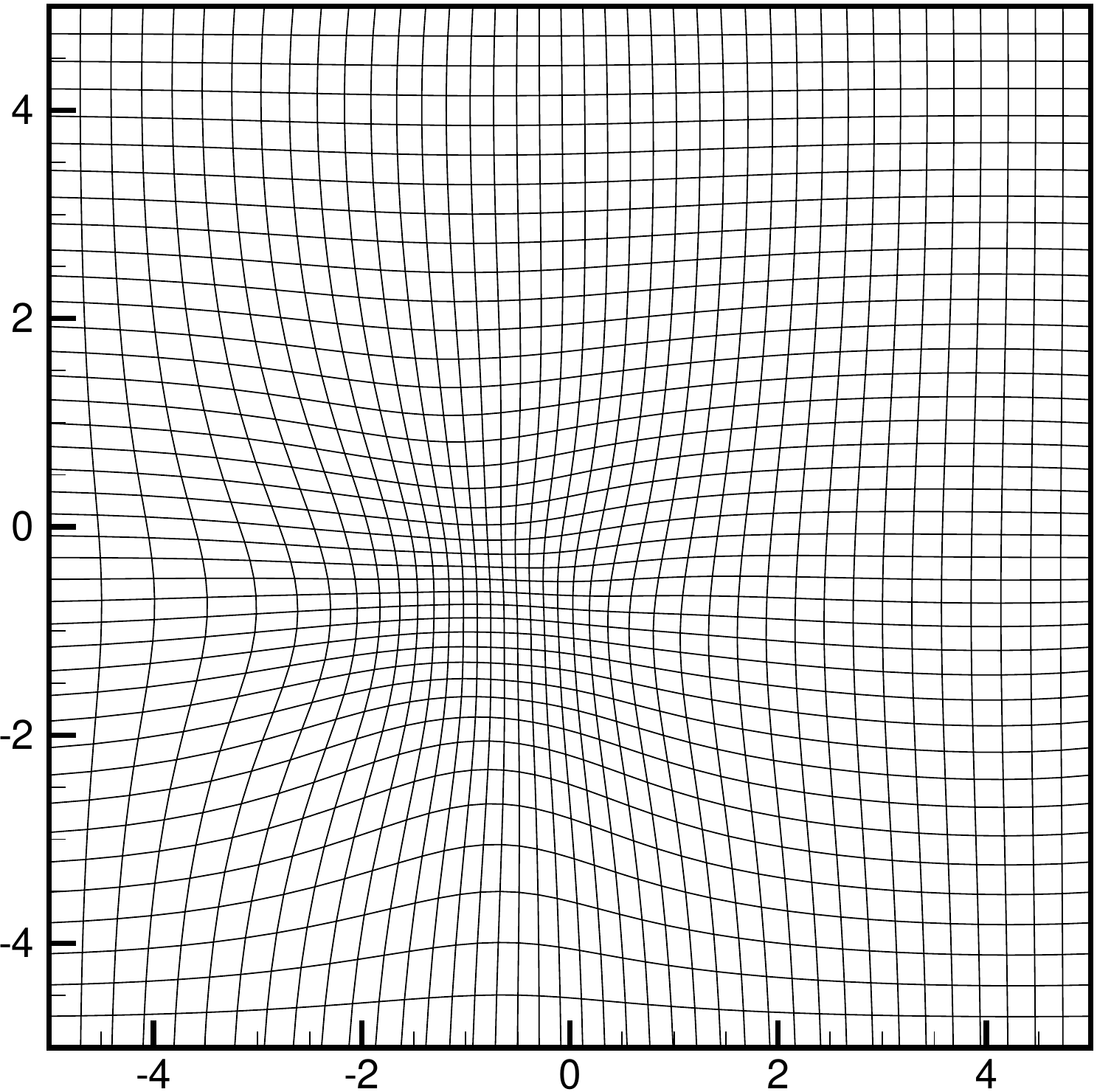}
    \caption{$t=4$}
  \end{subfigure}
  \caption{Example \ref{ex:2DVortex}: The adaptive meshes of $N=40$ at different times.}
  \label{fig:2DVortex}
\end{figure}

\begin{figure}[!ht]
  \centering
  \begin{subfigure}[b]{0.4\textwidth}
    \centering
    \includegraphics[width=1.0\textwidth]{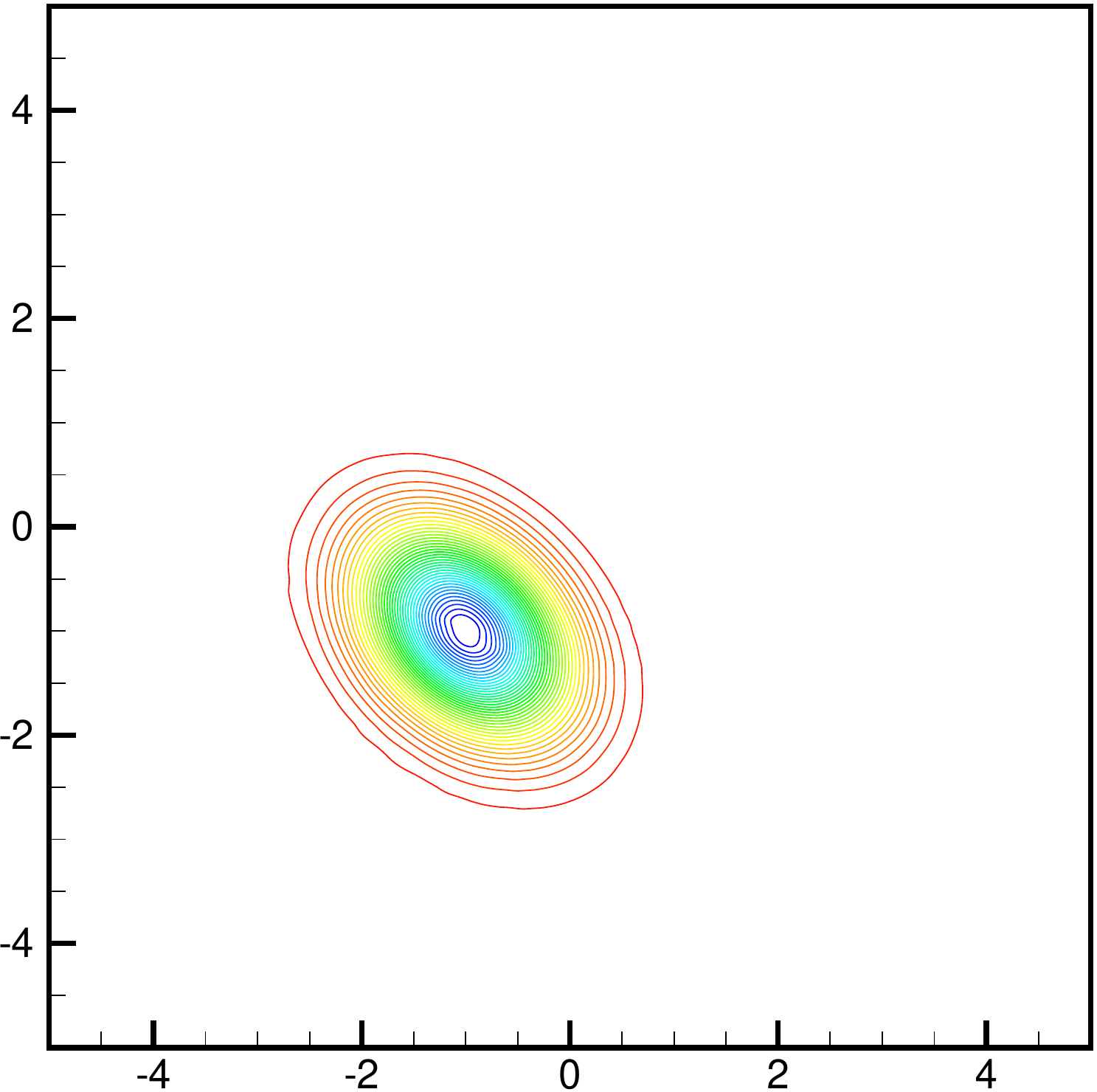}
  \end{subfigure}
  \begin{subfigure}[b]{0.45\textwidth}
    \centering
    \includegraphics[width=1.0\textwidth]{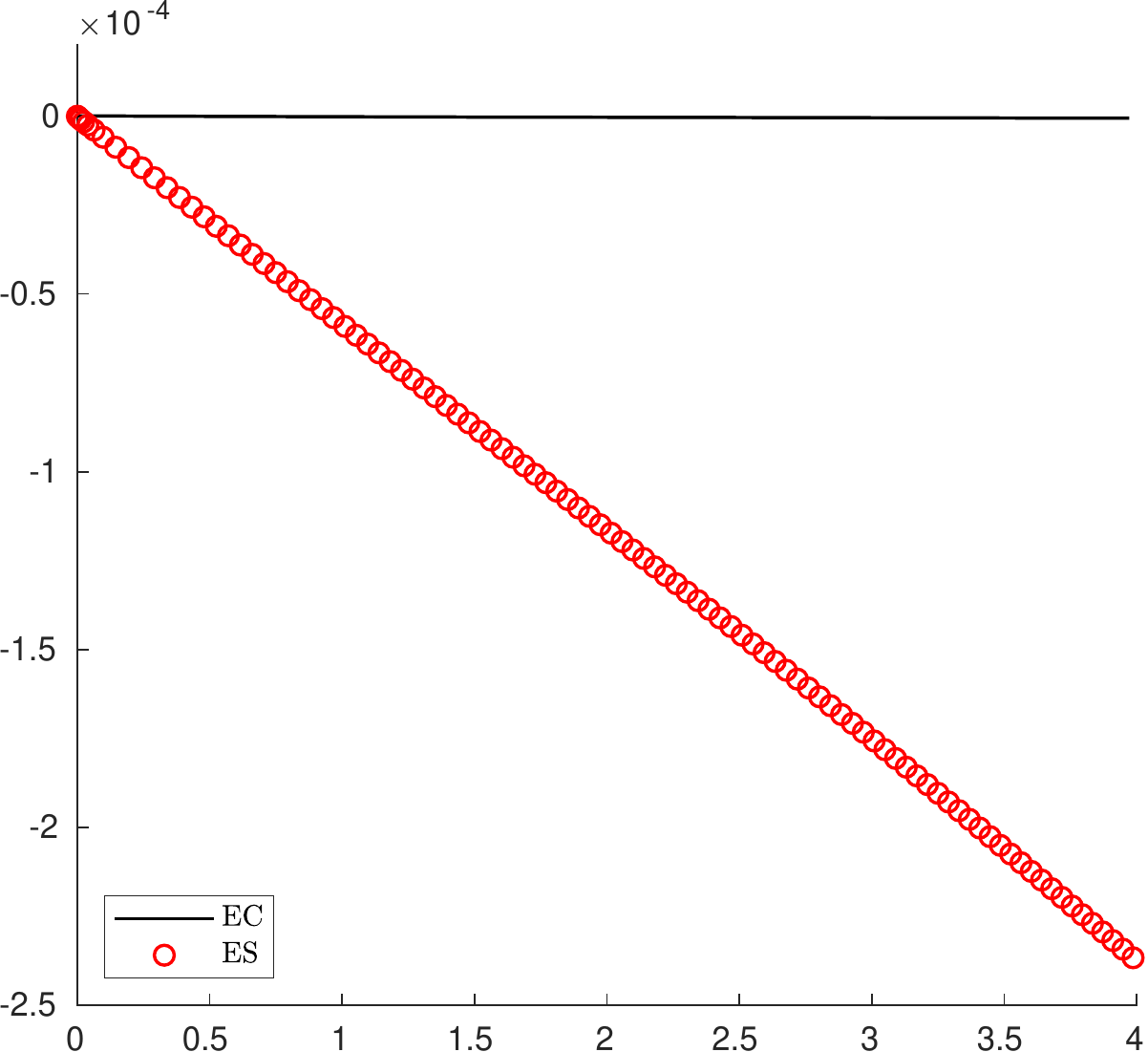}
  \end{subfigure}
  \caption{Example \ref{ex:2DVortex}.
    Left: 
     contour of $\rho$ with
40 equally spaced contour lines;
    right:  change of the total entropy in $t$.  $N=320$.}
  \label{fig:2DVortexEntropy}
\end{figure}

\begin{example}[Riemann problem \uppercase\expandafter{\romannumeral1}]\label{ex:2DRP1}\rm
  The initial data are
  \begin{align*}
    (\rho,v_1,v_2,p)=\begin{cases}
      (0.5,~0.5,-0.5,~5), &\quad x_1>0.5,~x_2>0.5,\\
      (1,~0.5,~0.5,~5),   &\quad x_1<0.5,~x_2>0.5,\\
      (3,-0.5,~0.5,~5),   &\quad x_1<0.5,~x_2<0.5,\\
      (1.5,-0.5,-0.5,~5), &\quad x_1>0.5,~x_2<0.5.
    \end{cases}
  \end{align*}
  It will describe the interaction of four contact discontinuities (vortex sheets)
  with the same sign (the negative sign).
\end{example}

The monitor function is chosen as \eqref{eq:monitor} with $\alpha=1200$ and $\sigma=\ln\rho$.
Figure \ref{fig:RP1} shows the adaptive mesh, the contours of the density logarithms $\ln\rho$
with $40$ equally spaced lines, and the cut of $\ln\rho$  along $x_2=x_1$ at $t=0.4$.
It is seen that the four initial vortex sheets interact each other to
form a spiral with the low rest-mass density around the center of the domain as time
increases, which is the typical cavitation phenomenon in gas dynamics,
and the adaptive mesh points well concentrate near the large gradient area of $\ln\rho$ as expected, which agrees well with the important features.
 Moreover, the solution on the adaptive moving mesh with $200\times 200$ cells is better than that on the uniform mesh with the same cells,
and is comparable to that on the uniform mesh with $600\times 600$ cells according to the cut lines. Those verify the effectiveness of the present adaptive moving mesh strategy.
The CPU times in Table \ref{tab:2DRP_CPU} clearly highlight the efficiency of the adaptive moving mesh scheme,
since it takes only $13.9\%$ CPU time of the uniform mesh with $600\times 600$ cells.

\begin{table}[!ht]
  \centering
  \begin{tabular}{r|ccc}
  	\hline
  	                       & adaptive ($200\times 200$ cells) & uniform ($200\times 200$ cells) & uniform ($600\times 600$ cells) \\ \hline
  	Example \ref{ex:2DRP1} &         1m08s          &          20s           &         7m47s          \\
  	Example \ref{ex:2DRP2} &         1m44s          &          18s           &         7m01s          \\
  	Example \ref{ex:2DRP3} &         2m48s          &          19s           &         7m16s          \\ \hline
  \end{tabular}
  \caption{CPU times of Examples \ref{ex:2DRP1}-\ref{ex:2DRP3}.}
  \label{tab:2DRP_CPU}
\end{table}

\begin{figure}[ht!]
  \begin{subfigure}[b]{0.32\textwidth}
    \centering
    \includegraphics[width=1.0\textwidth]{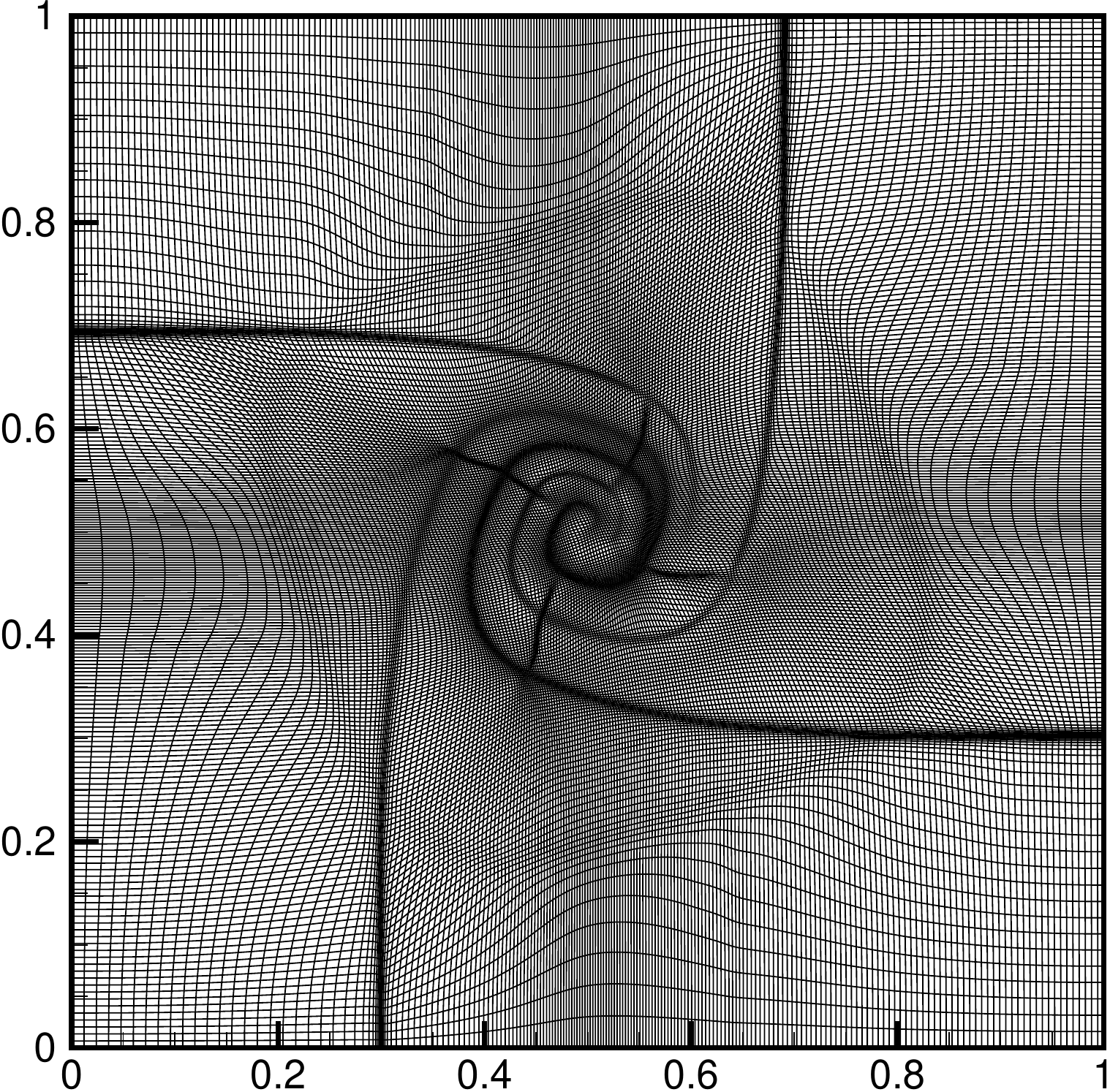}
  \end{subfigure}
  \begin{subfigure}[b]{0.32\textwidth}
    \centering
    \includegraphics[width=1.0\textwidth]{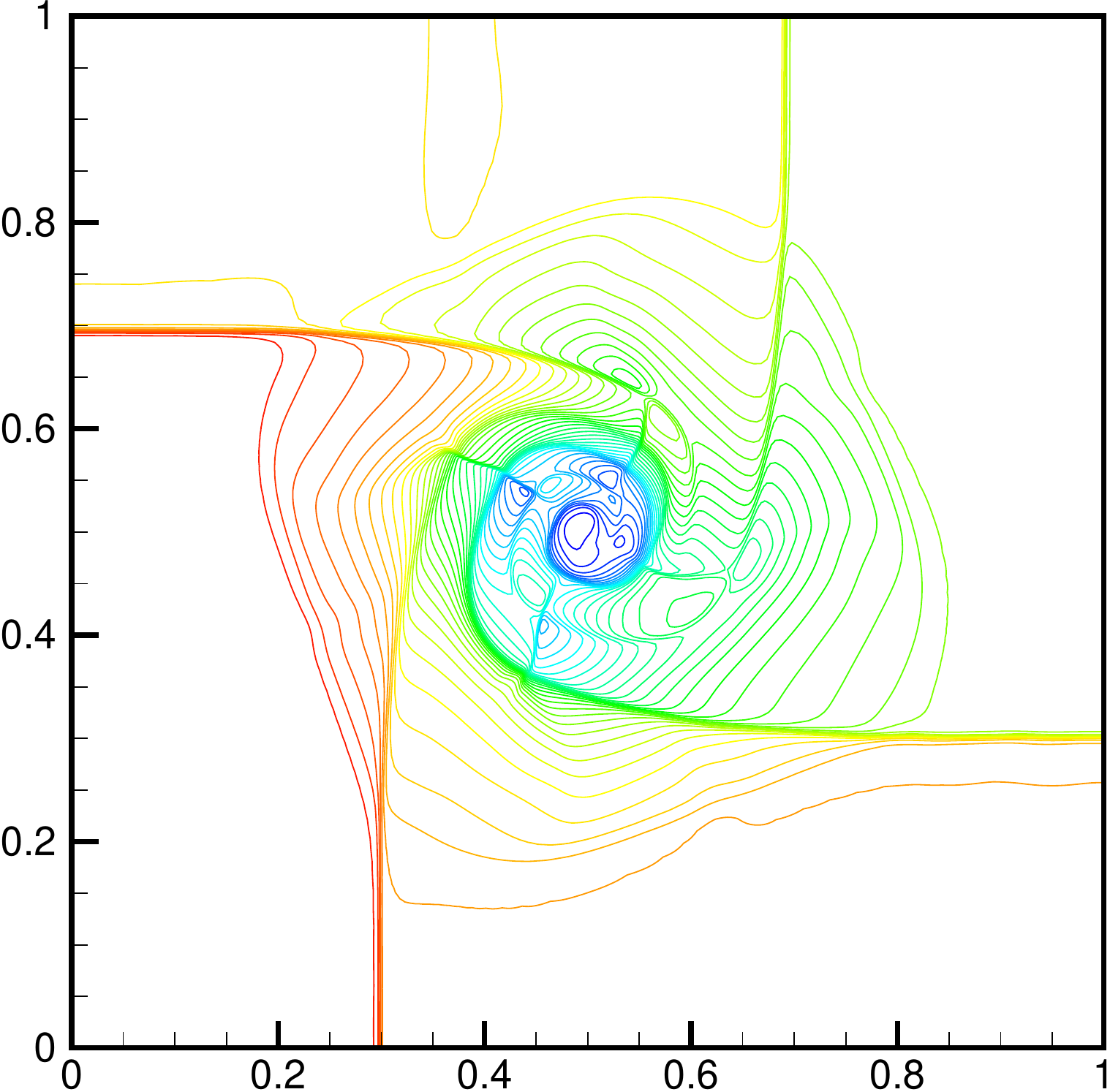}
  \end{subfigure}
\begin{subfigure}[b]{0.34\textwidth}
	\centering
	\includegraphics[width=1.0\textwidth]{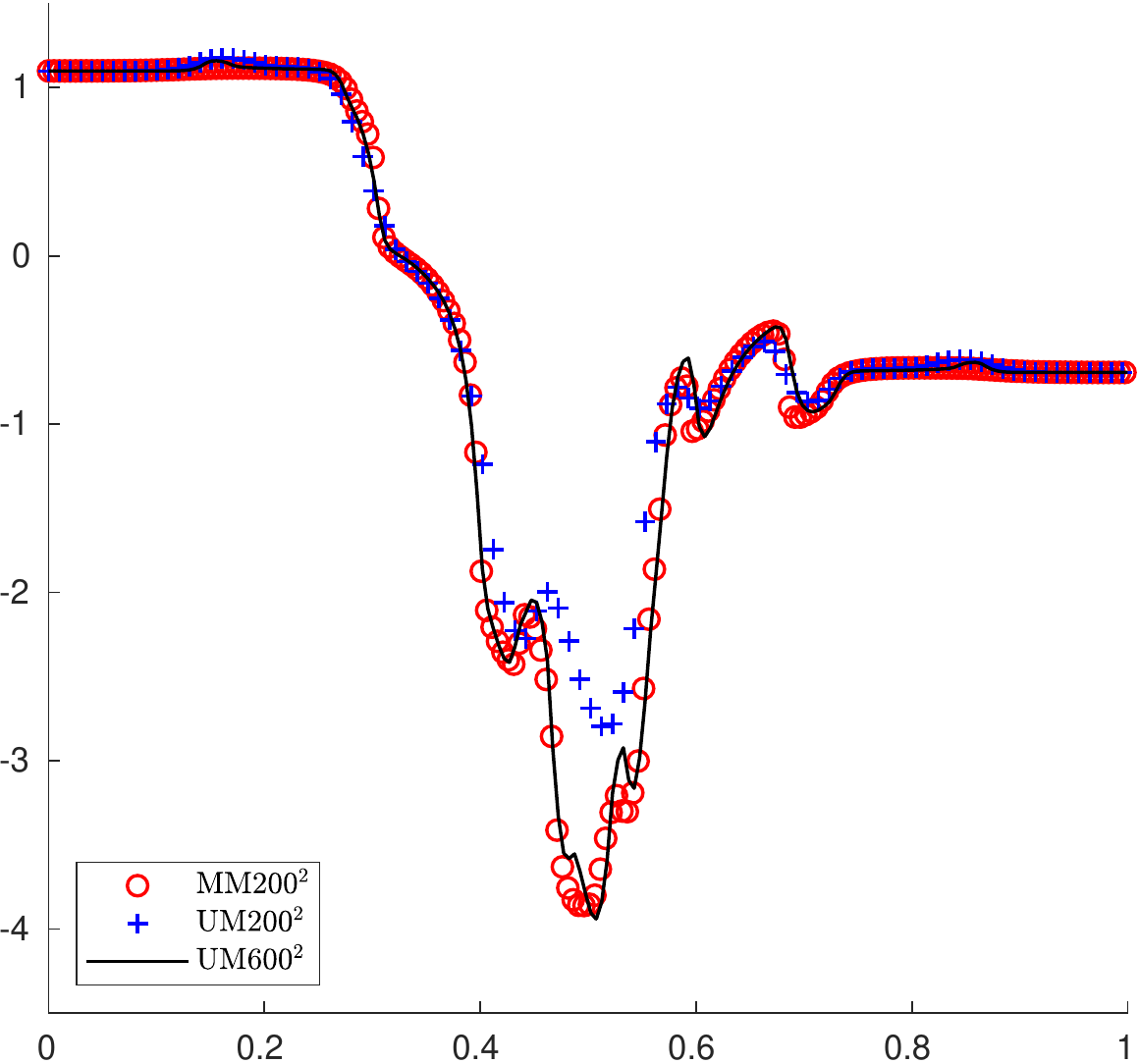}
\end{subfigure}
  \caption{Example \ref{ex:2DRP1}: From left to right: adaptive mesh of $200\times 200$ cells,
   contour of $\ln \rho$ with $40$ equally spaced contour lines, and $\ln \rho$ along $x_2=x_1$ at $t=0.4$.}
  \label{fig:RP1}
\end{figure}

\begin{example}[Riemann problem \uppercase\expandafter{\romannumeral2}]\label{ex:2DRP2}\rm
  The initial data are
  \begin{align*}
    (\rho,v_1,v_2,p)=\begin{cases}
    	(1,~0,~0,~1),             & \quad x_1>0.5,~x_2>0.5, \\
    	(0.5771,-0.3529,~0,~0.4), & \quad x_1<0.5,~x_2>0.5, \\
    	(1,-0.3529,-0.3529,~1),   & \quad x_1<0.5,~x_2<0.5, \\
    	(0.5771,~0,-0.3529,~0.4), & \quad x_1>0.5,~x_2<0.5,
    \end{cases}
  \end{align*}
  which is about the interaction of four rarefaction waves.
\end{example}

The monitor function is the same as that in the last example.
Figure \ref{fig:RP2} shows the adaptive mesh, the contours of the density logarithms $\ln\rho$
with $40$ equally spaced lines, and $\ln\rho$  along $x_2=x_1$ at $t=0.4$.
The CPU times are listed in Table \ref{tab:2DRP_CPU}. 
The results show that those four initial discontinuities first evolve as four rarefaction waves
and then interact each other and form two (almost parallel) curved shock waves perpendicular to the line $x_2=x_1$ as time increases.
It is seen  that the adaptive mesh method effectively captures the important features
such as rarefaction waves and shock waves, and is well comparable to the fixed mesh method with a finer mesh.

\begin{figure}[ht!]
  \begin{subfigure}[b]{0.32\textwidth}
    \centering
    \includegraphics[width=1.0\textwidth]{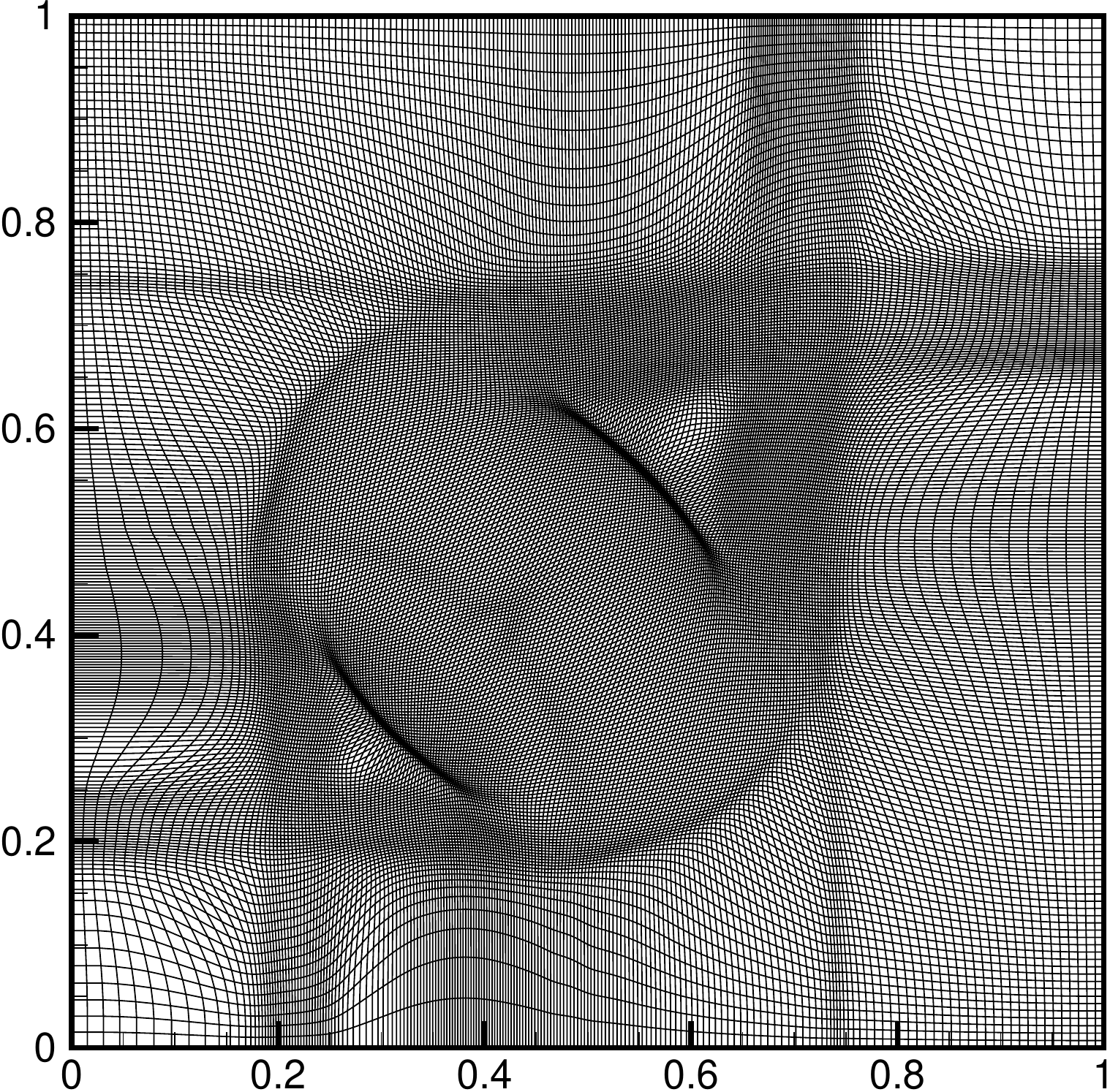}
  \end{subfigure}
  \begin{subfigure}[b]{0.32\textwidth}
    \centering
    \includegraphics[width=1.0\textwidth]{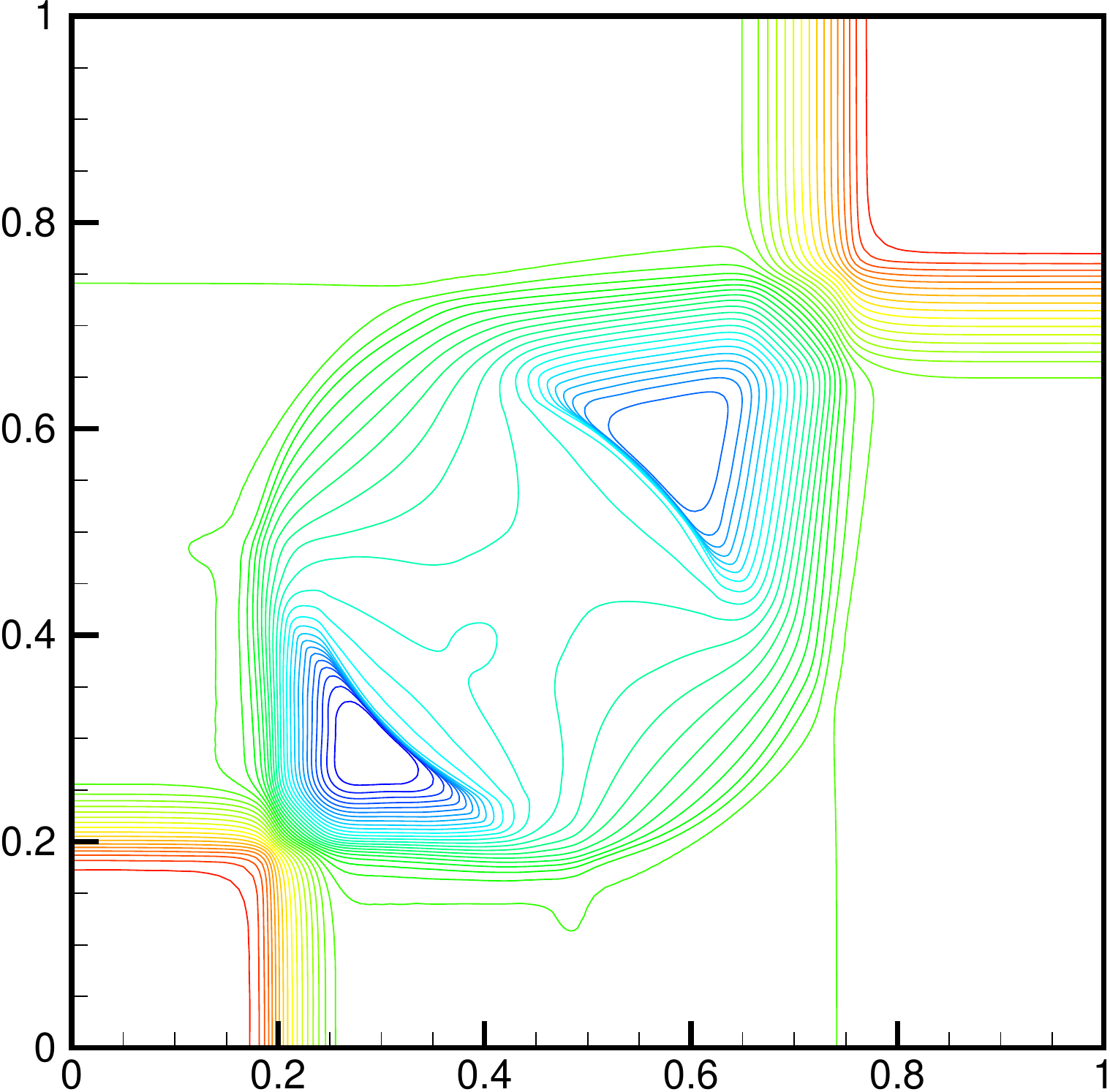}
  \end{subfigure}
\begin{subfigure}[b]{0.36\textwidth}
	\centering
	\includegraphics[width=1.0\textwidth]{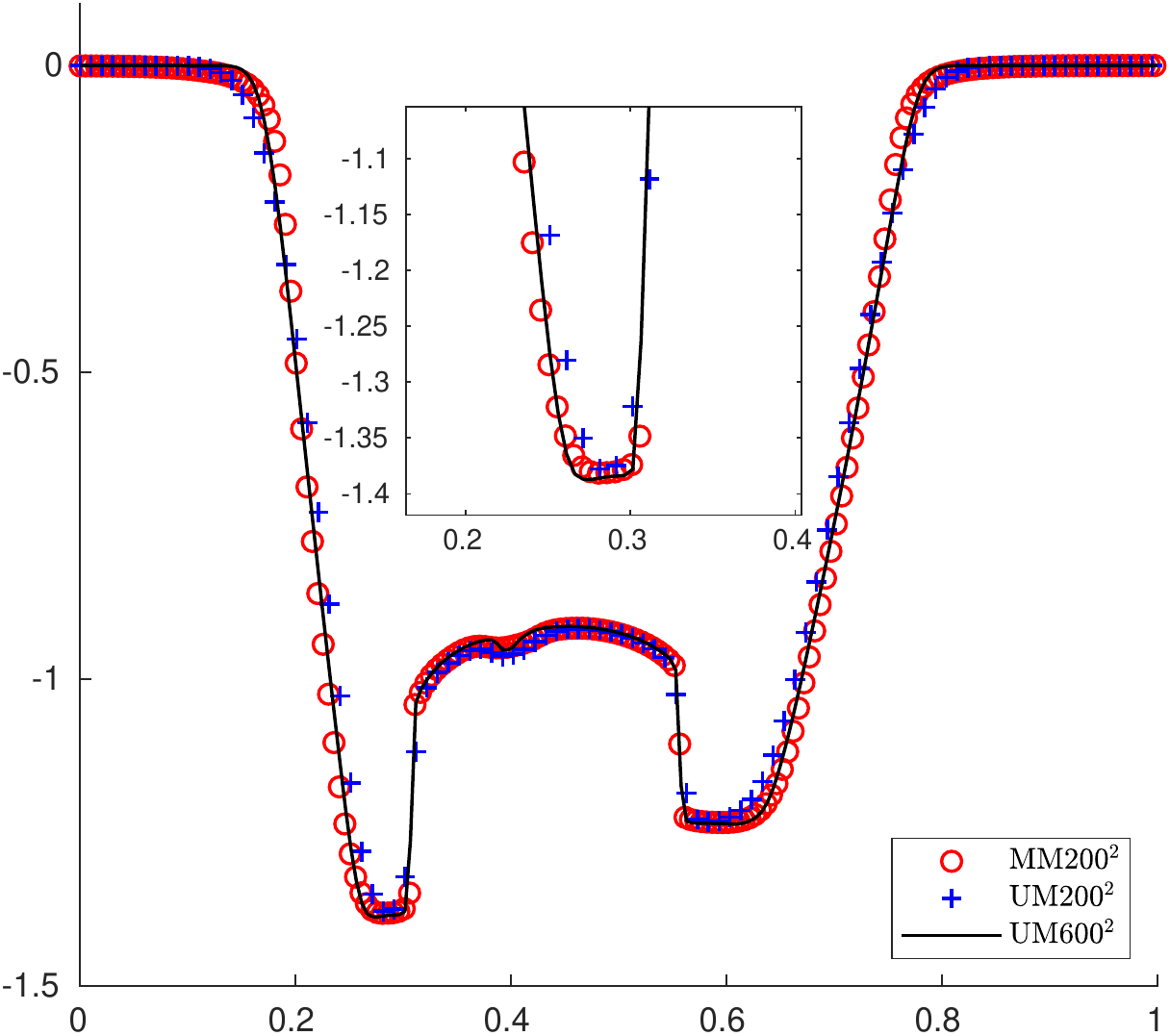}
\end{subfigure}
  \caption{Example \ref{ex:2DRP2}:
  From left to right: adaptive mesh of $200\times 200$ cells,
   contour of $\ln \rho$ with $40$ equally spaced contour lines, and $\ln \rho$ along $x_2=x_1$ at $t=0.4$.
  }
  \label{fig:RP2}
\end{figure}

\begin{example}[Riemann problem \uppercase\expandafter{\romannumeral3}]\label{ex:2DRP3}\rm
  The initial data are
  \begin{align*}
    (\rho,v_1,v_2,p)=\begin{cases}
    	(0.035145216124503,~0,~0,~0.162931056509027), & \quad x_1>0.5,~x_2>0.5, \\
    	(0.1,~0.7,~0,~1),                             & \quad x_1<0.5,~x_2>0.5, \\
    	(0.5,~0,~0,~1),                               & \quad x_1<0.5,~x_2<0.5, \\
    	(0.1,~0,~0.7,~1),                             & \quad x_1>0.5,~x_2<0.5,
    \end{cases}
  \end{align*}
  where the left and bottom discontinuities are two contact discontinuities and
  the top and right are two shock waves.
\end{example}

The monitor function is the same as above.
The adaptive mesh, the contours of the density logarithms $\ln\rho$
with $40$ equally spaced lines, and $\ln\rho$ cut along $x_2=x_1$ at $t=0.4$
are present in Figure \ref{fig:RP3}.
The initial discontinuities interact each other and form a ``mushroom cloud'' around the point $(0.5, 0.5)$,
which is well captured by the adaptive moving mesh method with the chosen monitor function.
Similar to the last two examples, the solution obtained by the adaptive moving mesh with $200\times 200$ cells
is much better than that on the same uniform cells, and agrees well with that
with $600\times 600$ uniform cells, while the adaptive moving mesh scheme only takes $34.6\%$ CPU time, see Table \ref{tab:2DRP_CPU}, showing the high efficiency of the
adaptive scheme.

\begin{figure}[ht!]
  \begin{subfigure}[b]{0.32\textwidth}
    \centering
    \includegraphics[width=1.0\textwidth]{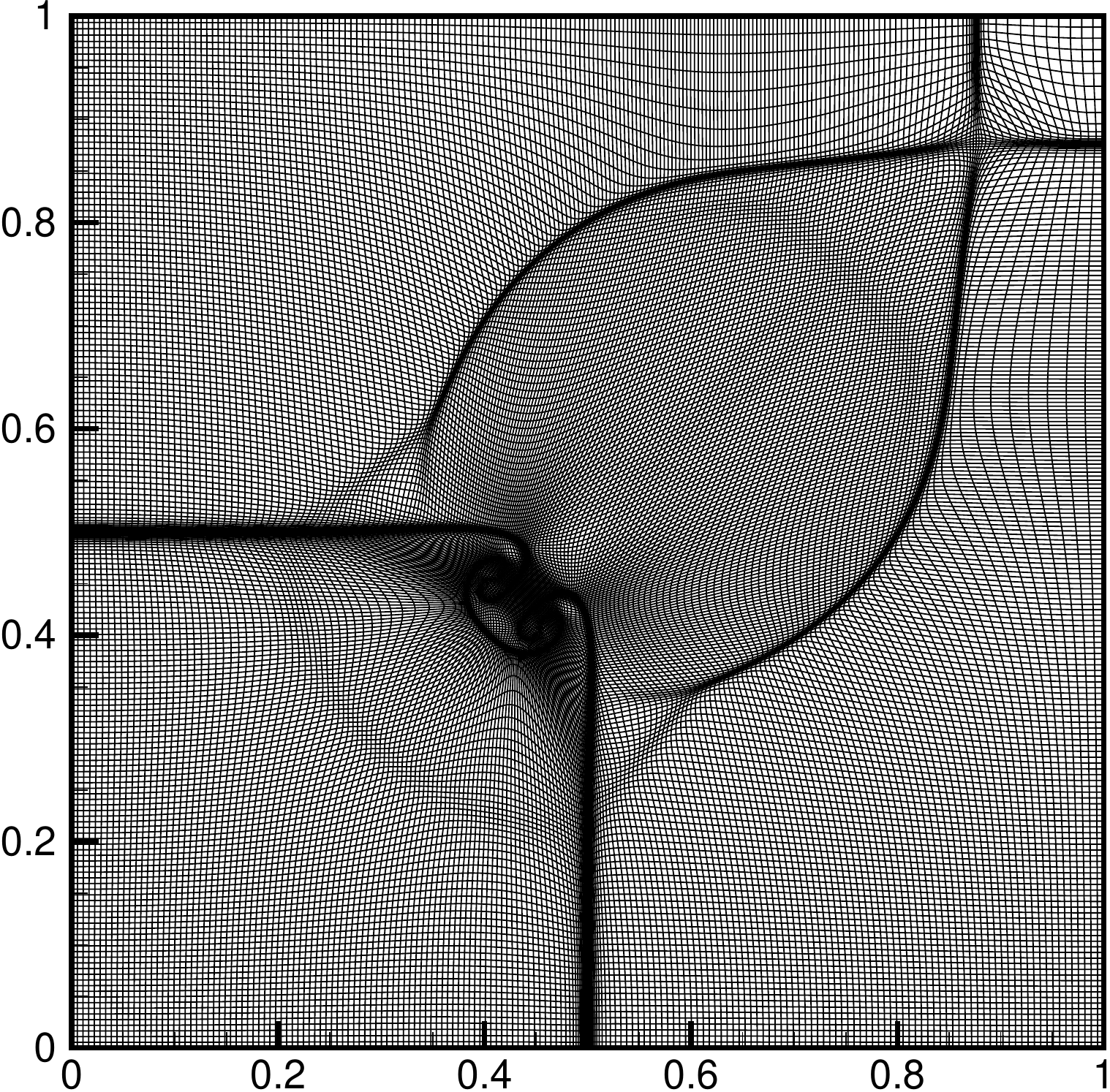}
  \end{subfigure}
  \begin{subfigure}[b]{0.32\textwidth}
    \centering
    \includegraphics[width=1.0\textwidth]{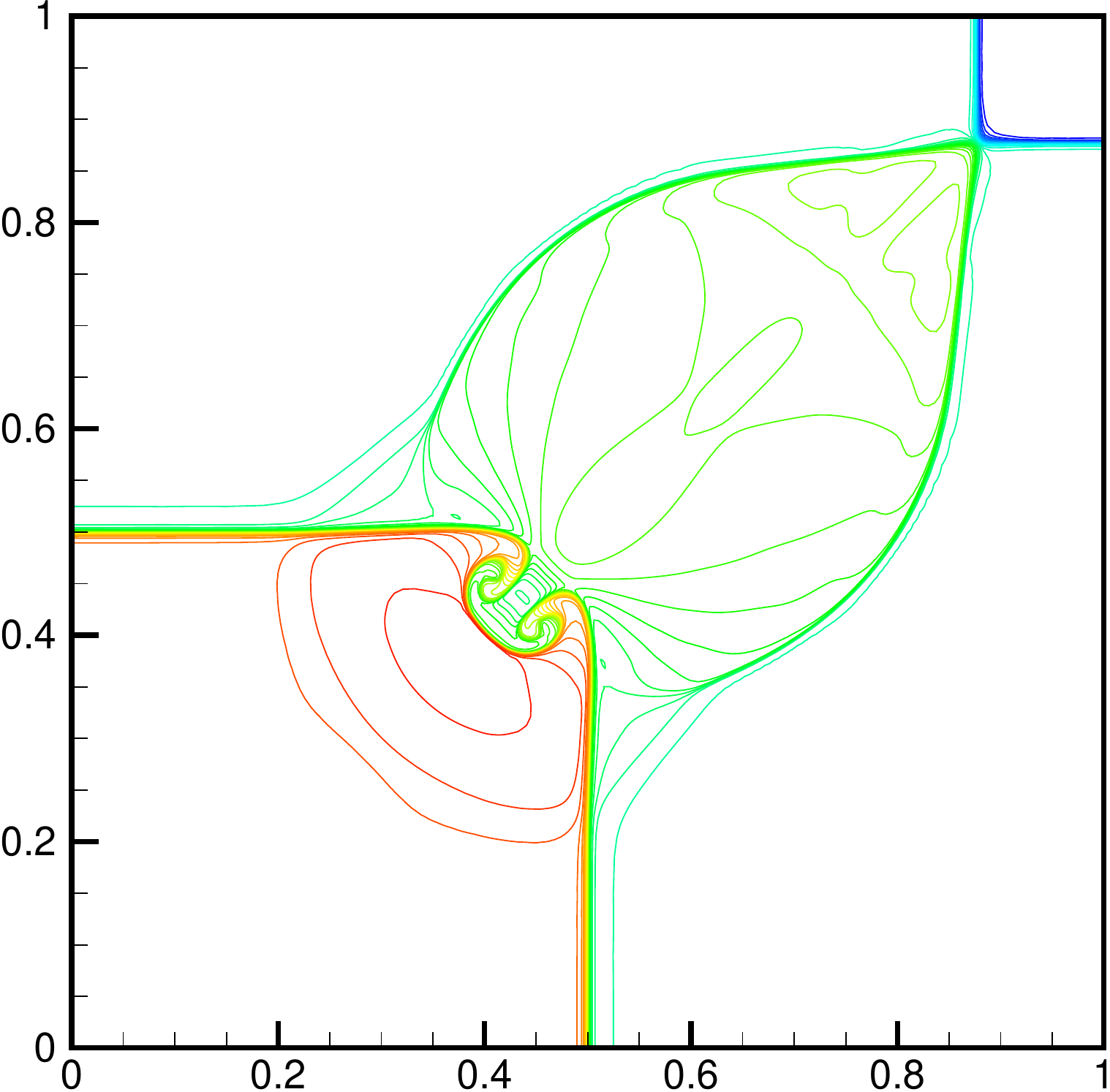}
  \end{subfigure}
\begin{subfigure}[b]{0.36\textwidth}
  \centering
  \includegraphics[width=1.0\textwidth]{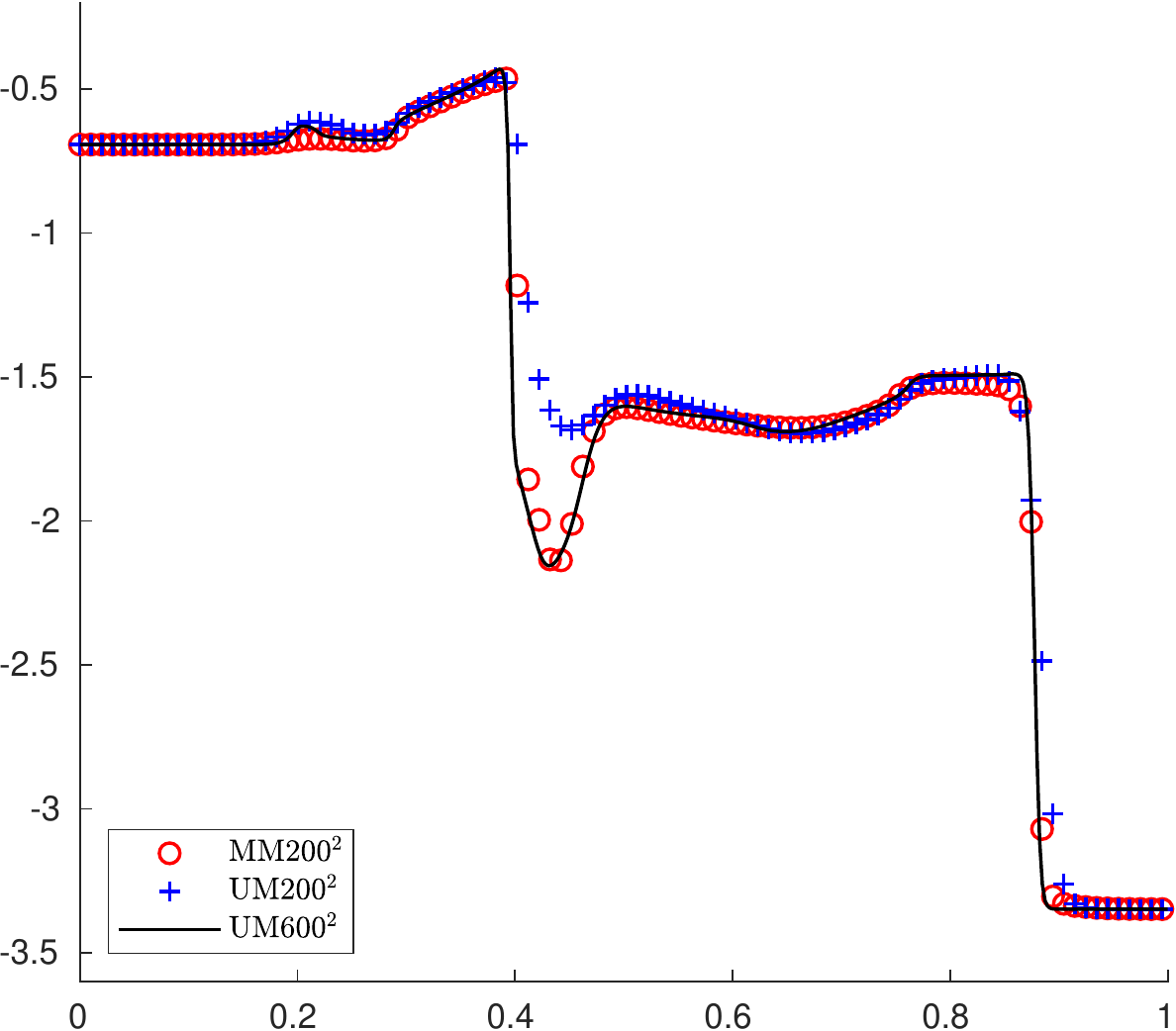}
\end{subfigure}
  \caption{Example \ref{ex:2DRP3}:
   From left to right: adaptive mesh of $200\times 200$ cells,
   contour of $\ln \rho$ with $40$ equally spaced contour lines, and $\ln \rho$ along $x_2=x_1$ at $t=0.4$.
   }
  \label{fig:RP3}
\end{figure}

\subsection{3D results}
\begin{example}[3D smooth sine wave]\label{ex:acc3D}\rm
	This test is used to verify the accuracy of the 3D ES moving mesh scheme.
	The physical domain is a unit cube with periodic boundary conditions,
	and partitioned into $N\times N\times N$ cells.
	The exact solutions are given by
	\begin{align*}
	&(\rho,\vx,\vy,\vz,p)(x_1,x_2,x_3,t)=(1+0.2\sin[2\pi(x_1+x_2+x_3-(\vx+\vy+\vz) t)], ~0.2, ~0.4, ~0.6, ~1).
	\end{align*}
\end{example}

The monitor function and the number of the Jacobi iteration are the same as those in the 2D accuracy test.
Table \ref{tab:acc3D} lists the errors and the orders of convergence in $\rho$ at $t=0.1$.
Figure \ref{fig:acc3D_mesh} displays the adaptive mesh at the final time.
The conclusions are similar to the 2D case.


\begin{table}[!ht]
	\centering
	\begin{tabular}{r|cc|cc|cc}
		\hline
		$N$ & $\ell^1$ error & order & $\ell^2$ error & order & $\ell^\infty$ error &  order \\ \hline
		 20 &   2.085e-02    &   -   &   2.551e-02    &   -   &      4.699e-02      &  -   \\
		 40 &   1.173e-02    & 0.83  &   1.446e-02    & 0.82  &      2.638e-02      & 0.83 \\
		 80 &   4.166e-03    & 1.49  &   6.266e-03    & 1.21  &      1.455e-02      & 0.86 \\
		160 &   1.287e-03    & 1.69  &   2.239e-03    & 1.48  &      6.524e-03      & 1.16 \\
		320 &   3.319e-04    & 1.96  &   5.992e-04    & 1.90  &      2.025e-03      & 1.69 \\ \hline
	\end{tabular}
	\caption{Example \ref{ex:acc3D}: Errors and orders of convergence in $\rho$ at $t=0.1$.}
	\label{tab:acc3D}
\end{table}

\begin{figure}[!ht]
	\centering
	\begin{subfigure}[b]{0.4\textwidth}
		\centering
		\includegraphics[width=1.0\textwidth]{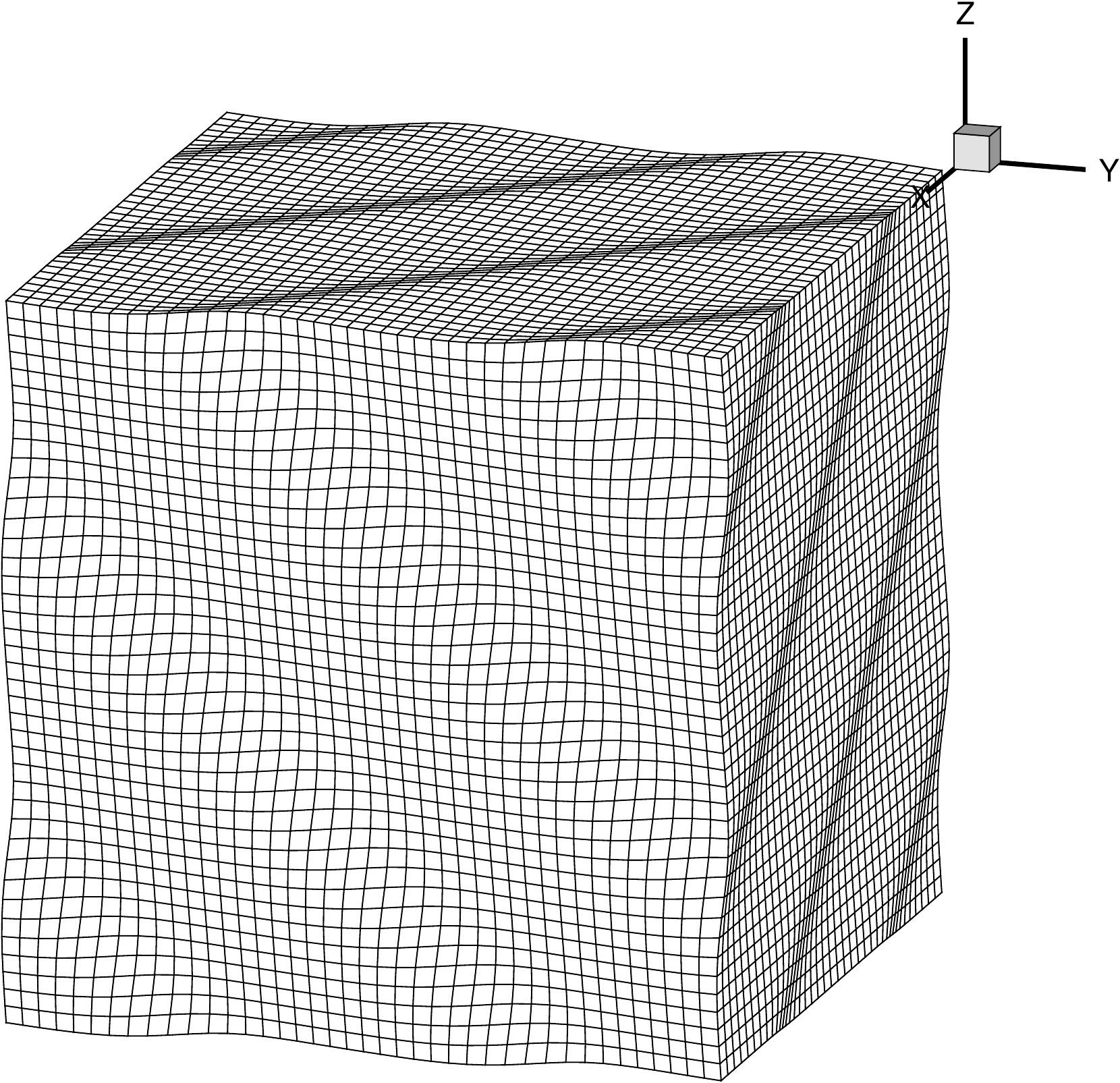}
	\end{subfigure}
	\caption{Example \ref{ex:acc3D}:  Adaptive mesh of $40\times 40\times 40$ cells at $t=0.1$.}
	\label{fig:acc3D_mesh}
\end{figure}

\begin{example}[Spherical symmetric Riemann problem]\label{ex:SyRP}\rm
  To examine the performance of the 3D scheme, we first consider this Riemann problem
  with a reference solution obtained by using a second-order TVD scheme to solve
  the RHD equations in 1D spherical coordinates.
	The initial data are
	\begin{equation*}
	  (\rho,\vx,\vy,\vz,p)=\begin{cases}
	  	(10,~0,~0,~0,40/3),     & ~ r = {\sqrt{x_1^2+x_2^2+x_3^2}} < 0.5, \\
	  	(1, ~0,~0,~0, 10^{-6}), & ~ \text{otherwise}.
	  \end{cases}
	\end{equation*}
\end{example}

The monitor function is chosen as \eqref{eq:monitor} with $\alpha=1000$ and $\sigma=\ln\rho$.
Figures \ref{fig:SyRP_mesh} and \ref{fig:SyRP_cmp} give the adaptive mesh
 and the comparison of the density $\rho$ and the magnitude of velocity $\abs{\bv}$
along the line connecting $(0,0,0)$ and $(1,1,1)$ at $t=0.4$ obtained by
using  $100\times 100\times 100$ or $200\times 200\times 200$ cells, respectively.
It is obvious that all the schemes give correct solutions, while the adaptive moving mesh scheme
gives better results than the uniform mesh even with double mesh cells in each direction,
since the adaptive mesh concentrates near where large gradient in $\ln\rho$ occurs,
increasing the resolution near the discontinuities.
From Table \ref{tab:3D_CPU}, the CPU time of the ES moving mesh scheme is $18.0\%$ of the refined uniform mesh,
showing the high efficiency of the adaptive moving mesh scheme.

 The performances of   {\tt VCL1} 
 and {\tt VCL2} 
 are compared
  in Figure \ref{fig:VCL_cmp}. The left figure shows the evolution of the logarithm of the difference (in the $\ell^1$-norm) between
  the Jacobian $\{J_{\bm{i}}^n\}$ updated by  {\tt VCL1} or {\tt VCL2} and $\{\tilde{J}^n_{\bm{i}}\}$ obtained by the direct discretization of the first equation in \eqref{eq:CMM_VCL}.
  It can be seen that using {\tt VCL1} with the SSP second- and third-order RK methods (abbreviated respectively as {\tt RK2} and {\tt RK3}) gives almost the same error,
  which is larger (about two order of magnitude) than {\tt VCL2} with {\tt RK2}.
  The error obtained by {\tt VCL2} with {\tt RK3} is nearly $10^{-12}$, verifying the analysis in Section \ref{subsec:GCLs}.
  The right figure presents a comparison of the rest-mass density $\rho$, where no obvious difference is observed.
  The CPU time of the adaptive ES scheme with {\tt VCL2} is about $4\%$ (resp.$6\%$) larger than that of the adaptive ES scheme with {\tt VCL1} when {\tt RK2} (resp. {\tt RK3}) is used.
  In view of those, {\tt VCL1} is used in all other examples.

\begin{figure}[!ht]
	\centering
	\includegraphics[width=0.4\textwidth]{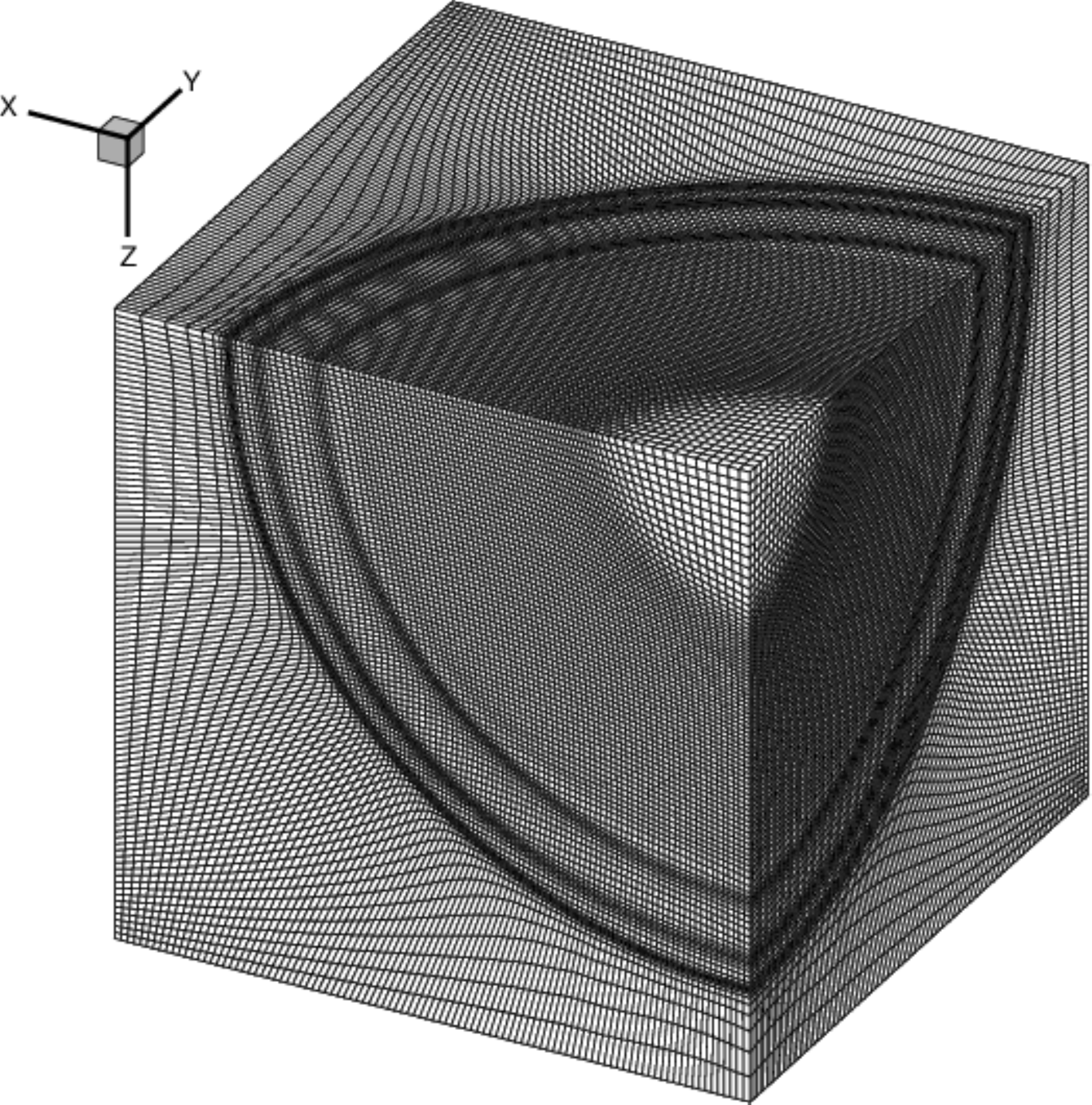}
	\caption{Example \ref{ex:SyRP}:  Adaptive mesh at $t=0.4$. }
	\label{fig:SyRP_mesh}
\end{figure}

\begin{figure}[!ht]
	\centering
	\begin{subfigure}[b]{0.48\textwidth}
		\centering
		\includegraphics[width=1.0\textwidth]{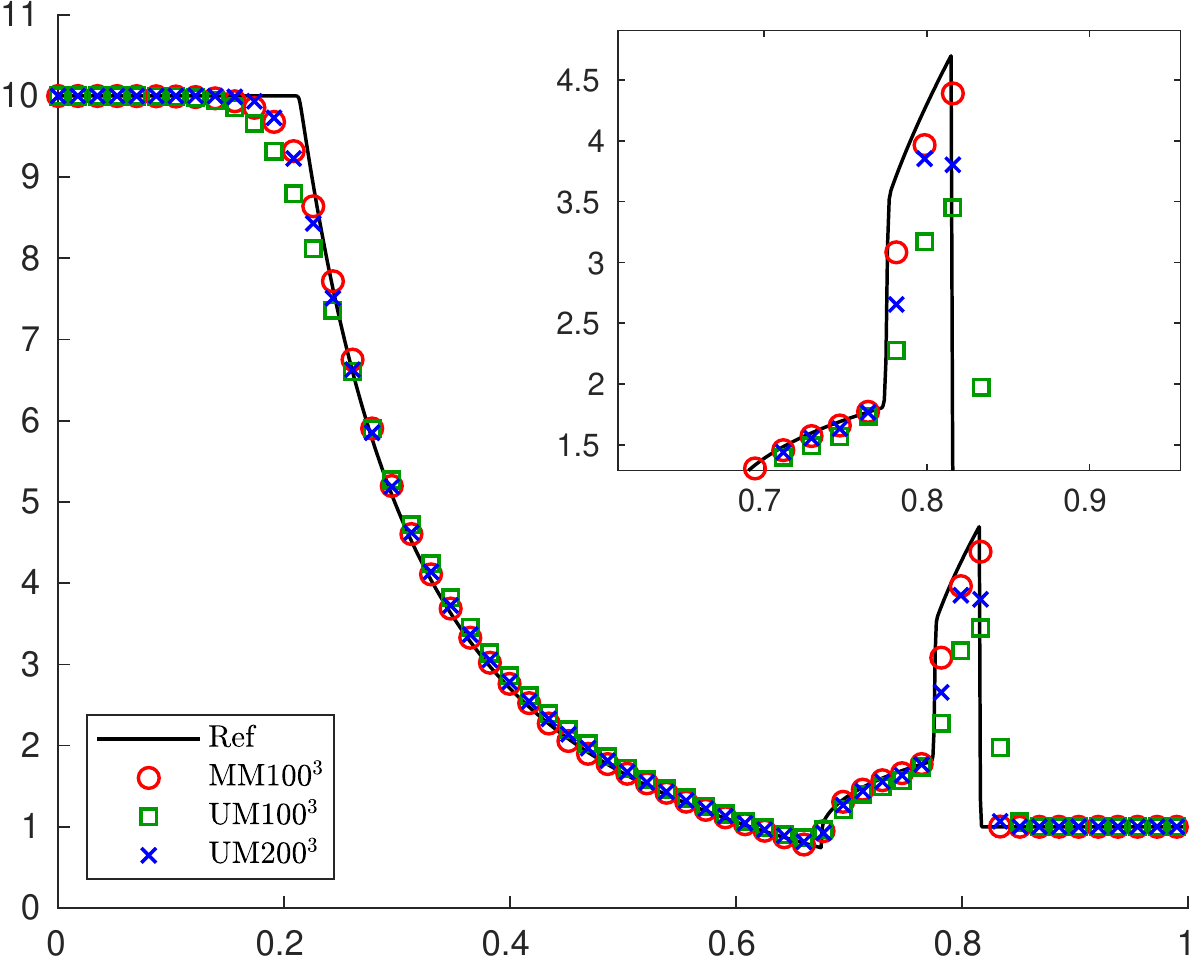}
		\caption{$\rho$}
	\end{subfigure}
	\begin{subfigure}[b]{0.48\textwidth}
		\centering
		\includegraphics[width=1.0\textwidth]{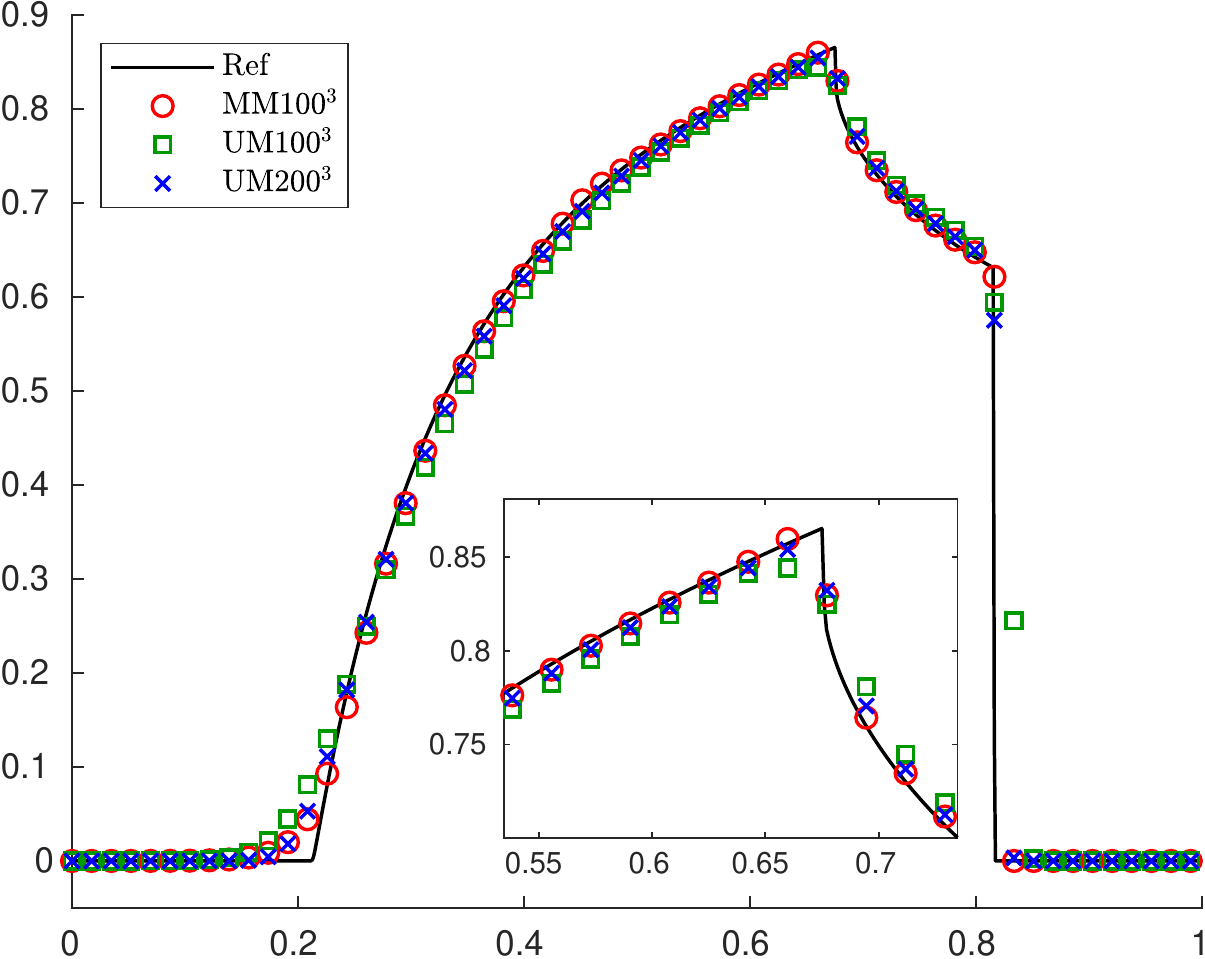}
		\caption{$\abs{\bv}$}
	\end{subfigure}
	\caption{Example \ref{ex:SyRP}: $\rho$ and $\abs{\bv}$
    along the line connecting $(0,0,0)$ and $(1,1,1)$ at $t=0.4$.}
	\label{fig:SyRP_cmp}
\end{figure}

\begin{figure}[!ht]
  \centering
  \begin{subfigure}[b]{0.48\textwidth}
    \centering
    \includegraphics[width=1.0\textwidth]{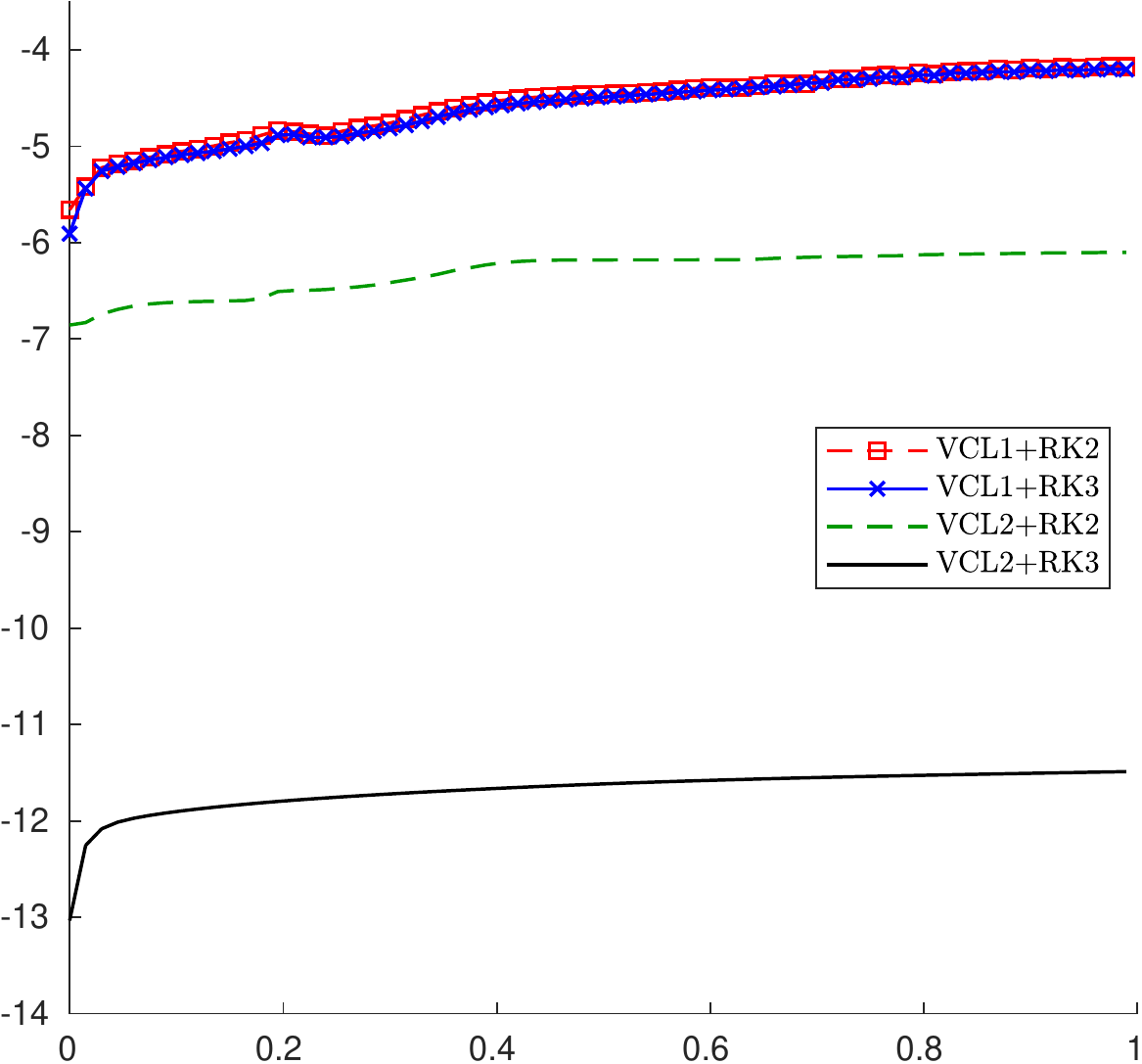}
    \caption{$\log_{10}(|| J^n-\tilde{J}^n||_{\ell^1})$}
  \end{subfigure}
  \begin{subfigure}[b]{0.48\textwidth}
    \centering
    \includegraphics[width=1.0\textwidth]{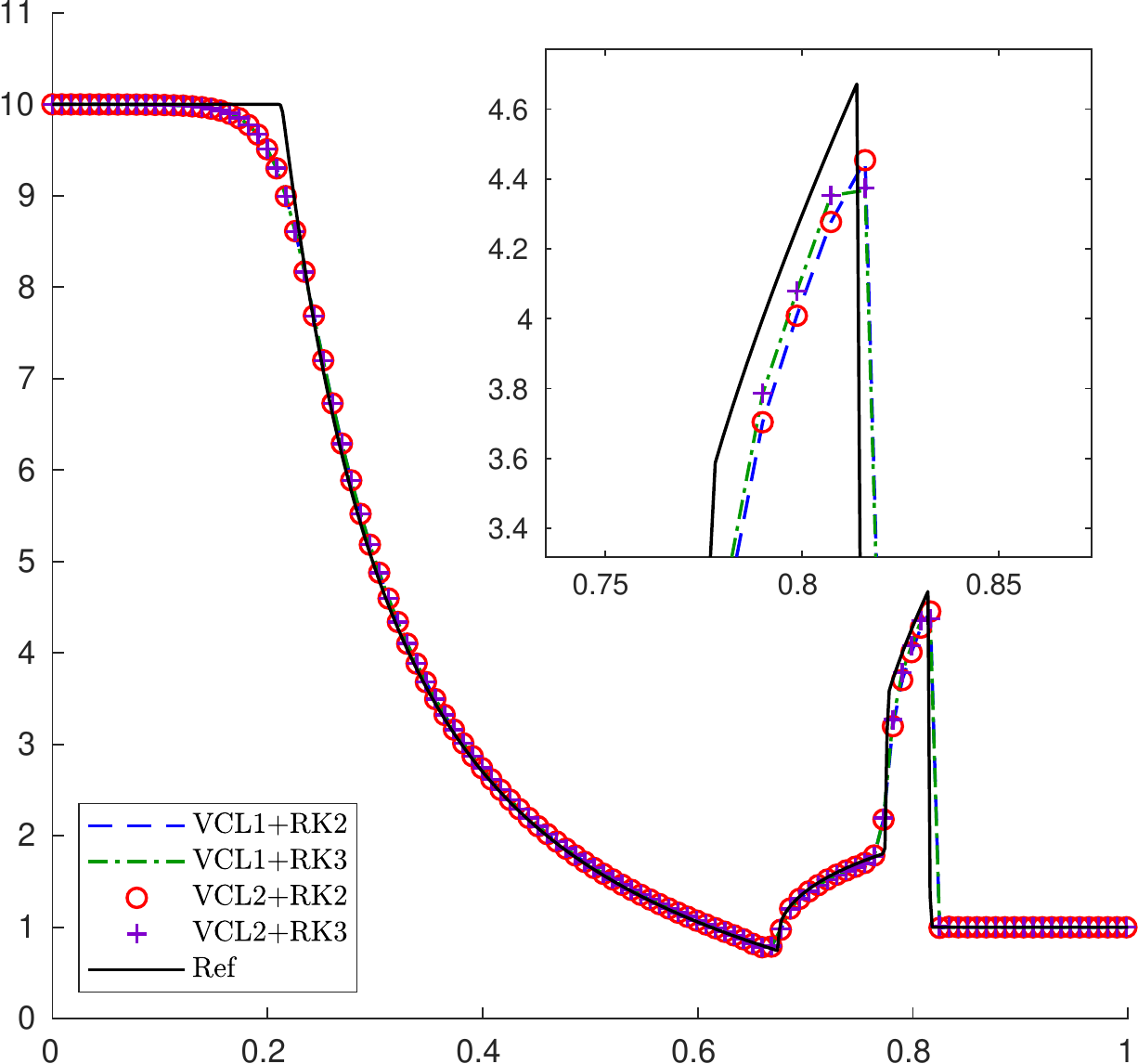}
    \caption{$\rho$}
  \end{subfigure}
  \caption{ A comparison of   {\tt VCL1} and {\tt VCL2} with
  {\tt RK2} or {\tt RK3}. 
  The Jacobian $\{J_{\bm{i}}^n\}$ is updated by  {\tt VCL1} or {\tt VCL2}, while $\{\tilde{J}^n_{\bm{i}}\}$ is obtained by the direct discretization of the first equation in \eqref{eq:CMM_VCL}. }
  \label{fig:VCL_cmp}
\end{figure}

\begin{table}[!ht]
  \centering
  \begin{tabular}{r|ccc}
  	\hline
  	                      & adaptive mesh (cells) & uniform mesh (cells)
  & fine uniform (cells)\\ \hline
  	Example \ref{ex:SyRP} &         2m48s ($100\times 100\times 100$)         &         1m05s ($100\times 100\times 100$)         &         15m32s    ($200\times 200\times 200$)     \\
  	 Example \ref{ex:SBL} &        1h24m20s  ($325\times 90\times 90$)      &         26m07s   ($325\times 90\times 90$)      &        6h27m31s  ($650\times 180\times 180$)      \\ \hline
  \end{tabular}
  \caption{CPU times of Examples \ref{ex:SyRP}-\ref{ex:SBL} (32 cores are used).}
  \label{tab:3D_CPU}
\end{table}

\begin{example}[Shock-bubble interaction problem]\label{ex:SBL}\rm
  This example considers a moving planar shock interacts with a light bubble within the domain
  $[0,325]\times[-45,45]\times[-45,45]$. The detailed setup can be found in \cite{He2012RHD}.
  The initial pre- and post-shock states are
  \begin{equation*}
  (\rho,\vx,\vy,\vz,p)=\begin{cases}
  	(1,~0,~0,~0,~0.05),                                 & x_1<265, \\
  	(1.865225080631180,-0.196781107378299,~0,~0,~0.15), & x_1>265,
  \end{cases}
  \end{equation*}
  and the state of the bubble is $$(\rho,\vx,\vy,\vz,p)=(0.1358,~0,~0,~0,~0.05),\quad
  \sqrt{(x_1-215)^2+x_2^2+x_3^2}\leqslant 25.$$
\end{example}
The monitor is the same as that in the last example.
Figure \ref{fig:SBL_mesh} shows close-up of the adaptive mesh and the 6 iso-surfaces of $\rho$ equally spaced from 0.55 to 1.75, and  two  surface meshes near the bubble at $t=450$.
It is seen that the adaptive mesh points well concentrate near the planar shock and the bubble according to the monitor function.
Figure \ref{fig:SBL_rho} presents the schlieren images on the slice $x_2=0$ of the
rest-mass density $\rho$ at $t=90,180,270,360,450$ (from top to bottom) with
$325\times90\times90$ moving mesh, $325\times90\times90$ uniform mesh and $650\times180\times180$ moving mesh
(from left ro tight), respectively.
Those plots clearly show the dynamics of the interaction between the
shock wave and the bubble, and the sharp interfaces of the bubble at different output times are well captured by
the moving mesh scheme. The ES adaptive moving mesh scheme only takes $21.7\%$ CPU time of the refined uniform mesh
from Table \ref{tab:3D_CPU}, and  gives better results, highlighting its high efficiency.

\begin{figure}[!ht]
	\centering
	  \includegraphics[width=0.45\textwidth]{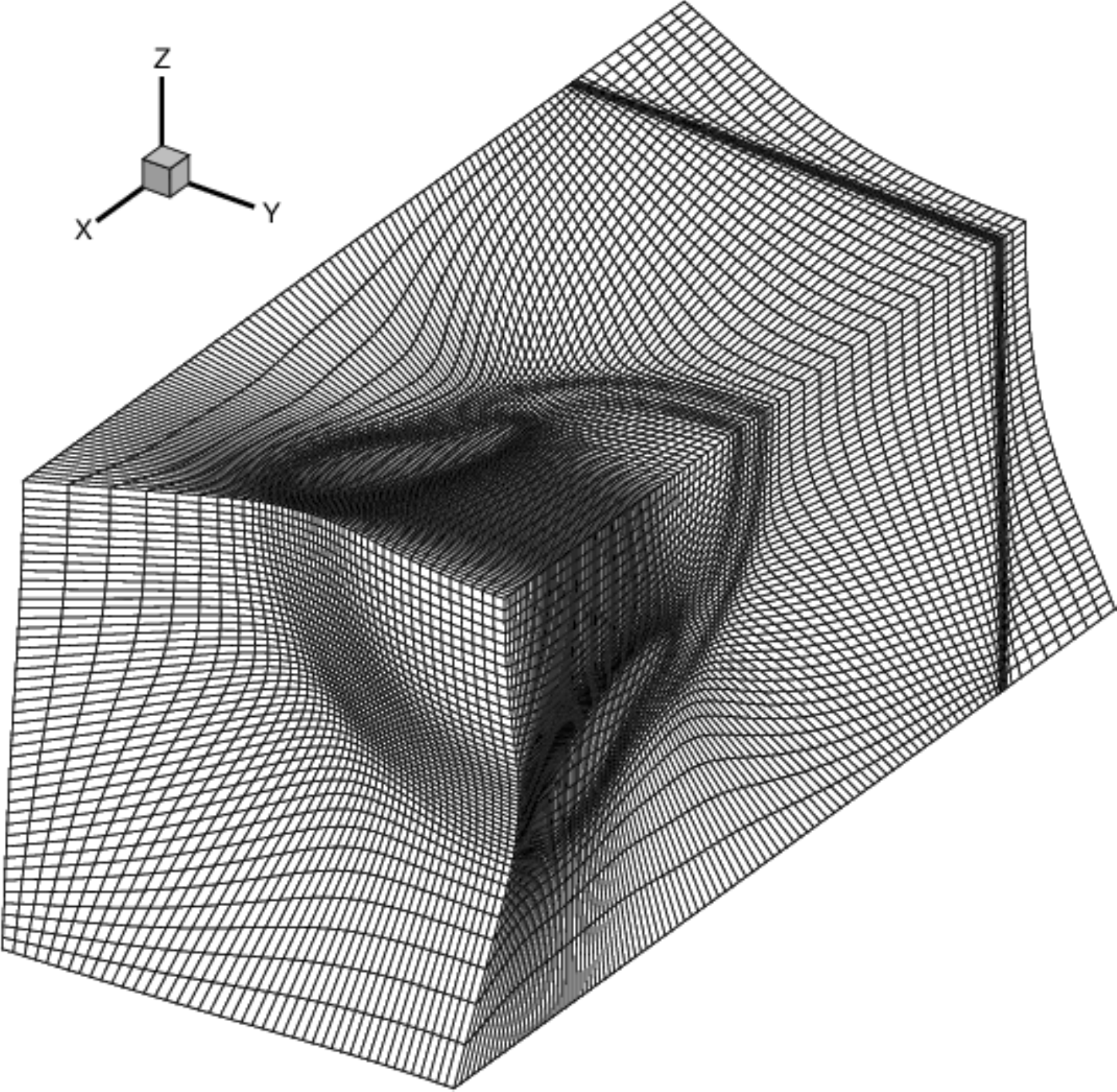}
	  \includegraphics[width=0.45\textwidth]{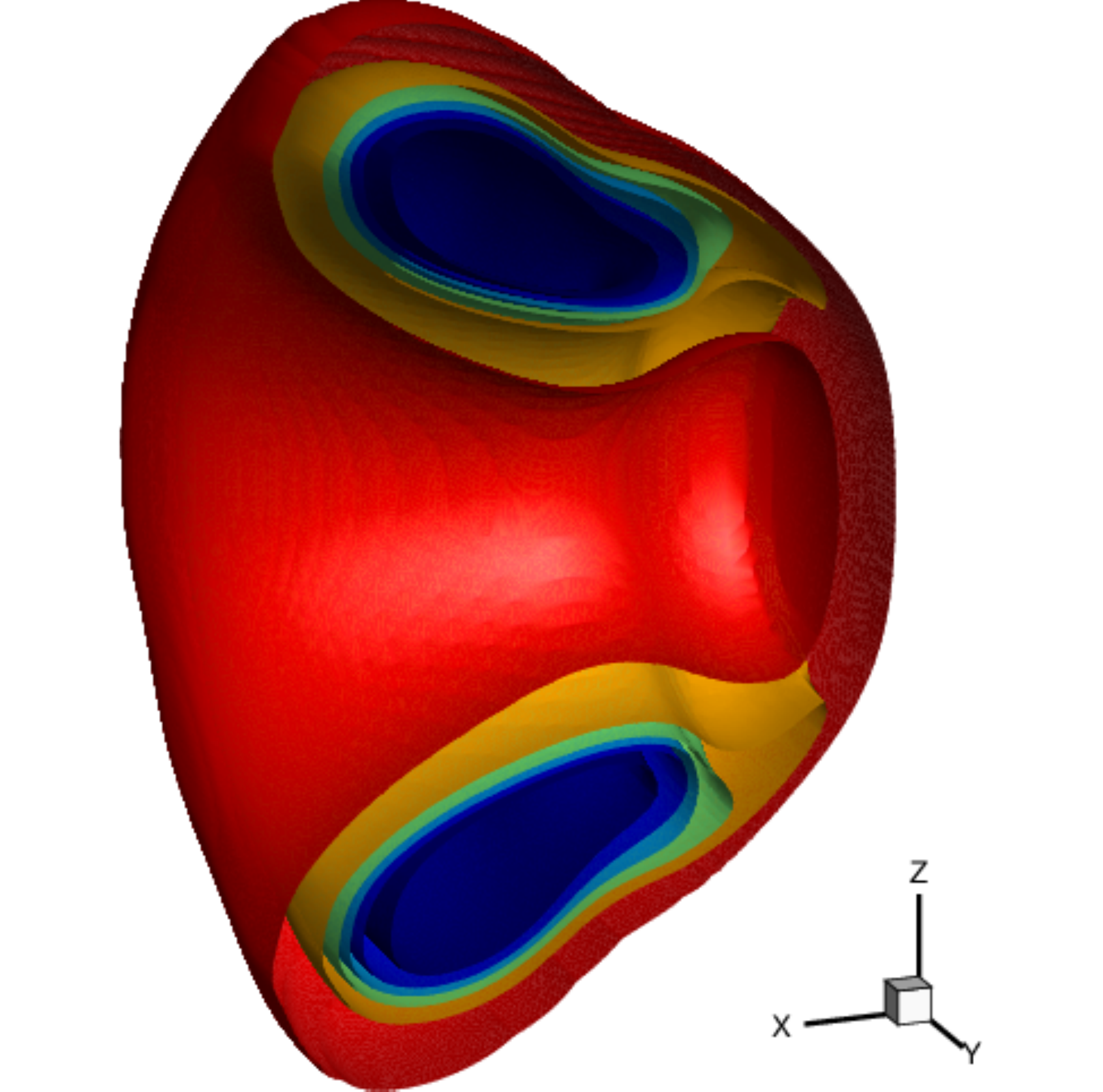}
\\[0.15cm]
	  \includegraphics[width=0.45\textwidth]{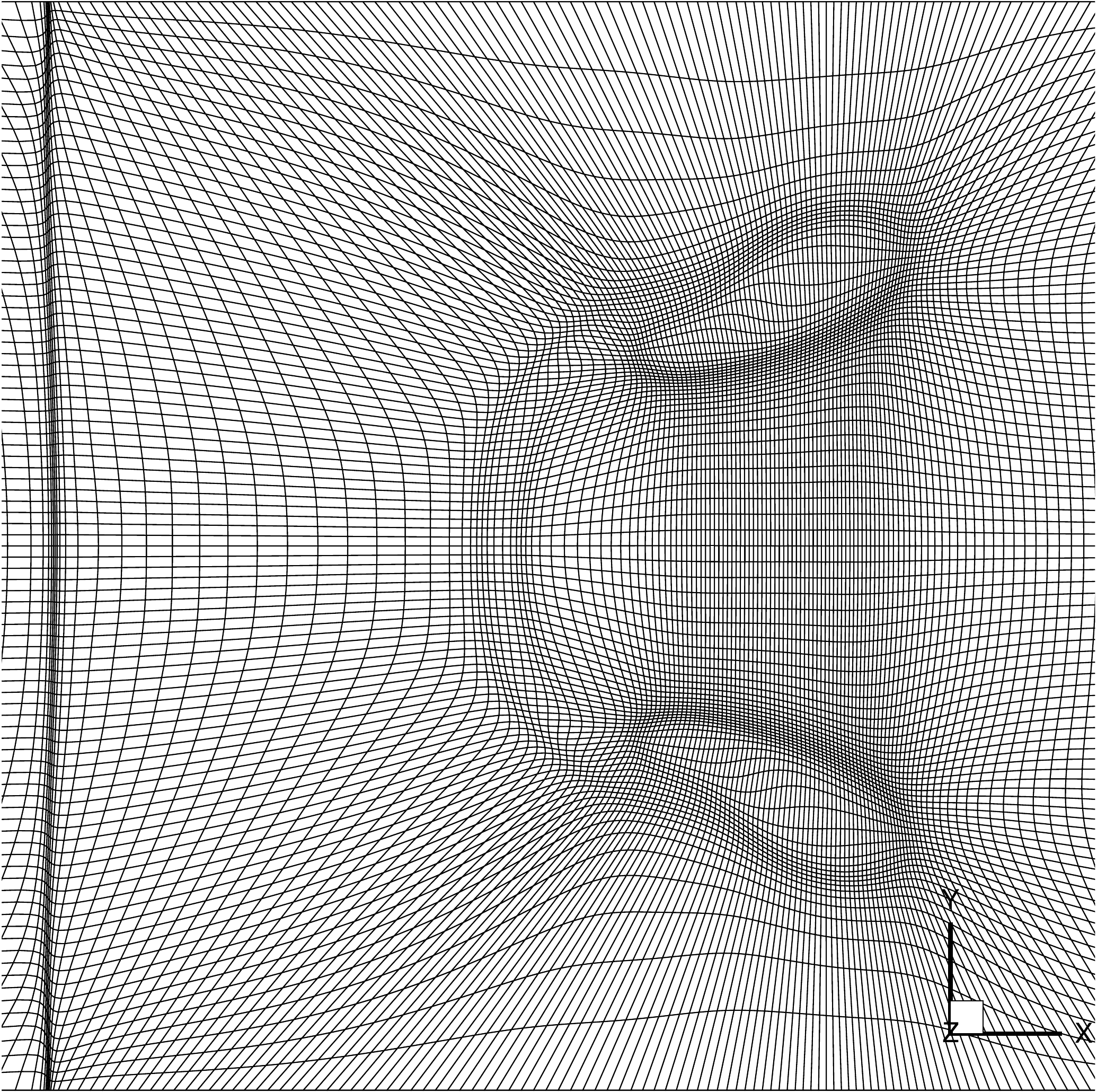}
	  \includegraphics[width=0.45\textwidth]{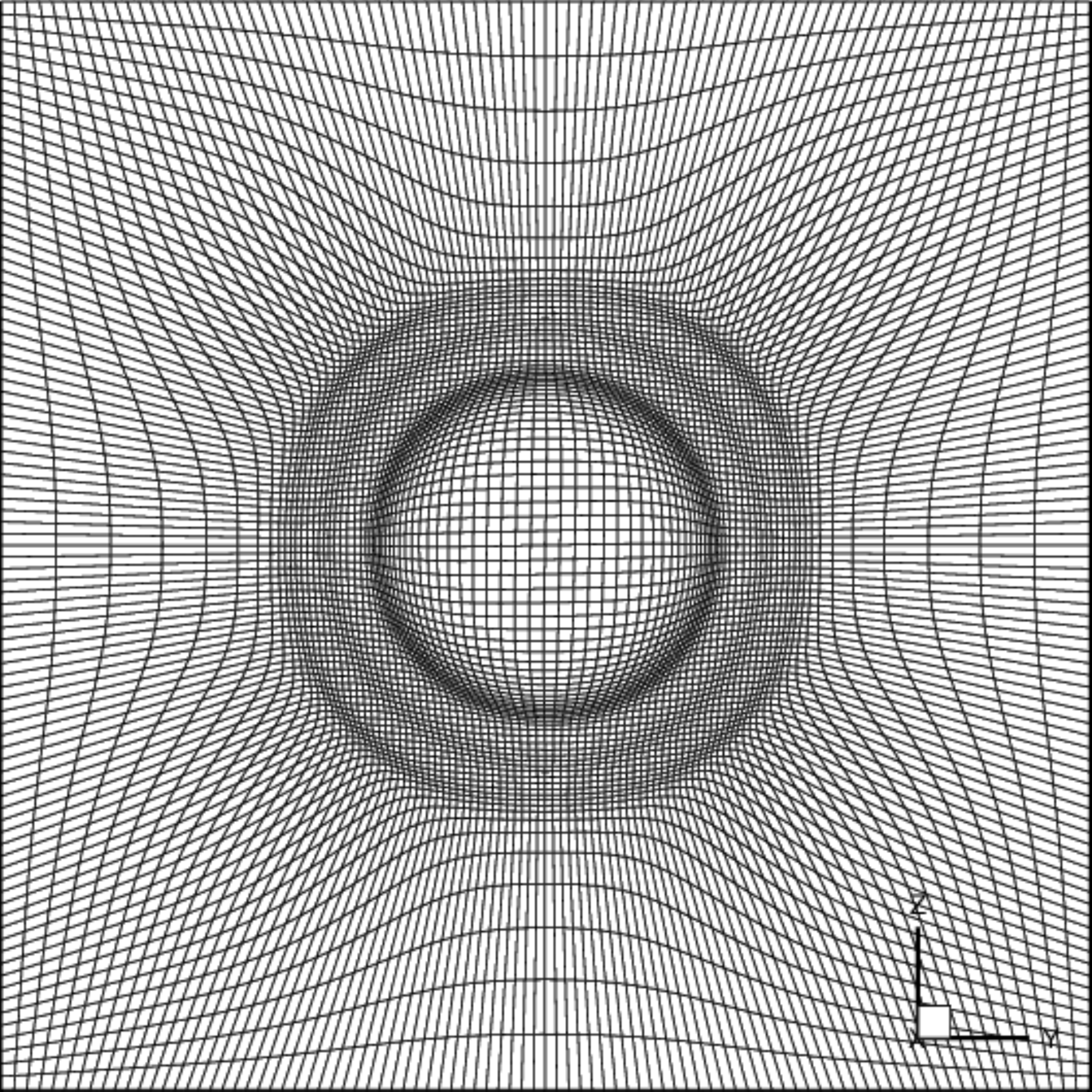}
	\caption{Example \ref{ex:SBL}: Adaptive meshes and $\rho$ at  $t=450$.
		Top left: close-up of the adaptive mesh, $i_1\in[30,140],i_2\in[0,45],i_3\in[0,45]$;
		top right: $6$ iso-surfaces of $\rho$;
		bottom left: the surface mesh with $i_3=i_{3,46+\frac12}$;
		bottom right: the surface mesh on $i_1=i_{1,80+\frac12}$. }
	\label{fig:SBL_mesh}
\end{figure}

\begin{figure}[!ht]
  \centering
\includegraphics[width=0.32\textwidth, trim=1 180 1 180, clip]{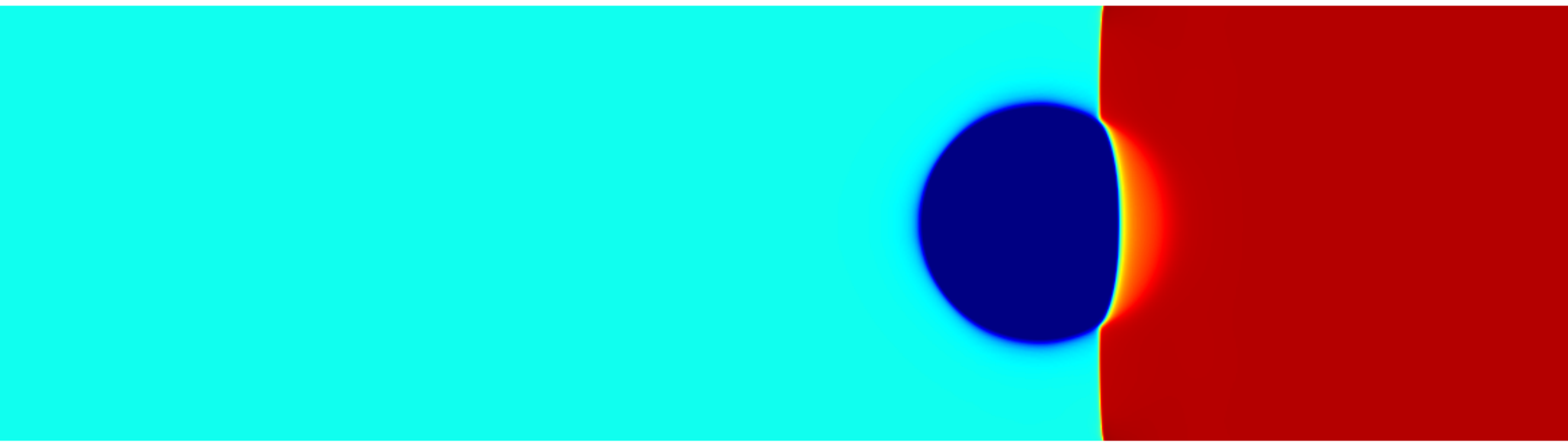}
\includegraphics[width=0.32\textwidth, trim=1 180 1 180, clip]{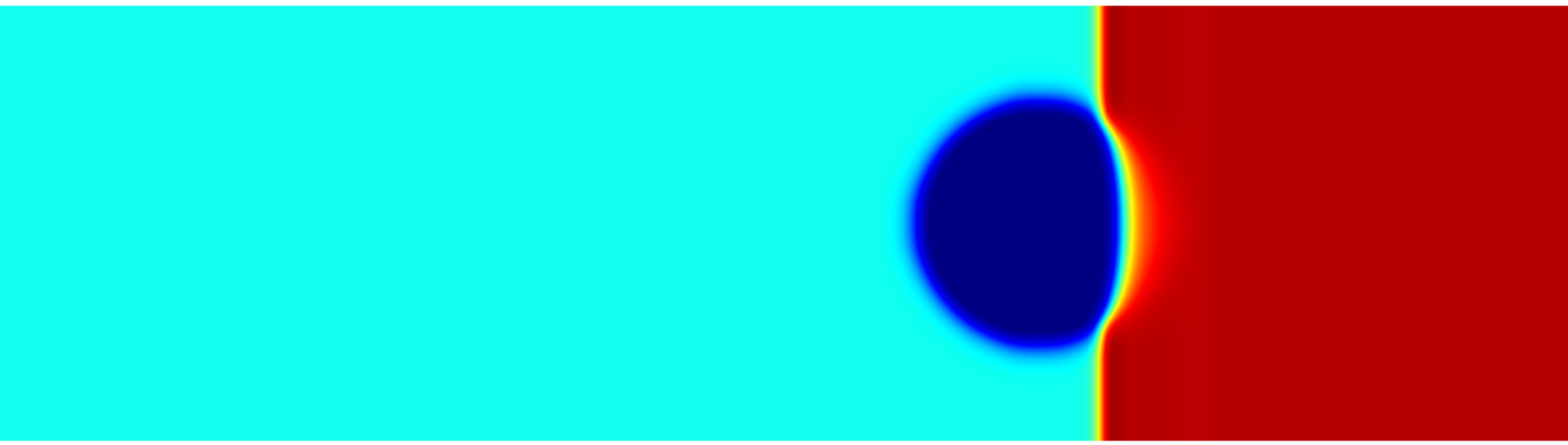}
\includegraphics[width=0.32\textwidth, trim=1 180 1 180, clip]{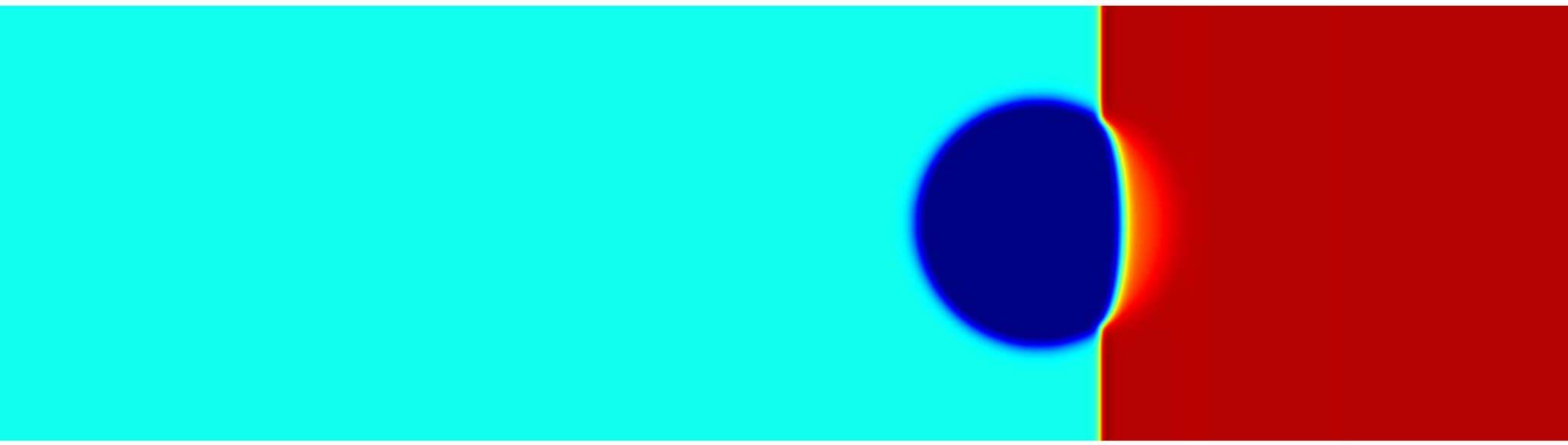}
\includegraphics[width=0.32\textwidth, trim=1 180 1 180, clip]{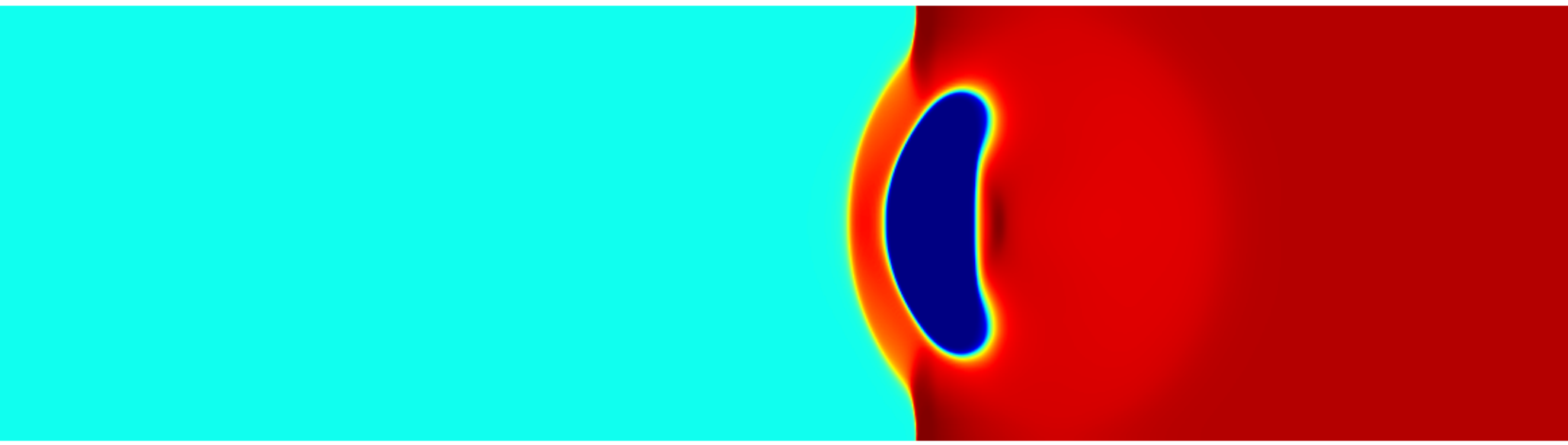}
\includegraphics[width=0.32\textwidth, trim=1 180 1 180, clip]{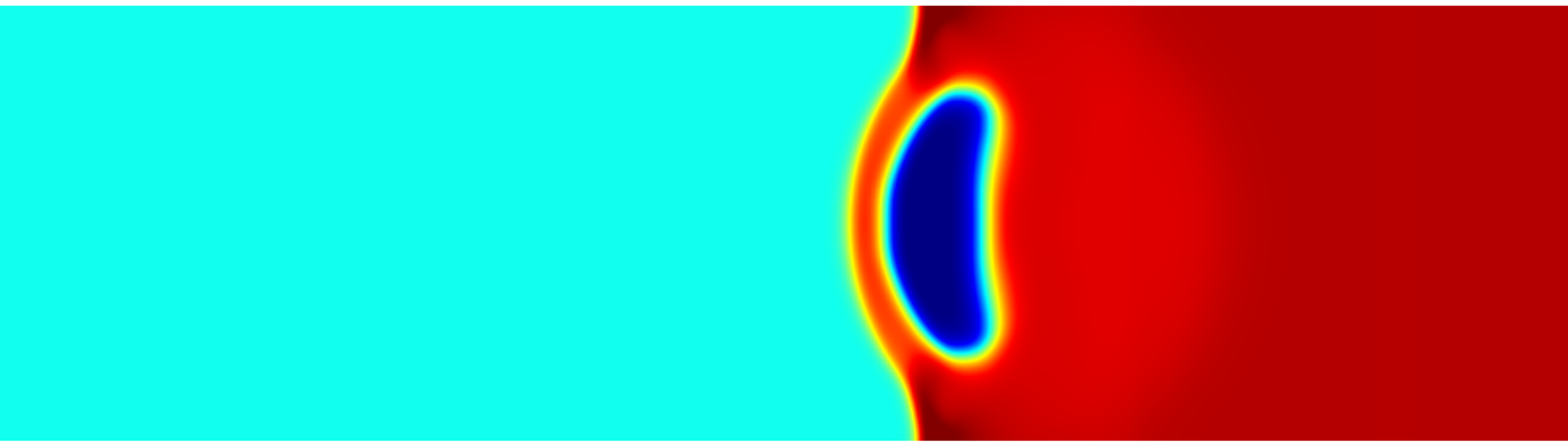}
\includegraphics[width=0.32\textwidth, trim=1 180 1 180, clip]{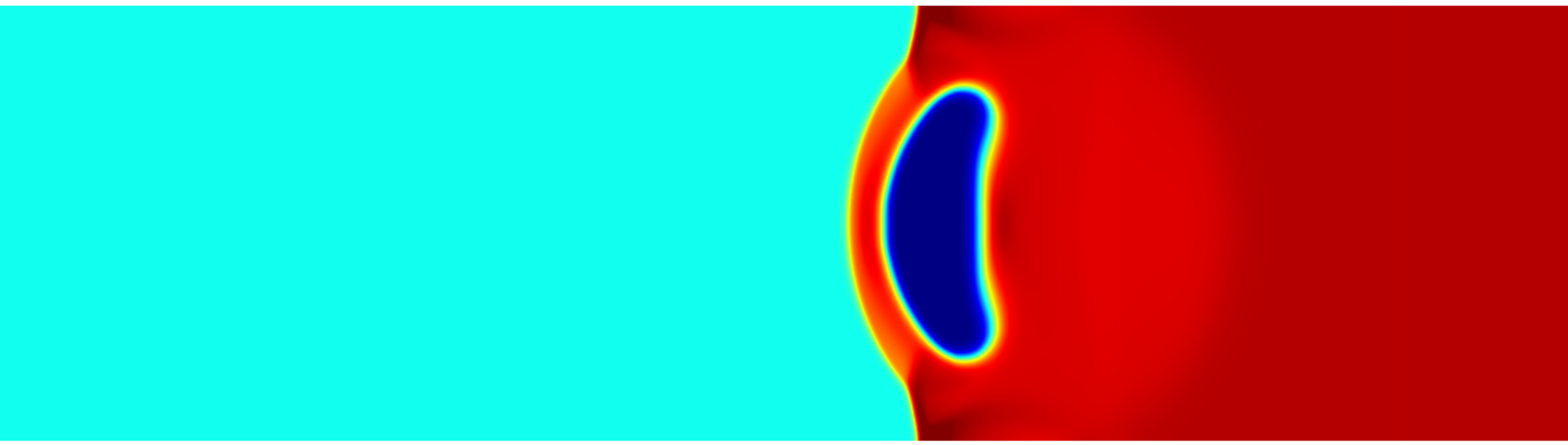}
\includegraphics[width=0.32\textwidth, trim=1 180 1 180, clip]{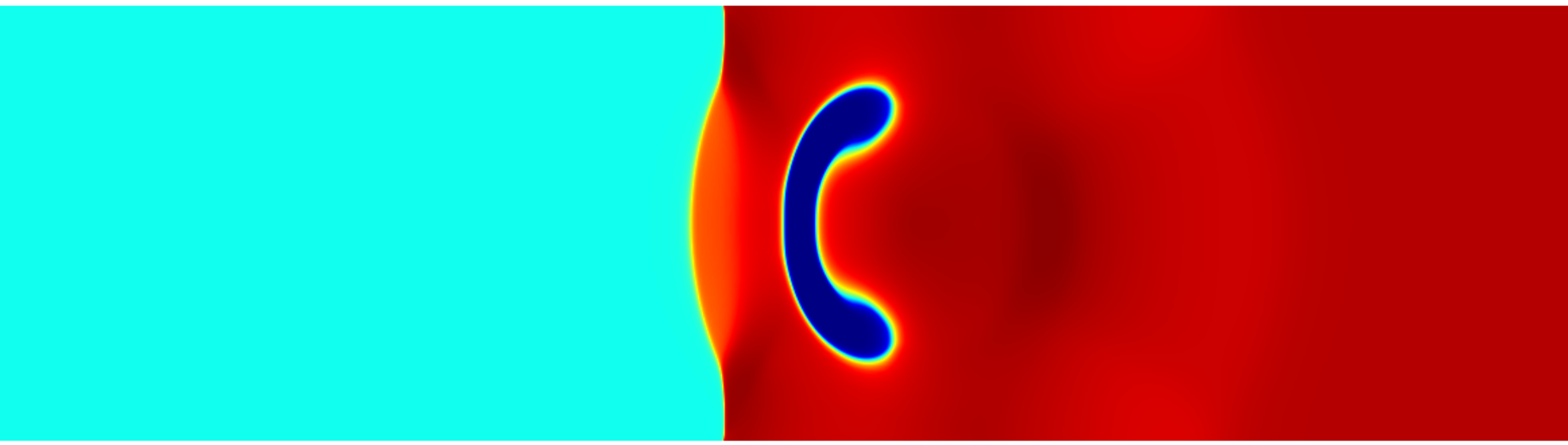}
\includegraphics[width=0.32\textwidth, trim=1 180 1 180, clip]{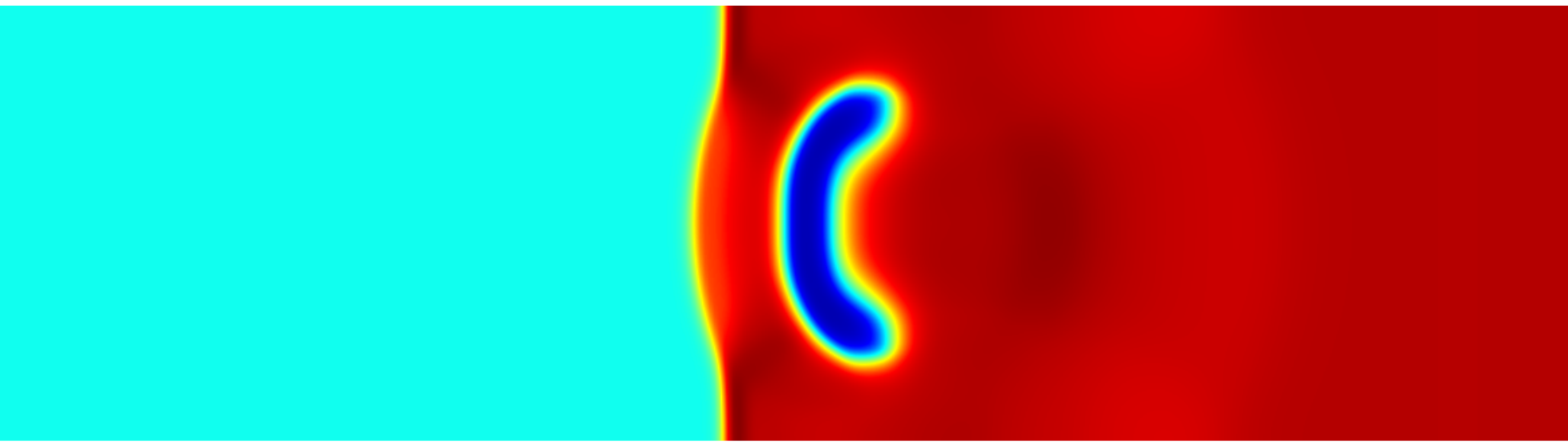}
\includegraphics[width=0.32\textwidth, trim=1 180 1 180, clip]{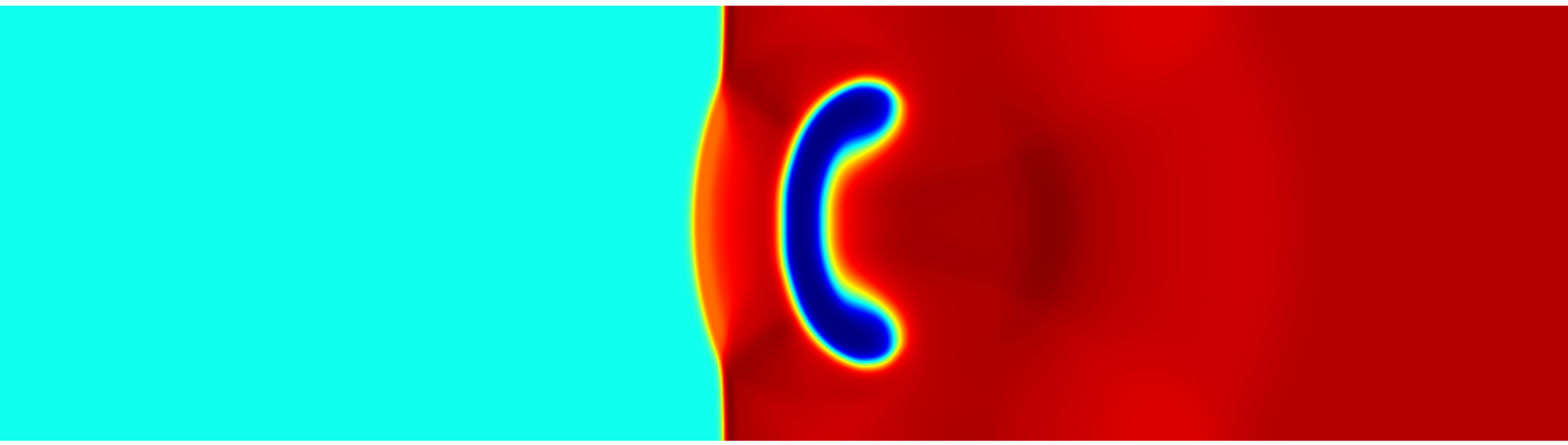}
\includegraphics[width=0.32\textwidth, trim=1 180 1 180, clip]{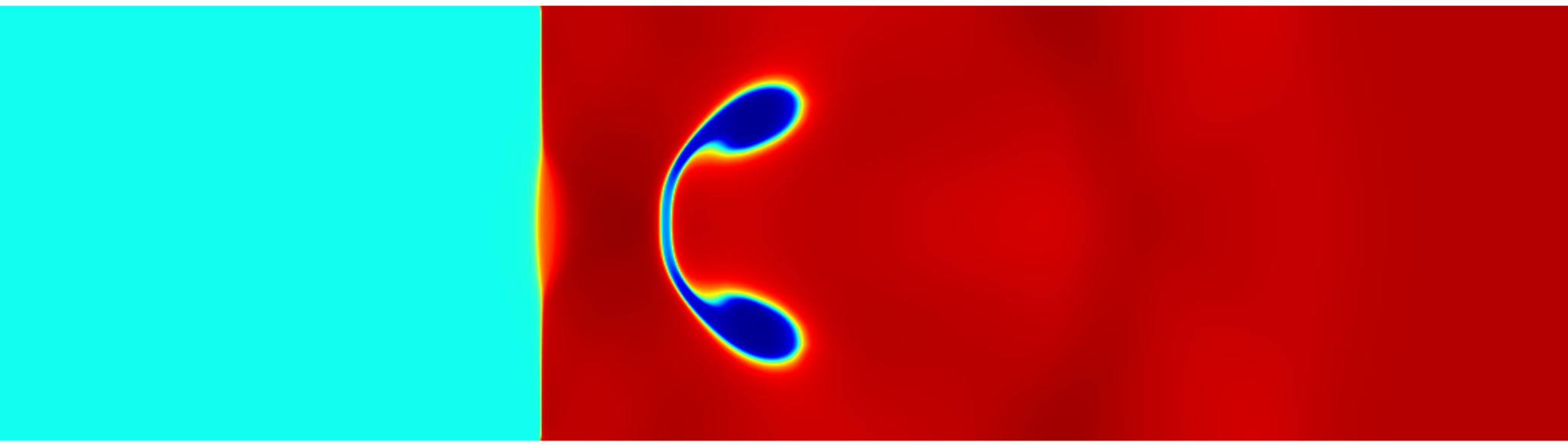}
\includegraphics[width=0.32\textwidth, trim=1 180 1 180, clip]{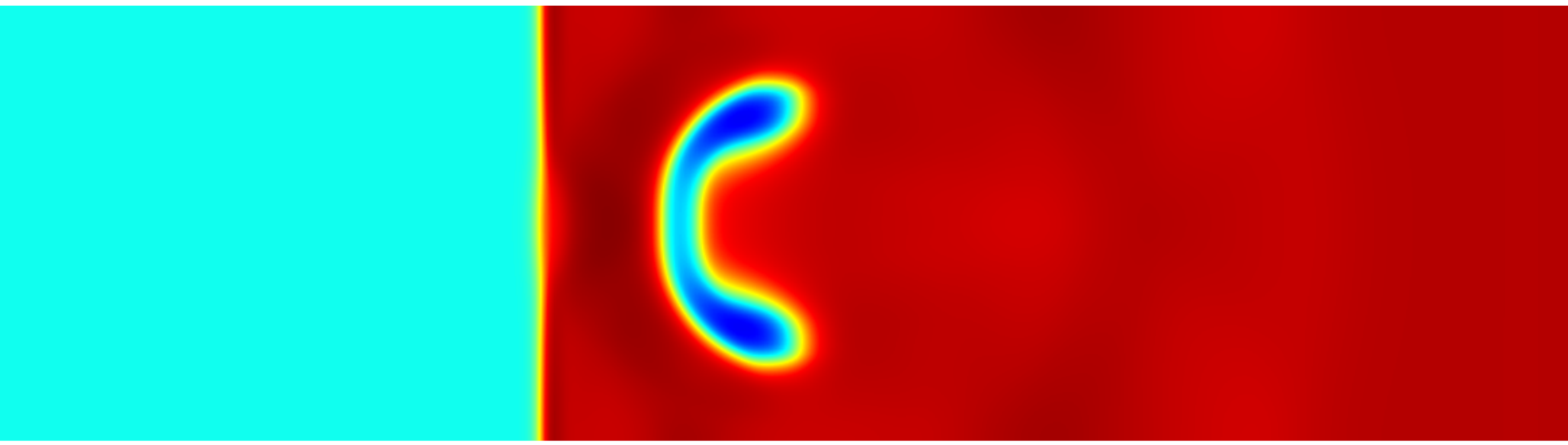}
\includegraphics[width=0.32\textwidth, trim=1 180 1 180, clip]{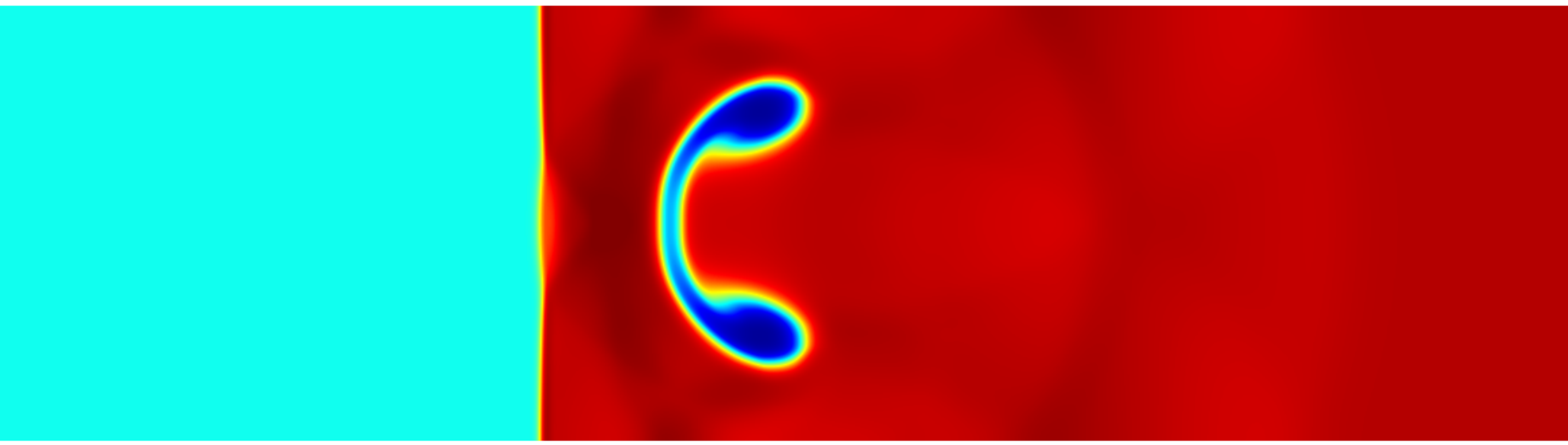}
\includegraphics[width=0.32\textwidth, trim=1 180 1 180, clip]{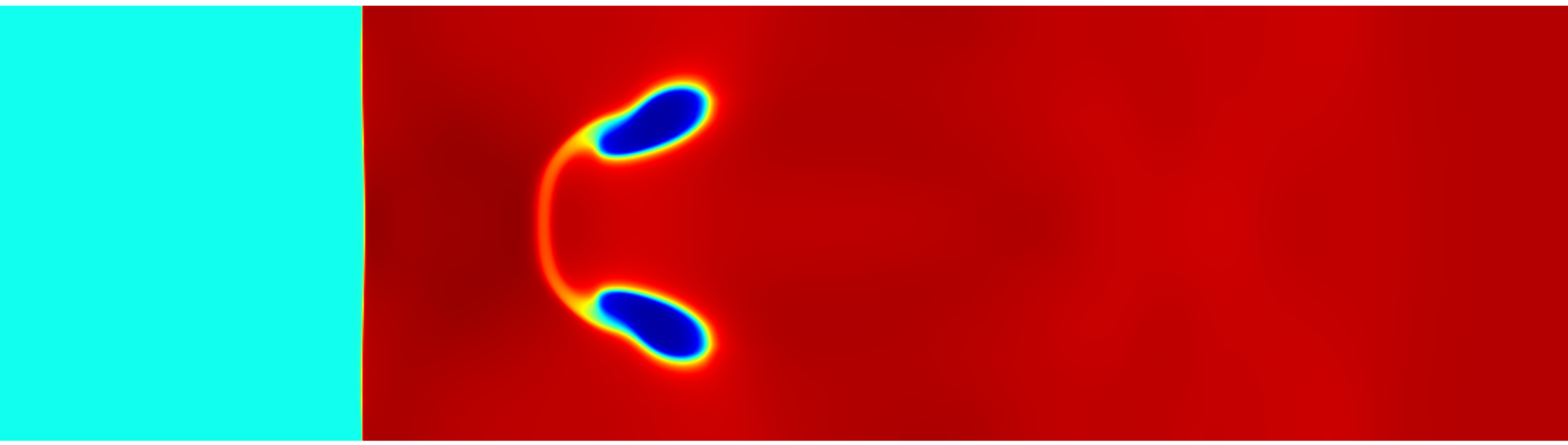}
\includegraphics[width=0.32\textwidth, trim=1 180 1 180, clip]{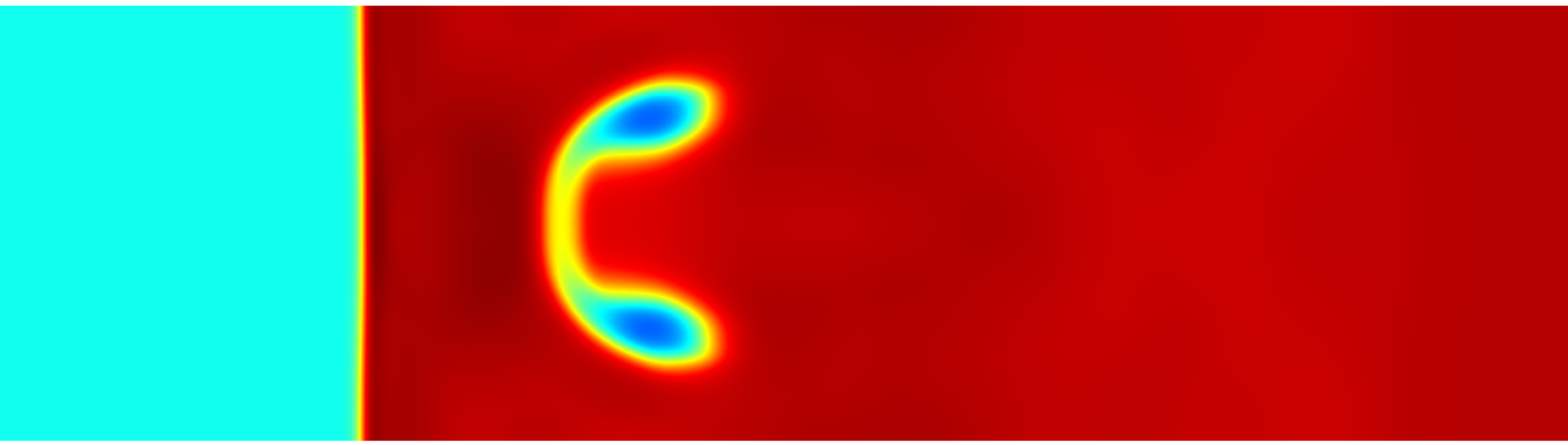}
\includegraphics[width=0.32\textwidth, trim=1 180 1 180, clip]{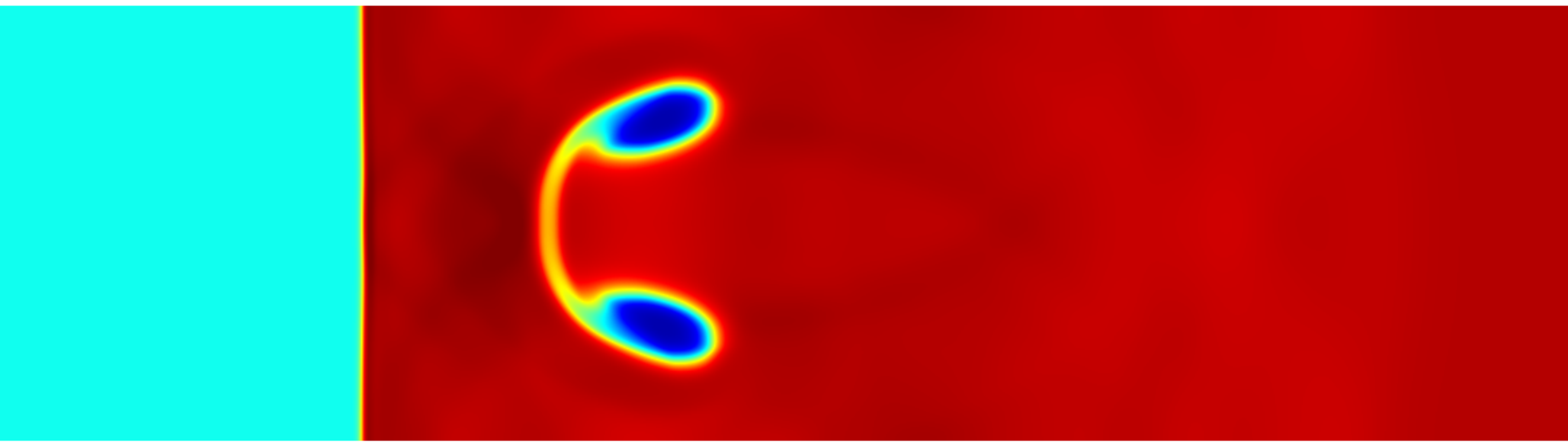}
  \caption{Example \ref{ex:SBL}:  $\rho$ on the slice $x_2=0$ at $t=90,180,270,360,450$ (from top to bottom).
  	Left: moving mesh of $325\times 90\times 90$, middle: uniform mesh of $325\times 90\times 90$, right: uniform mesh of $650\times 180\times 180$. }
  \label{fig:SBL_rho}
\end{figure}

%% file: Conc.tex
\section{Conclusion}\label{section:Conclusion}
This paper   presented the ES adaptive moving mesh schemes for the 2D and 3D special RHD equations,
which could be viewed as an extension of the second-order ES schemes in \cite{Duan2020RHD} to the adaptive moving mesh.
Our schemes were built on the ES finite volume approximation of the
RHD equations in curvilinear coordinates, the discrete geometric conservation laws, and the adaptive mesh redistribution.
Following the  procedure in \cite{Duan2020RMHD},
we constructed the  EC fluxes in curvilinear coordinates 
for the given entropy pair.
To do that,
a sufficient condition for the so-called two-point EC fluxes was first given.
Its proof mimicked the derivation of the continuous entropy identity in curvilinear coordinates and utilized the discrete GCLs achieved by the conservative metrics method \cite{Thomas1979}.
In order to avoid the numerical oscillation produced by the EC scheme
around the discontinuities, some suitable dissipation term utilizing linear reconstruction with the minmod limiter
in the scaled entropy variables was added to the EC flux to get the second-order accurate ES scheme satisfying
the semi-discrete entropy inequality.
The fully discrete schemes were derived by integrating
the above semi-discrete ES schemes  in time by
using the second-order accurate explicit strong-stability preserving Runge-Kutta schemes.
The resulting fully-discrete scheme was proved to preserve the free-stream states and two approximations of the volume conservation law  were given and compared. The first was easy to be implemented, while
the second could well approach to the value of the Jacobian $J$ calculated by its definition, i.e. the first equation of \eqref{eq:CMM_VCL}.
The mesh points were adaptively moved or redistributed by solving the Euler-Lagrange equation of the mesh adaption functional
on the computational mesh at each time step with suitably chosen monitor functions.
Several 2D and 3D numerical results showed that the ES adaptive moving mesh schemes effectively captured the localized structures,
such as sharp transitions or discontinuities, and were more efficient than their counterparts on uniform mesh.